\documentclass[10pt,reqno]{amsart}
\usepackage{amsthm}
\usepackage{esint}
\usepackage{mathrsfs}
\usepackage{amssymb}
\usepackage{hyperref}

\usepackage{color}
\usepackage{enumerate}

\makeatletter
\numberwithin{equation}{section}

\newtheorem{thm}{Theorem}[section]

\newtheorem{lem}[thm]{Lemma}
\newtheorem{prop}[thm]{Proposition}
\theoremstyle{definition}
\newtheorem{defn}[thm]{Definition}

\newtheorem{assumption}[thm]{Assumption}
\theoremstyle{remark}
\newtheorem{rem}[thm]{\bf{Remark}}

\newcommand\aint{-\hspace{-0.38cm}\int}

\newcommand\bC{\mathbb{C}}

\newcommand\bE{\mathbb{E}}

\newcommand\bH{\mathbb{H}}

\newcommand\bL{\mathbb{L}}
\newcommand\bM{\mathbb{M}}
\newcommand\bN{\mathbb{N}}

\newcommand\bP{\mathbb{P}}

\newcommand\bR{\mathbb{R}}

\newcommand\cF{\mathcal{F}}

\newcommand\cM{\mathcal{M}}

\newcommand\cP{\mathcal{P}}

\newcommand\cS{\mathcal{S}}
\newcommand\cT{\mathcal{T}}

\newcommand\Ccinf{C^{\infty}_{c}}

\newcommand{\mysection}[1]{\section{#1}}

\begin{document}

\title[SPDE with space-time non-local operators]
{A Sobolev space theory for the Stochastic Partial Differential Equations with space-time non-local operators}

\author{Kyeong-Hun Kim}
\address{Department of Mathematics, Korea University, 145 Anam-ro, Seongbuk-gu, Seoul,
02841, Republic of Korea}
\email{kyeonghun@korea.ac.kr}
\thanks{The authors were supported by the National Research Foundation of Korea(NRF) grant funded by the Korea government(MSIT) (No. NRF-2019R1A5A1028324)}

\author{Daehan Park}
\address{Stochastic Analysis and Application Research Center, Korea Advanced Institute of Science and Technology, 291 Daehak-ro, Yuseong-gu, Daejeon, 34141, Republic of Korea} \email{daehanpark@kaist.ac.kr}

\author{Junhee Ryu}
\address{Department of Mathematics, Korea University, 145 Anam-ro, Seongbuk-gu, Seoul,
02841, Republic of Korea} \email{junhryu@korea.ac.kr}



\maketitle

\begin{abstract}
We deal with the Sobolev space theory for the  stochastic partial differential equation (SPDE) driven by  Wiener processes
$$
\partial_{t}^{\alpha}u=\left( \phi(\Delta) u +f(u) \right) + \partial_t^\beta \sum_{k=1}^\infty \int_0^t g^k(u)\,dw_s^k, \quad t>0, \,x\in \bR^d
$$
as well as the SPDE driven by space-time white noise
$$
\partial^{\alpha}_{t}u=\phi(\Delta)u + f(u) + \partial^{\beta-1}_{t}h(u) \dot{W}, \quad t>0,\, x\in \bR^d.
$$ 
 Here, $\alpha\in (0,1), \beta< \alpha+1/2$, $\{w_t^k : k=1,2,\cdots\}$ is a family of  independent one-dimensional Wiener processes, and   $\dot{W}$ is a  space-time white noise defined on $[0,\infty)\times \bR^d$.
The time non-local operator $\partial_{t}^{\alpha}$ denotes the Caputo fractional derivative of order $\alpha$, the function $\phi$ is a Bernstein function, and the spatial non-local operator $\phi(\Delta)$ is the  integro-differential operator whose symbol is $-\phi(|\xi|^2)$. In other words, $\phi(\Delta)$ is the infinitesimal generator of the $d$-dimensional subordinate Brownian motion.

 We prove the uniqueness and  existence results  in Sobolev spaces, and obtain the maximal regularity results of solutions. 
\end{abstract}

\mysection{Introduction}

We study the stochastic partial differential equations with both time and spatial  non-local operators. The time and spatial non-local operators we adopt in this article are  $\partial^{\alpha}_t$ and $\phi(\Delta)$, respectively. The Caputo fractional derivative $\partial^{\alpha}_t$ is used in the time fractional heat equation 
 to describe e.g. the anomalous diffusion exhibiting subdiffusive behavior caused by particle sticking and trapping effects (e.g. \cite{metzler1999,metzler2000}), and the spatial non-local operator $\phi(\Delta)$ is the infinitesimal generator of the subordinate Brownian motion. The operator  $\phi(\Delta)$ describes long range jumps of particles, diffusion on fractal structures, and long time behavior of particles moving in space with quenched and disordered force field (e.g. \cite{bouchaud1990,fogedby1994}). For instance, if $\phi(\lambda)=\lambda^{\delta/2}$, then $\phi(\Delta)=\Delta^{\delta/2}$ becomes the fractional Laplacian, which is related to the isotropic $\delta$-stable process.  In this article we use both  $\partial^{\alpha}_t$ and $\phi(\Delta)$ for the description of the combined phenomena, for instance, jump diffusions with a higher peak and heavier tails (e.g. \cite{chen2012,compte1996,gorenflo1999,meerschaert2002}).

The goal of this article is to present an $L_p$-theory ($p\geq 2$) for the SPDE
driven by Wiener processes 
\begin{equation}\label{eqn 4.27.1}
\partial_{t}^{\alpha}u=\left( \phi(\Delta) u +f(u) \right) + \partial_t^\beta \sum_{k=1}^\infty \int_0^t g^k(u)\,dw_s^k, \quad t>0, x\in \bR^d; \,\,\, u(0,\cdot)=0
\end{equation}
as well as for the  SPDE driven by multi-dimensional space-time white noise
\begin{equation}
\label{eqn 4.27.2}
\partial^{\alpha}_{t}u=\phi(\Delta)u + f(u) + \partial^{\beta-1}_{t}h(u) \dot{W}, \quad t>0,x\in \bR^d; \quad u(0,\cdot)=0.
\end{equation}
As mentioned above,   $\{w_t^1, w^2_t, \cdots\}$ is a family of   independent one-dimensional Wiener processes,   $\dot{W}$ is a  space-time white noise on $[0,\infty)\times \bR^d$,  and $\alpha$ and $\beta$ are constants satisfying  $\alpha\in (0,1)$ and $\beta <\alpha+1/2$, respectively.  The non-linear terms  $f(u),g^k(u)$ and $h(u)$ are functions depending on $(\omega,t,x,u)$.
 Such types of SPDEs can be used e.g. to describe random effects of particles in medium with thermal memory or particles subject to sticking and trapping (see  e.g. \cite{chen2015fractional}).

 We interpret SPDEs \eqref{eqn 4.27.1} and \eqref{eqn 4.27.2} by their integral forms, and the restriction $\beta<\alpha+1/2$ is necessary to make sense of the equations. For instance, the integral form of  \eqref{eqn 4.27.1} is 
 \begin{eqnarray*}
 u(t,x)-u(0,x)&=&\frac{1}{\Gamma(\alpha)}\int^t_0 (t-s)^{\alpha-1}(\phi(\Delta)u(s,x)+f(s,x,u(s,x)))ds   \nonumber 
 \\
 &+&\sum_{k=1}^{\infty}
 \frac{1}{\Gamma(1+\alpha-\beta)}\int^t_0 (t-s)^{\alpha-\beta}g^k(s,x,u(s,x)) dw^k_s,
\end{eqnarray*}
 and  even if $g^k$ is bounded, say $g^k\equiv 1$, the condition $\alpha-\beta>-1/2$ is needed to make sense of the  integral $\int^t_0 (t-s)^{\alpha-\beta}dw^k_s$.

In this article, under appropriate continuity of  $f,g$,  we prove the unique solvability   of equation \eqref{eqn 4.27.1} together with the maximal regularity
\begin{eqnarray}
&&\mathbb{E}\| \phi(\Delta)^{(\gamma+2)/2}u \|^{p}_{L_{p}([0,T];L_{p})} \label{eqn 4.29.5} \\
&& \leq C\, \mathbb{E}\left( \| \phi(\Delta)^{\gamma/2}f(0)\|^{p}_{L_{p}([0,T];L_{p})} +\| |\phi(\Delta)^{(\gamma+c_0)/2} g(0)|_{l_2}\|^{p}_{L_{p}([0,T];L_{p})} \right), \nonumber
\end{eqnarray}
for any $p\geq 2$ and $\gamma\in \bR$. Here $c_{0}:=(2\beta-1)^+/{\alpha}<2$.  Also, we use estimate \eqref{eqn 4.29.5}  for $\gamma<0$  to deal with  equation \eqref{eqn 4.27.2}, that is, the SPDE driven by  space-time white noise. This is possible since one can transform equation \eqref{eqn 4.27.2} into the one of type    \eqref{eqn 4.27.1}.

Now let us provide a description on the related works and their approaches. 
The $L_p$-theory ($p\geq 2$) of   the classical stochastic heat equation of the type
\begin{equation*}
\label{sto heat}
du=\Delta u\, dt + g dw_t, \quad t>0, \,x\in \bR^d; \quad u(0,\cdot)=0
\end{equation*}
 was first introduced by N.V. Krylov \cite{kry99analytic,Krylov 1996}.   Krylov introduced  an analytic approach and proved the maximal regularity estimate
\begin{equation}
 \label{krylov lp}
\bE\|\nabla u\|^p_{L_p((0,T); L_p)}\leq C\bE \|g\|^p_{L_p((0,T); L_p)}, \qquad p\geq 2.
\end{equation}
The essence of Krylov's approach is to control  the sharp function of derivatives of $u$ in terms of the free term, that is 
\begin{equation}
\label{eqn 4.30.11}
(\nabla u)^{\sharp}(t,x)\leq C \left(\bM |g|^2(t,x)\right)^{1/2}, \quad \forall \, (t,x) \quad \text{uniformly on}\, \Omega,
\end{equation}
where $\sharp$ and $\bM$ are used to denote the sharp and maximal functions respectively (see Section \ref{Analytic}). This with the $L_p$-norm equivalent relation between functions and their sharp and maximal functions leads to \eqref{krylov lp}.  
Since the work of \cite{kry99analytic,Krylov 1996}, the analytic approach has been further used for SPDEs having different spatial operators.  The fractional Laplacian $\Delta^{\delta/2}$ is considered in  \cite{CL,Ildoo},  fractional Laplacian-like operator is considered in  \cite{MP}, and the operator $\phi(\Delta)$ is considered in \cite{kim2013parabolic}. It is also used for SPDE having time non-local operator in \cite{kim16timefractionalspde,desch2007stochastic} and \cite{desch2011p}, in which the spatial operators used are $\Delta$ and $\Delta^{\delta/2}$, respectively.  As  for other approaches on Sobolev regularity theory,  the method based  on $H^{\infty}$-calculus is also available in the literature. This approach  was introduced in \cite{Veraar2008,Veraar2012,Veraar},  in which  the maximal regularity of $\sqrt{-A}u$ is obtained  for the stochastic convolution 
$$
u(t):=\int^t_0 e^{(t-s)A} g(s) dW_s.
$$
 Here,  $W$ is a  Brownian motion, and  the operator $-A$ is assumed to admit a bounded $H^{\infty}$-calculus of angle less than $\pi/2$ on $L_p$, where $p\geq 2$.  The result of \cite{Veraar2008,Veraar2012,Veraar} certainly generalizes Krylov's result \cite{kry99analytic,Krylov 1996}  as one can take $A=\Delta$. The method based on $H^{\infty}$-calculus is also used in  \cite{desch2013maximal} for the study of  the mild solution to the integral equation
\begin{equation}\label{eqn 03.02.18:23}
u(t)+\int_{0}^{t}(t-s)^{\alpha-1}Au(s)ds = \int_{0}^{t}(t-s)^{\beta-1}g(s)dW_{s},
\end{equation}
where the generator $A$ is supposed to satisfy the assumption described above.  
We also remark that a  non-linear version of equation \eqref{eqn 03.02.18:23} is studied recently in \cite{roeckner} with $A(U)$ in place of $AU$ in the Hilbert space setting, that is, the Gelfand triple setting.  Also see \cite{chen2015fractional} for a Hilbert space theory of SPDEs having time non-local operator  and the second-order spatial  operators with measurable coefficients.

As is expected, our results for non-linear equations are proved based on those for the corresponding linear equations and certain fixed point argument.
To handle linear equations, we use both analytic approach and the one based on $H^{\infty}$-calculus.  First, speaking of Krylov's analytic approach, we control the sharp functions of solutions and their fractional derivatives in terms of free terms. In other words, we prove a generalization of  \eqref{eqn 4.30.11}. This approach is elementary, however the extension to general equations involves quite non-trivial computations. Moreover, for a technical reason, this approach is carried out under the condition 
 \begin{equation}\label{e}
c \left(\frac{R}{r}\right)^{\delta_0}\leq\frac{\phi(R)}{\phi(r)}, \qquad \forall\, 0<r<R<\infty,
\end{equation}
 where $\delta_0\in (0,1]$ and $c>0$.
Regarding the second approach based on $H^{\infty}$-calculus,  we  check that $\phi(\Delta)$ admits the bounded $H^{\infty}$-calculus on $L_p(\bR^d)$ of angle zero.  The second approach works without condition \eqref{e} but it relies on abstract operator theory.

This article is organized as follows. In Section \ref{section2}, we introduce basic facts on time and spatial non-local operators and   related function spaces. Then, we present our main results, Theorems \ref{thm:main results} and \ref{thm 10.27:15:35}. In Section \ref{2105041510} we obtain a priori estimate for the solutions, and finally in Sections \ref{sec proof} and \ref{section 6} we prove Theorems \ref{thm:main results} and \ref{thm 10.27:15:35}, respectively. 

We finish the introduction with notations used in this article. We use $``:="$ or $``=:"$ to denote a definition. As usual, $\mathbb{R}^{d}$ stands for the $d$-dimensional
Euclidean space of points $x=(x^{1},\ldots,x^{d})$.   We set $B_{r}(x):=\{y\in\mathbb{R}^{d}:|x-y|<r\}$,
and $B_{r}:=B_{r}(0)$.
$\mathbb{N}$ denotes the natural number system and $\bN_0:=\bN\cup\{0\}$. For $i=1,\ldots,d$,
multi-indices $\sigma=(\sigma_{1},\ldots,\sigma_{d})$,
$\sigma_{i}\in\{0,1,2,\ldots\}$, and functions $u(x)$, we set
\[
u_{x^{i}}=\frac{\partial u}{\partial x^{i}}=D_{i}u,\quad D_{x}^{\sigma}u=D_{1}^{\sigma_{1}}\cdots D_{d}^{\sigma_{d}}u,\quad|\sigma|=\sigma_{1}+\cdots+\sigma_{d}.
\]
We also use the notation $D_{x}^{m}$ for the set of  partial derivatives of
order $m$ with respect to $x$. For a Banach space $B$, 
by $C_{c}^{\infty}(\mathbb{R}^{d};B)$
we denote the collection of $B$-valued smooth functions having compact support
in $\mathbb{R}^{d}$. We drop $B$ if $B=\bR^d$.
$\cS(\bR^d)$ denotes the Schwartz class on $\bR^d$. By $C^{2}_{b}(\bR^{d})$, we denote the space of twice continuously differentiable functions on $\bR^d$ with bounded derivatives.
For $p> 1$, we use $L_{p}$ to denote the set
of complex-valued Lebesgue measurable functions $u$ on $\bR^{d}$ satisfying
\[
\left\Vert u\right\Vert _{L_{p}}:=\left(\int_{\bR^{d}}|u(x)|^{p}dx\right)^{1/p}<\infty.
\]
Generally, for a given measure space $(X,\mathcal{M},\mu)$, $L_{p}(X,\cM,\mu;B)$
denotes the space of all $B$-valued $\mathcal{M}^{\mu}$-measurable functions
$u$ so that
\[
\left\Vert u\right\Vert _{L_{p}(X,\cM,\mu;B)}:=\left(\int_{X}\left\Vert u(x)\right\Vert _{B}^{p}\mu(dx)\right)^{1/p}<\infty,
\]
where $\mathcal{M}^{\mu}$ denotes the completion of $\cM$ with respect to the measure $\mu$.
If there is no confusion for the given measure and $\sigma$-algebra, we usually omit the measure and the $\sigma$-algebra. We use the notations
\[
\mathcal{F}\left( f\right) (\xi):=\hat{f}(\xi):=\int_{\mathbb{R}^{d}}e^{-i\xi\cdot x}f(x)dx,\quad\mathcal{F}^{-1}( g) (x):=\frac{1}{(2\pi)^{d}}\int_{\mathbb{R}^{d}}e^{i\xi\cdot x}g(\xi)d\xi
\]
to denote the Fourier and the inverse Fourier transforms respectively. $a \wedge b := \min \{a,b\}$, $a \vee b := \max \{a,b\}$, and  $a^+:=a \vee 0$.
 Also we write $a\sim b$ if there exists a constant $c>1$ independent of $a,b$ such that $c^{-1}a\leq b \leq ca$.
If we write $C=C(a,b,\cdots)$, this means
that the constant $C$ depends only on $a,b,\cdots$. Throughout
the article, for functions depending on $(\omega,t,x)$, the argument
$\omega \in \Omega$ will be usually omitted.

\mysection{Main Results}
\label{section2}

First, we introduce some preliminary facts on the fractional calculus. For $\alpha>0$
and $\varphi\in L_{1}((0,T))$,  we define the Riemann-Liouville fractional integral
of the order $\alpha$ by
$$
I_{t}^{\alpha}\varphi:=\frac{1}{\Gamma(\alpha)}\int_{0}^{t}(t-s)^{\alpha-1}\varphi(s)ds, \quad 0\leq t\leq T.
$$
We also define $I^0\varphi:=\varphi$. Due to Jensen's inequality, for $p\in[1,\infty]$,
\begin{equation}
                      \label{eq:Lp continuity of I}
\left\Vert I_{t}^{\alpha}\varphi\right\Vert _{L_{p}(0,T)}\leq
C(\alpha, p,T)\left\Vert \varphi\right\Vert _{L_{p}(0,T)}.
\end{equation}
One can easily check for any $\alpha,\beta\geq 0$
\begin{equation} \label{eqn 4.15.3}
I^{\alpha}_tI^{\beta}_t \varphi=I^{\alpha+\beta}_t  \varphi.
\end{equation}

Let $\alpha \in [n-1, n)$ for some $n\in \bN$. For a function $\varphi(t)$ which $\left(\frac{d}{dt}\right)^{n-1} I_t^{n-\alpha}  \varphi$ is absolutely continuous on $[0,T]$,
the Riemann-Liouville fractional derivative $D_{t}^{\alpha}$ and the Caputo fractional derivative $\partial_{t}^{\alpha}$ of the order $\alpha$ are defined as
\begin{equation}
                          \label{eqn 4.15}
D_{t}^{\alpha}\varphi:=\left(\frac{d}{dt}\right)^{n}\left(I_{t}^{n-\alpha}\varphi\right),
\end{equation}
and
\begin{align*}
\partial_{t}^{\alpha}\varphi= D_{t}^{\alpha} \left(\varphi(t)-\sum_{k=0}^{n-1}\frac{t^{k}}{k!}\varphi^{(k)}(0)\right).
\end{align*}
In particular, if $\alpha\in (0,1)$, then
$$
D^{\alpha}_t\varphi=\left(I_{t}^{1-\alpha}\varphi\right)'(t), \qquad \partial^{\alpha}_t\varphi (t)=\left(I_{t}^{1-\alpha}(\varphi-\varphi(0))\right)'(t).
$$
Note that $D^{\alpha}_t\varphi=\partial^{\alpha}_t \varphi$ if $\varphi(0)=\varphi^{(1)}(0)=\cdots = \varphi^{(n-1)}(0)=0$. By \eqref{eqn 4.15.3} and \eqref{eqn 4.15}, for any $\alpha,\beta\geq 0$, we have
\begin{align*}
D^{\alpha}_t I_{t}^{\beta} \varphi=\begin{cases}
D_{t}^{\alpha-\beta}\varphi & :\alpha>\beta\\
I_{t}^{\beta-\alpha}\varphi& :\alpha\leq\beta,
\end{cases}
\end{align*}
and
$D^{\alpha}_tD^{\beta}_t=D^{\alpha+\beta}_t$. Also if 
$\varphi(0)=\varphi^{(1)}(0)=\cdots = \varphi^{(n-1)}(0)=0$ 
then
\begin{equation*}
I^{\alpha}_{t}\partial^{\alpha}_{t}u=I^{\alpha}_{t}D^{\alpha}_{t}u=u.
\end{equation*}
Finally we define $I^{-\alpha}_{t}\varphi:=D^{\alpha}_{t}\varphi$ for $\alpha>0$.

Next, we introduce the spatial non-local operator $\phi(\Delta)$, and function spaces related to this operator.
Recall that a fucntion $\phi: \bR_+ \to \bR_+$ 
satisfying  $\phi(0+)=0$
 is called a Bernstein function if there exist a constant $b\geq 0$, called a drift, and a  L\'evy measure $\mu$ (i.e.  $\int_{(0,\infty)} (1\wedge t) \mu(dt)<\infty$) such that
\begin{equation}
  \label{eqn 07.17.16:36}
\phi(\lambda)=b\lambda + \int_{(0,\infty)} (1-e^{-\lambda t})\mu(dt).
\end{equation}
It is known that $\phi(\lambda)$ is a Bernstein function if and only if there is a  subordinator (i.e. one-dimensional nondecreasing L\'evy process)  $S_{t}$ whose Laplace exponent is $\phi(\lambda)$, that is,
\begin{equation}
\label{eqn 4.10.1}
\bE e^{-\lambda S_t}=e^{-t\phi(\lambda)}, \quad \forall\, \lambda>0.
\end{equation}
From \eqref{eqn 07.17.16:36}, we easily have $\phi'(\lambda)>0$ and
$$
(-1)^n \phi^{(n)}(\lambda)\leq 0, \quad \forall \lambda>0, \, n\in \bN.
$$
Actually,   for any $n\geq 1$, we also have (see e.g. \cite{kim2020nonlocal,schilling2012bernstein})
\begin{equation}
  \label{eqn 07.19.14.35}
\lambda^{n}|\phi^{(n)}(\lambda)| \leq C(n) \phi(\lambda).
\end{equation}
Note also that  $\phi^{-1}$, the inverse function of $\phi$, is well defined  since $\phi(0+)=0$, $\phi$ is strictly increasing and  $\phi(+\infty)=\infty$.

For $f\in\cS(\bR^d)$,  we define $\phi(\Delta)f:=-\phi(-\Delta)f$ as 
\begin{equation*}
\phi(\Delta)f(x)=\cF^{-1}(-\phi(|\xi|^{2})\cF(f)(\xi))(x).
\end{equation*}
It turns out (e.g. \cite[Theorem 31.5]{sato1999levy}) that  $\phi(\Delta)$  is an integro-differential operator defined by
\begin{equation}\label{fourier200408}
\phi(\Delta)f(x)=b\Delta f+ \int_{\bR^d} \left(f(x+y)-f(x)-\nabla f(x)\cdot y \mathbf{1}_{|y|\leq 1}\right) J(y) dy
\end{equation}
where $J(x)=j(|x|)$ and $j:(0,\infty)\to(0,\infty)$ is given by
$$
j(r)=\int_{(0,\infty)} (4\pi t)^{-d/2} e^{-r^2/(4t)} \mu(dt).
$$
The non-local operator $\phi(\Delta)$ is related to a certain jump process as follows (see \cite{kim2014global,schilling2012bernstein}). Let $W_t$ be a $d$-dimensional Brownian motion independent of $S_t$, and denote $X_t:=W_{S_t}$  ($d$-dimensional subordinate Brownian motion). Then, it  holds that  $\phi(\Delta)$ is the infinitesimal generator of  $X_t$, that is, 
$$
\phi(\Delta)f(x)=\lim_{t \downarrow 0} \frac{\bE f(x+X_t) -f(x)}{t},
$$
where ${\bE}$ denotes the expectation.   Furthermore, the solution to the equation
$$
u_t=\phi(\Delta)u, \quad t>0; \quad u(0,\cdot)=u_0
$$
is given by $u(t,x)=\bE u_0(x+X_t)$.

\vspace{2mm}

Now we introduce Sobolev spaces related to the operator $\phi(\Delta)$.
For $p > 1$ and $\gamma\in\bR$, we denote $H_p^{\phi,\gamma}$ by the closure of $\cS(\bR^d)$ under the norm (cf. \cite{farkas2001function})
\begin{equation*}
\|u\|_{H_p^{\phi,\gamma}}:=\|(1-\phi(\Delta))^{\gamma/2}u\|_{L_p}.
\end{equation*}

Note that if $\phi(\lambda)=\lambda$, then $H^{\phi,\gamma}_{p}$ is the classical Bessel potential space $H^{\gamma}_{p}$.
For any $u\in H_p^{\phi,\gamma}$ and $\varphi \in \cS(\bR^d)$, by $(u,\varphi)$ we denote the value of linear functional $u$ at $\varphi$, that is,
\begin{equation*}
(u,\varphi):=\left((1-\phi(\Delta))^{\gamma/2}u,\, (1-\phi(\Delta))^{-\gamma/2}\varphi \right)_{L_2(\bR^d)}.
\end{equation*}
For any $\gamma, \nu\in \bR$ and $u\in  H_p^{\phi,\gamma}$,   we have $(1-\phi(\Delta))^{\nu/2}u\in H_p^{\phi,\gamma-\nu}$, and furthermore
\begin{equation}
\label{eqn 4.7.5.1}
((1-\phi(\Delta))^{\nu/2}u, \phi)=(u, (1-\phi(\Delta))^{\nu/2}\varphi), \quad \forall \, \varphi\in \cS(\bR^d).
\end{equation}
Let $l_2$ denote the set of all sequences $a=(a^1,a^2,\cdots)$ such that
$$|a|_{l_{2}}:=\left(\sum_{k=1}^{\infty}|a^{k}|^{2}\right)^{1/2}<\infty.
$$
By $H_{p}^{\phi,\gamma}(l_{2})=H_{p}^{\phi,\gamma}(\bR^d; l_2)$  we denote the class of all $l_2$-valued tempered distributions $v=(v^1,v^2,\cdots)$ on $\mathbb{R}^{d}$ such that
$$
\|v\|_{H_{p}^{\phi,\gamma}(l_{2})}:=\||(1-\phi(\Delta))^{\gamma/2}v|_{l_2}\|_{L_p}<\infty.
$$

The following lemma gives basic properties of $H^{\phi,\gamma}_{p}$ and $H_p^{\phi,\gamma}(l_2)$.

\begin{lem} \label{H_p^phi,gamma space}
(i) For any $p>1$, $\gamma\in\bR$, $H_p^{\phi,\gamma}$ is a Banach space.

(ii) For any $p>1$ and  $\mu,\gamma\in\bR$, the map $(1-\phi(\Delta))^{\mu/2}$ is an isometry from $H_p^{\phi,\gamma}$ to $H_p^{\phi,\gamma-\mu}$. 

(iii) If $p>1$ and $\gamma_1\leq\gamma_2$, then $H_p^{\phi,\gamma_2}\subset H_p^{\phi,\gamma_1}$, and there is a constant $C>0$ independent of $u$ such that
\begin{equation*}
\|u\|_{H_p^{\phi,\gamma_1}}\leq C \|u\|_{H_p^{\phi,\gamma_2}}.
\end{equation*}

(iv) If $p>1$ and $\gamma\geq0$, then  it holds that
\begin{equation*}
\|u\|_{H_p^{\phi,\gamma}} \sim  \left( \|u\|_{L_p}+\|\phi(\Delta)^{\gamma/2}u\|_{L_p} \right).
\end{equation*}

(v)  The assertions in (i)--(iv) also hold true for the $l_2$-valued function spaces $H_p^{\phi,\gamma}(l_2)$.
\end{lem}

\begin{proof}
First, $(i)$ and $(ii)$ easily follow from the definition of $H_p^{\phi,\gamma}$. For $(iii)$ and $(iv)$, see Theorems 2.3.1 and 2.2.7 in \cite{farkas2001function}, respectively. Here we remark that the proofs are based on Lemma 1.5.6, Theorems 1.5.10 and 2.2.10 in \cite{farkas2001function}, which can be proved for $l_2$-valued spaces. 
\end{proof}

\begin{rem} \label{Hvaluedconti}

(i) Following  \cite[Remark 3]{mikulevivcius2017p},  one can show that the embeddings $H_p^{2n} \subset H_p^{\phi,2n}$ and $H_p^{2n}(l_2) \subset H_p^{\phi,2n}(l_2)$ are continuous for any $n\in\bN$. Therefore, using this and the fact that $\Ccinf(\bR^{n})$ is dense in $H_p^{\gamma'}$ for any $\gamma'\in\bR$, we deduce that $\Ccinf(\bR^{d})$ is dense in $H^{\phi,\gamma}_{p}$ for all $\gamma$.

(ii) Let $\nu\in\bR$ and  $\varphi\in\cS(\bR^d)$.  Then for any multi-index $\sigma$, $D^\sigma \varphi\in H_p^{\phi,\nu}$ by (i). Therefore,   $(1-\phi(\Delta))^{\nu/2} D^{\sigma}\varphi  \in L_p$, and this implies $(1-\phi(\Delta))^{\nu/2}\varphi \in H_p^{2n}$ for any $n\in\bN$. 
\end{rem}

 Let $(\Omega,\mathscr{F},\bP)$ be a complete probability
space and $\{\mathscr{F}_{t},t\geq0\}$  an increasing filtration of
$\sigma$-fields $\mathscr{F}_{t}\subset\mathscr{F}$, each of which
contains all $(\mathscr{F},\bP)$-null sets. We assume that
a  family of independent one-dimensional
Wiener processes $\{w_{t}^{k}\}_{k\in\mathbb{N}}$ relative to the
filtration $\{\mathscr{F}_{t},t\geq0\}$ is given on $\Omega$.
By $\cP$, we denote the predictable $\sigma$-field generated by $\mathscr{F}_{t}$, i.e.
$\cP$ is the smallest $\sigma$-field containing every set $A \times (s,t]$, where $s<t$ and $A \in \mathscr{F}_s$.

Now we define stochastic Banach spaces for $p\geq 2$:
$$
\mathbb{H}_{p}^{\phi,\gamma}(T):=L_{p}\left(\Omega\times[0,T],\mathcal{P};H_{p}^{\phi,\gamma}\right),\quad\mathbb{L}_{p}(T)=\mathbb{H}_{p}^{\phi,0}(T),
$$
$$
\mathbb{H}_{p}^{\phi,\gamma}(T,l_{2}):=L_{p}\left(\Omega\times[0,T],\mathcal{P};H_{p}^{\phi,\gamma}(l_{2})\right),\quad\mathbb{L}_{p}(T,l_{2})=\mathbb{H}_{p}^{\phi,0}(T,l_{2}).
$$
We write $g=(g^1,g^2,\cdots)\in \mathbb{H}_{0}^{\infty}(T, l_2)$ if $g^k=0$ for all sufficiently large $k$, and each $g^k$ is of the type
$$
g^k(t,x)=\sum_{i=1}^{n(k)}1_{(\tau^k_{i-1},\tau^k_{i}]}(t)g^{ik}(x), \quad g^{ik}\in  C_c^\infty(\bR^d),
$$
where $0\leq \tau^k_0 \leq \tau^k_1 \leq \cdots \leq \tau^k_{n(k)}$ are bounded  stopping times. One can check that the space
 $\mathbb{H}_{0}^{\infty}(T,l_{2})$ is dense in $\mathbb{H}^{\phi,\gamma}_{p}(T,l_{2})$  (see e.g. \cite[Theorem 3.10]{kry99analytic}).

\begin{rem}\label{rmk 09.01.15:42}

Let $p\geq 2$, and $q$ denote the conjugate of $p$. Then, 
by Minkowski and H\"older inequalities, for any $g\in \bH^{\phi,\gamma}_p(T,l_2)$ and $\varphi\in H^{\phi,-\gamma}_q$, 
\begin{eqnarray}                
&&\bE \left[\sum_{k=1}^{\infty} \int^T_0 (g^k(s,\cdot),\varphi)^2 ds \right]  \nonumber\\
&&=
\bE \left[\sum_{k=1}^{\infty} \int^T_0 \Big( (1-\phi(\Delta)^{\gamma/2})g^k(s,\cdot),(1-\phi(\Delta)^{-\gamma/2})\varphi \Big)^2_{L_2} ds \right] \nonumber\\
&&\leq \bE\int^T_0 \||(1-\phi(\Delta))^{\gamma/2}g|_{l_2}\|^2_{L_p} \|(1-\phi(\Delta))^{-\gamma/2}\varphi\|^2_{L_q}\,dt \nonumber\\
&&\leq C(T) \|\varphi\|^2_{H^{-\gamma,\phi}_q} \|g\|^2_{\bH^{\phi,\gamma}_p(T,l_2)}. \label{eqn 4.7.1}
\end{eqnarray}
Consequently,
$(g, \varphi)\in L_2(\Omega\times [0,T], \cP; l_2)$, and it also  follows  that the sequence of stochastic integral $\sum_{i=1}^n \int^t_0 (g^k(s,\cdot), \varphi) dw^k_s$ converges in probability uniformly on $[0,T]$, and consequently  the infinite series
$\sum_{i=1}^{\infty}\int^t_0 (g^k(s,\cdot), \varphi)dw^k_s$ becomes a continuous $L_2$-martingale on $ [0,T]$.  
\end{rem}

\vspace{2mm}

The following lemma will be used later for  certain  approximation arguments. 
\begin{lem}
\label{fracint of dw}
\noindent
(i) Let $\nu \geq 0$ and  $h\in L_{2}(\Omega\times[0,T],\mathcal{P};l_{2})$.
Then the equality
\begin{equation*}
I^{\nu}_t \left(\sum_{k=1}^{\infty}\int_{0}^{\cdot}h^{k}(s)dw_{s}^{k}\right)(t)=\sum_{k=1}^{\infty}\left(I^{\nu}_t \int_{0}^{\cdot}h^{k}(s)dw_{s}^{k}\right)(t)\label{eq:equality s int}
\end{equation*}
holds for all $t\leq T$ $(a.s.)$ and also in $L_{2}(\Omega\times[0,T])$.

\noindent
(ii) Suppose $\nu \geq 0$ and $h_{n} \to h$ in $L_{2}(\Omega\times[0,T],\mathcal{P} ; l_{2})$
as $n\rightarrow\infty$. Then
$$
\sum_{k=1}^{\infty}\left(I^{\nu}_t \int_{0}^{\cdot}h_{n}^{k}dw_{s}^{k}\right)(t) \longrightarrow\sum_{k=1}^{\infty}\left(I^{\nu }_t \int_{0}^{\cdot}h^{k}dw_{s}^{k}\right)(t)
$$
in probability uniformly on $[0,T]$.

\noindent
(iii)  If $\nu<1/2$ and $h\in L_{2}(\Omega\times[0,T],\mathcal{P} ; l_{2})$, then
\begin{align*}
\partial^{\nu}_t \left(\sum_{k=1}^{\infty}\int_{0}^{\cdot}h^{k}(s)dw_{s}^{k}\right)(t)
&=\frac{1}{\Gamma(1-\nu)}\sum_{k=1}^{\infty}\int_{0}^{t}(t-s)^{-\nu}h^{k}(s)dw_{s}^{k}
\end{align*}
$(a.e.)$ on $\Omega \times [0,T]$.

\noindent
(iv)  Let $0< \nu<1/2$ and $h_n \to h$ in $L_{2}(\Omega\times[0,T],\mathcal{P} ; l_{2})$ as $n \to \infty$.  Then there exists a subsequence $n_j$ such that 
\begin{align*}
\partial^{\nu}_t \left(\sum_{k=1}^{\infty}\int_{0}^{\cdot}h^{k}_{n_j}(s)dw_{s}^{k}\right)(t)   \longrightarrow
 \partial^{\nu}_t \left(\sum_{k=1}^{\infty}\int_{0}^{\cdot}h^{k}(s)dw_{s}^{k}\right)(t) 
\end{align*}
$(a.e.)$ on $\Omega \times [0,T]$.\end{lem}

\begin{proof}
See \cite[Lemma 3.1, Lemma 3.3]{chen2015fractional} for (i)--(iii). We prove (iv).   Put $g_n=h_n-h$. Then, for each $t>0$, by Burkholder-Davis-Gundy inequality,
\begin{eqnarray*}
\bE\left| \sum_{k=1}^{\infty}\int^t_0 (t-s)^{-\nu}g^k_n(s)dw^k_s\right|^2 &\leq &
\bE \sup_{r\leq t} \left| \sum_{k=1}^{\infty}\int^r_0 (t-s)^{-\nu}g^k_n(s)dw^k_s\right|^2\\
&\leq&c \bE\int^t_0 |t-s|^{-2\nu}|g_n(s)|^2_{l_2} ds
\end{eqnarray*}
Therefore, we have the $L_2(\Omega\times [0,T])$ convergence, i.e. 
$$
\bE\int^T_0 \left| \sum_{k=1}^{\infty}\int^t_0 (t-s)^{-\nu}g^k_n(s)dw^k_s\right|^2\,dt \leq c\bE \int_{0}^{T} \int^t_0 |t-s|^{-2\nu}|g_n(s)|^2_{l_2} ds dt \to 0
$$
as $n\to \infty$, and the claim easily follows. 
\end{proof}

\vspace{3mm}

Now we explain our sense of solutions.

\begin{defn}
\label{defn 1} Let $u\in\mathbb{H}^{\phi,\gamma_1}_{p}(T)$, $f\in\mathbb{H}_{p}^{\phi,\gamma_2}(T)$, and $g\in\mathbb{H}_{p}^{\phi,\gamma_3}(T,l_{2})$ for some $\gamma_i\in \bR$, $i=1,2,3$.  Then we  say $u$ satisfies 
\begin{equation} \label{eqn 7.15}
\partial_{t}^{\alpha}u(t,x)=f(t,x)+\partial_{t}^{\beta}\sum_{k=1}^\infty\int_{0}^{t}g^{k}(s,x)dw_{s}^{k}, \quad t\in (0,T]; \quad u(0,\cdot)=0
\end{equation}
 in the sense of distributions if  for any $\varphi\in \cS(\bR^d)$ the equality
\begin{equation} \label{eq:solution space 2}
\left(u(t),\varphi\right)  =I_{t}^{\alpha}\left(f(t),\varphi\right)+\partial_{t}^{\beta-\alpha}\sum_{k=1}^{\infty}\int_{0}^{t}\left(g^{k}(s),\varphi\right)dw_{s}^{k}
 \end{equation}
holds  ``a.e. on $\Omega\times[0,T]$''. Here $\partial^{\beta-\alpha}_t:=I^{\alpha-\beta}_t$ if $\beta\leq\alpha$. 
\end{defn}

\begin{rem} \label{test ftn}

(i)  Due to Remark \ref{rmk 09.01.15:42} and   Lemma \ref{fracint of dw},  the infinite series  in \eqref{eq:solution space 2} makes sense.

(ii) Let \eqref{eqn 7.15} hold with $u,f,g$ as in Definition \ref{defn 1}. Denote 
$\gamma=\min \{\gamma_1,\gamma_2,\gamma_3 \}$.  Then,  \eqref{eqn 4.7.1}, Lemma \ref{fracint of dw} and standard approximation argument show that  \eqref{eq:solution space 2} holds a.e. on $\Omega\times[0,T]$ for any $\varphi\in H^{\phi,-\gamma}_q$, where $q=p/{(p-1)}$.  In particular, it holds a.e. for any $(1-\phi(\Delta))^{\nu/2} \bar{\varphi}$, where $\nu \in \bR$ and $\bar{\varphi}\in \cS(\bR^d)$.

\end{rem}

\begin{rem}
\label{remark 4.7.3} 
 Let $u,f$ and $g$ be given as in Definition   \ref{defn 1}.  Fix $\nu\in \bR$, and denote $\bar{\gamma}_i=\gamma_i-\nu$, $i=1,2,3$. Also denote 
 $\bar{u}=(1-\phi(\Delta))^{\nu/2}u$, and define $\bar{f}$  and $\bar{g}$ in the same way.  Then, by  Lemma \ref{H_p^phi,gamma space},
 $$
 \bar{u}\in\mathbb{H}^{\phi,\bar{\gamma}_1}_{p}(T), \quad \bar{f} \in\mathbb{H}_{p}^{\phi,\bar{\gamma}_2}(T), \quad
 \bar{g} \in\mathbb{H}_{p}^{\phi,\bar{\gamma}_3}(T,l_{2}).
 $$
 Moreover,  due to Remark \ref{test ftn} (ii) and  \eqref{eqn 4.7.5.1}, it follows that  \eqref{eqn 7.15} holds in the sense of distributions with $(\bar{u}, \bar{f}, \bar{g})$,  in place of $(u,f,g)$. 
 
 \end{rem}

In Definition \ref{defn 1},  we only require \eqref{eq:solution space 2} holds a.e. on $\Omega\times [0,T]$, not for all $t\leq T$ (a.s.). Below we give an equivalent statement which will clarify our notion of  solutions.

\begin{prop} \label{equiv sol sense}
Let $u\in \mathbb{H}_p^{\phi,\gamma}$, $f\in\mathbb{H}_{p}^{\phi,\gamma}(T)$, and
 $g\in\mathbb{H}_{p}^{\phi,\gamma}(T,l_{2})$ for some $\gamma\in \bR$.
 Then the following are equivalent.

 (i) $u$ satisfies  \eqref{eqn 7.15} in the sense of Definition \ref{defn 1} with $f$, and $g$.

 (ii) For any  constant $\Lambda$ satisfying
 $$
\Lambda \geq (\alpha \vee \beta) \text{ and } \Lambda > \frac{1}{p},
 $$
 $I_t^{\Lambda-\alpha}u$ has an $H_p^{\phi,\gamma}$-valued continuous version in $\mathbb{H}_p^{\phi,\gamma}(T)$, still denoted by $I_t^{\Lambda-\alpha}(u)$, such that for any $\varphi\in\cS(\bR^d)$,
 \begin{equation}
            \label{eqn 7.15.2}
(I^{\Lambda-\alpha}_t u(t),\varphi) =I_{t}^{\Lambda}\left(f(t,\cdot),\varphi\right)+\sum_{k=1}^{\infty}I_{t}^{\Lambda-\beta}\int_{0}^{t}\left(g^{k}(s,\cdot),\varphi\right)dw_{s}^{k}
 \end{equation}
 holds  ``for all $t\in[0,T]$ (a.s.)''.

(iii) The claim of (ii) holds for some $\Lambda$ satisfying $\Lambda \geq (\alpha \vee \beta) \text{ and } \Lambda > \frac{1}{p}$.
\end{prop}

\begin{proof}
Due to Remark \ref{remark 4.7.3}, it suffices to prove only the case $\gamma=0$. In this case we have $H^{\phi,0}_p=L_p$, and thus the lemma follows from \cite[Proposition 2.13]{kim2020sobolev}.  Actually  \cite[Proposition 2.13]{kim2020sobolev} does not include  statement (iii) above. However, the proof ``(ii) $\to$(i)'' only utilizes the result of (iii).  
\end{proof}

\begin{rem}
If $\alpha=\beta=1$, then we cake take $\Lambda=1$. Then, \eqref{eqn 7.15.2} reads as
$$
(u(t),\varphi) =\int^t_0 (f(s,\cdot),\varphi)ds+\sum_{k=1}^{\infty}\int_{0}^{t}\left(g^{k}(s,\cdot),\varphi\right)dw_{s}^{k}, \quad \forall \, t\leq T\, (a.s.).
$$
  This might look wrong  at first glance because in Definition \ref{defn 1} we only require this equality holds a.e. on $\Omega\times [0,T]$. The point of 
  Proposition \ref{equiv sol sense} is that $u$ has a continuous version, still denoted by $u$, such that this equality holds for all $t\leq T$ (a.s.). 
\end{rem}

\begin{prop}
\label{prof 4.7}
Let assumptions in Proposition \ref{equiv sol sense} hold, and let $u$ satisfy \eqref{eqn 7.15} in the sense of Definition \ref{defn 1} with $f$ and $g$. 

(i)  If  $\Lambda \geq (\alpha \vee \beta) \text{ and } \Lambda > \frac{1}{p}$,   then
\begin{equation*}
\mathbb{E}\sup_{t\leq T}\|I^{\Lambda-\alpha}_{t}(u(t))\|^{p}_{H^{\phi,\gamma}_{p}}
\leq C \left( \|f\|^{p}_{\mathbb{H}^{\phi,\gamma}_{p}(T)}+\|g\|^{p}_{\mathbb{H}^{\phi,\gamma}_{p}(T,l_{2})} \right),
\end{equation*}
where $C=C(\Lambda,\alpha,\beta,T,p)$.

(ii) Let $\theta:=\min\{1,\alpha,2(\alpha-\beta)+1\}$. Then, for any $t\leq T$,
\begin{equation} \label{eq:solution space estimate 1}
\left\Vert u\right\Vert _{\mathbb{H}_{p}^{\phi,\gamma}(t)}^{p}
\leq C\int_{0}^{t}(t-s)^{\theta-1} \left(\|f\|^p_{\mathbb{H}_{p}^{\phi,\gamma}(s)}
+\|g\|^p_{\mathbb{H}_{p}^{\phi,\gamma}(s,l_{2})} \right)ds,
\end{equation}
where  the constant $C$  depends only $\alpha,\beta,d,p,\gamma,\phi,T$.
\end{prop}

\begin{proof}
Again, we only need to consider the case $\gamma=0$. Therefore, (i) and (ii) follow from \cite[Proposition 2.13]{kim2020sobolev} and  \cite[Theorem 2.1 (iv)]{kim16timefractionalspde}, respectively.
\end{proof}

 For $\alpha\in(0,1)$, $\beta< \alpha+\frac{1}{2}$, and $\kappa\in (0,1)$, define
\begin{equation*}
c_0 :=c_{0}(\alpha,\beta,\kappa)=\frac{(2\beta-1)^{+}}{\alpha}+\kappa 1_{\beta=1/2}.
\end{equation*}

Below we use  notation $f(u)$, and $g(u)$ to denote $f(\omega, t,x,u)$, and $g(\omega, t,x,u)$ respectively.

\begin{assumption}\label{asm 11.15.13:16}

(i) For any $u\in\mathbb{H}^{\phi,\gamma+2}_{p}(T)$,
$$
f(u)\in\mathbb{H}^{\phi,\gamma}_{p}(T), \quad g(u)\in\mathbb{H}^{\phi,\gamma+c_{0}}_{p}(T,l_{2}).
$$

(ii) For any $\varepsilon>0$, there exists a constant $N=N(\varepsilon)>0$ so that
\begin{equation*}
\begin{aligned}
\|f(t,u)-f(t,v)\|_{H^{\phi,\gamma}_{p}}&+\|g(t,u)-g(t,v)\|_{H^{\phi,\gamma+c_{0}}_{p}(l_{2})}
\\
&   \qquad \leq  \varepsilon \|u-v\|_{H^{\phi,\gamma+2}_{p}}+N \|u-v\|_{H^{\phi,\gamma}_{p}}
\end{aligned}
\end{equation*}
 for any $\omega,t$, and $u,v\in H^{\phi,\gamma+2}_{p}$ .
\end{assumption}

Here is our main result for SPDE driven by a family of independent  one-dimensional Wiener processes. The proof of Theorem \ref{thm:main results} is  given in Section \ref{sec proof}.

\begin{thm}
                    \label{thm:main results}

Let $\alpha\in(0,1)$, $\beta<\alpha+1/2$, $p\geq2$, $\gamma\in \bR$, and $T\in(0,\infty)$. Let Assumption  \ref{asm 11.15.13:16} hold. Then, the equation
\begin{equation} \label{main equation}
\partial_{t}^{\alpha}u=\phi(\Delta)u+f(u)+\partial_{t}^{\beta}\sum_{k=1}^\infty\int_{0}^{t}g^{k}(u)dw_{s}^{k}, \quad t\in (0,T],
\end{equation}
with $u(0,\cdot)=0$ has a unique solution $u\in\bH_{p}^{\phi,\gamma+2}(T)$ in the sense of distribution, and for this solution we have
\begin{equation}   \label{eq: a priori estimate non-div}
\|u\|_{\bH_{p}^{\phi,\gamma+2}(T)}\leq
 C\left(\|f(0)\|_{\mathbb{H}_{p}^{\phi,\gamma}(T)}+\|g(0)\|_{\mathbb{H}_{p}^{\phi,\gamma+c_{0}}(T,l_{2})}\right),
\end{equation}
where the constant $C$ depends only on $\alpha,\beta,d,p,\phi,\gamma,\kappa$, and $T$. 
\end{thm}

\begin{rem}
Recall $c_0=\frac{(2\beta-1)^{+}}{\alpha}+\kappa 1_{\beta=1/2}$.  Note that the  pair $(\gamma+2, \gamma+c_0)$  determines the regularity relation between the solution $u$ and forcing term $g(0)$. Here are some comments and details on $c_0$:

\begin{itemize}

\item{} if $\beta>1/2$, then $c_0=\frac{(2\beta-1)}{\alpha}$. Hence, to have $H^{\phi, \gamma+2}_p$-valued solution $u$, we require  $g(0)$ to be  an $H^{\phi, \gamma+\frac{(2\beta-1)}{\alpha}}_p(l_2)$-valued process.  This relation is optimal and can be easily proved using a scaling argument (see e.g. \cite[Remark 2.20]{kim2020sobolev}). 

\item{}  if $\beta<1/2$, then   $c_0=0$. Thus, the stochastic forcing term $g(0)$ is not assumed to be smoother than the deterministic forcing term $f(0)$. Actually, if $\beta<1/2$, we can transform equation \eqref{main equation} into a PDE by absorbing the stochastic term  $\partial_{t}^{\beta}\sum_{k=1}^\infty\int_{0}^{t}g^{k}(u)dw_{s}^{k}$ into $f(u)$.

\item{} if $\beta=1/2$, we assume $c_0>0$.  This is a technical assumption: we handle the case $\beta=1/2$ based on the result for $\beta>1/2$, and this approach yields extra regularity on $g(0)$.

\end{itemize}
\end{rem}

\begin{rem}
Regarding the equations with nonzero initial values, the similar argument in \cite[Section 4.4]{Veraar measurable time} yields that for the solution to
$$
\partial_t^\alpha u=\phi(\Delta)u, \quad t>0, x\in \bR^d; \quad u(0,\cdot)=u_0,
$$
we have
$$
\|u\|_{\bH_{p}^{\phi,\gamma+2}(T)}\leq
 C\|u_0\|_{L_p(\Omega,\cF_0;{B}_{p}^{\phi,\gamma+2-2/\alpha p})},
$$
where ${B}_{p}^{\phi,\gamma+2-2/\alpha p}$ denotes the Besov spaces related to $\phi$ (see e.g. \cite[Definition 2.3]{kim2020nonlocal}). However, for the simplicity, we always assume $u(0)=0$. 
\end{rem}

\begin{rem}

By letting $\alpha\to 1$ and taking $\beta=1$ and $\phi(\lambda)=\lambda$, we (at least formally) get a classical result by Krylov \cite[Theorem 5.2]{kry99analytic}.
\end{rem}

\vspace{3mm}

Next, we consider the semi-linear SPDE driven by space-time white noise:
\begin{equation}\label{eqn 4.8.8}
\partial^{\alpha}_{t}u=\phi(\Delta)u+f(u)+\partial^{\beta-1}_{t} h(u) \dot{W}, \quad t>0\, ; \quad u(0,\cdot)=0.
\end{equation}
Here $\dot{W}$ is a space-time white noise on $[0,\infty)\times \bR^d$, and  the functions $f$ and $h$ depend on $(\omega,t,x,u)$.

First, to explain our sense of solutions,  let us multiply by a test function $\varphi\in \cS(\bR^d)$ to the equation,  integrate  over $[0,t)\times \bR^d$, and (at least formally)  get 
$$
(I^{1-\alpha}_t u(t,\cdot),\varphi)=\int^t_0(\phi(\Delta)u+f(u), \varphi)ds+\partial^{\beta-1}_t \int^t_0\int_{\bR^d} h(u)\varphi \dot{W}(dxds),
$$
where Walsh's stochastic integral against the space-time white noise is employed above. 
Applying $D^{1-\alpha}_t$ we further get
\begin{equation}
\label{eqn 4.8.1}
(u(t), \varphi)=I^{\alpha}_t (\phi(\Delta)u+f(u), \varphi)+\partial^{\beta-\alpha}_t \int^t_0\int_{\bR^d} h(u)\varphi \dot{W}(dxds).
\end{equation}
Now, let $\{\eta^{k}:k=1,2,\dots\}$ be an orthogonal basis on $L_2(\bR^d)$. Then (see \cite{Dalang,kry99analytic}) there exists a sequence of independent one-dimensional Wiener processes $\{w^k_t: k=1,2,\cdots\}$ such that 
$$
\int^t_0 \int_{\bR^d} X(s,x) \dot{W}(dxds)=\sum_{k=1}^{\infty}\int^t_0\int_{\bR^d} X(s,x)\eta^k(x)dx dw^k_s, \quad \forall t \, (a.s.)
$$
for any $X$ of the type $X=\zeta(x)1_{(\tau,\sigma]}(t)$, where $\tau,\sigma$ are bounded stopping times and $\zeta\in C^{\infty}_c(\bR^d)$.  
Thus \eqref{eqn 4.8.1} leads to the equation
\begin{equation}
\label{eqn 4.8.5}
\partial^{\alpha}_{t}u=\phi(\Delta)u+f(u)+\partial^{\beta}_{t}\sum_{k=1}^{\infty}\int_{0}^{t} h(u)\eta^k dw^{k}_{t}, \quad t>0;
\quad u(0,\cdot)=0.
\end{equation}

\begin{defn}
We say $u$ is a solution to equation \eqref{eqn 4.8.8} in the sense of distributions if  $u$ satisfies equation \eqref{eqn 4.8.5} in the sense of Definition \ref{defn 1}, that is,   for each $\varphi\in \cS(\bR^d)$,
$$
\left(u(t),\varphi\right)  =I_{t}^{\alpha}\left(\phi(\Delta)u+f(u),\varphi\right)+\partial_{t}^{\beta-\alpha}\sum_{k=1}^{\infty}\int_{0}^{t}\left(h(u)\eta^k,\varphi\right)dw_{s}^{k}
$$
holds  ``a.e. on $\Omega\times[0,T]$''.  Here $\partial^{\beta-\alpha}_t:=I^{\alpha-\beta}_t$ if $\beta\leq\alpha$. 
\end{defn}

Here comes our assumption on non-linear terms $f(u)$ and $h(u)$ together with some restrictions on $\beta$ and $d$. The argument $\omega$ is omitted  as usual.

\begin{assumption}\label{asm 4.7}

(i) The functions $f$ and $h$ are $\mathcal{P}\times\mathcal{B}(\bR^{d+1})$-measurable.

(ii) For each $\omega,t,x,u$ and $v$,
$$
|f(t,x,u)-f(t,x,v)| \leq K |u-v|,\quad |h(t,x,u)-h(t,x,v)| \leq  \xi(t,x) |u-v|,
$$
where $K$ is a constant and $\xi$ is a function of $(\omega,t,x)$. 

(iii) $\delta_0\in(1/4,1]$, $H_p^{\phi,1}\subset H_p^{\delta_0}$,  
\begin{equation}\label{eqn 10.17.13:27}
\beta < \left(1-\frac{1}{4\delta_0} \right)\alpha +\frac{1}{2}, \qquad d<2\delta_0\left(2-\frac{(2\beta-1)^+}{\alpha}\right)=:d_{0}.
\end{equation}
\end{assumption}

\begin{rem}
Using the Fourier multiplier theorem one can easily check that the embedding $H_p^{\phi,1}\subset H_p^{\delta_0}$ in Assumption \ref{asm 4.7} holds under  condition  \eqref{e:H}. 
\end{rem}

Note that $d_{0}\in(1,  4]$, and if $\beta<\alpha (1-\frac{3}{4\delta_0})+\frac{1}{2}$, then one can take $d=1,2,3$.  Recall
$$
f(0)=f(t,x,0), \quad h(0)=h(t,x,0).
$$

Here is our main result for SPDE driven by space-time white noise.
\begin{thm}\label{thm 10.27:15:35}
Suppose  Assumption \ref{asm 4.7} holds, and
$$
\|f(0)\|_{\mathbb{H}^{\phi,-k_{0}-c_{0}}_{p}(T)}+\|h(0)\|_{\mathbb{L}_{p}(T)}+\sup_{\omega,t}\|\xi\|_{L_{2s}}  < \infty,
$$
where $k_{0}$ and $s$ satisfy
\begin{equation}\label{eqn 4.7.8}
\frac{d}{2\delta_{0}}<k_{0}<\left(2-\frac{(2\beta-1)^{+}}{\alpha}\right)\wedge \frac{d}{\delta_{0}}, \quad \frac{d}{2k_{0}\delta_{0}-d}<s.
\end{equation}
Then equation \eqref{eqn 4.8.8} has unique solution $u\in\bH^{\phi,2-k_{0}-c_{0}}_{p}(T)$, and for this solution we have
\begin{equation*}
\|u\|_{\bH^{\phi,2-k_{0}-c_{0}}_{p}(T)} \leq C \left( \|f(0)\|_{\mathbb{H}^{\phi,-k_{0}-c_{0}}_{p}(T)}+\|h(0)\|_{\mathbb{L}_{p}(T)} \right),
\end{equation*} 
where $C$ is a constant independent of $u$.
\end{thm}

\begin{rem}
(i) Due to the second relation in \eqref{eqn 10.17.13:27}  one can always choose $k_{0}$ satisfying \eqref{eqn 4.7.8}.

(ii)  Note that the space for the solution is   $\bH^{\phi,2-k_{0}-c_{0}}_{p}(T)$, and  the  constant $2-k_{0}-c_{0}$ represents the regularity (or differentiability) of the solution with respect to the spatial variable. By  the definition of $c_{0}$ we have
\begin{equation*}
\begin{aligned}
0<2-k_{0}-c_{0}<
\begin{cases}
2-\frac{d}{2\delta_{0}}-\frac{2\beta-1}{\alpha} &: \quad  \beta>\frac{1}{2}
\\
2-\frac{d}{2\delta_{0}} &: \quad  \beta\leq \frac{1}{2}.
\end{cases}
\end{aligned}
\end{equation*}
If $\xi$ is bounded, then one can choose $r=1$. Thus by taking $k_{0}$ sufficiently close to $d/(2\delta_{0})$, one can make $2-k_{0}-c_{0}$ as close to the above upper bounds as one wishes.
\end{rem}

\mysection{A priori estimate for linear equation} \label{2105041510}

In this section we obtain a priori estimate for the solution to the linear equation
\begin{equation}
\label{eqn 4.10}
\partial_t^\alpha u = \phi(\Delta)u +\partial^{\beta}_t\sum_{k=1}^{\infty} \int^t_0 g^k dw^k_s,\quad t>0; \quad u(0,\cdot)=0.
\end{equation}
More precisely, we prove if $\frac{1}{2}<\beta<\alpha+\frac{1}{2}$,
  then
  \begin{equation}
\label{a priori}
\bE\|\phi(\Delta)^{1-(2\beta-1)/{2\alpha}} u\|^p_{L_p( (0,\infty)\times \bR^d)}\leq C \bE \||g|_{l_2}\|^p_{L_p((0,\infty)\times \bR^d)},
\end{equation}
where $C$ is a constant independent of $u$. 

If $\alpha=\beta=1$, then a version of \eqref{a priori} is obtained in \cite{kim2013parabolic}.  In this case, the solution to \eqref{eqn 4.10} is given (at least formally) by the formula
 $$
 u(t)=\sum_{k=1}^{\infty} \int^t_0 S_{t-s}g^k(\cdot, s)(x)dw^k_s,
$$
where $S_t: f\to e^{t\phi(\Delta)}f$,
and  \eqref{a priori} reads as
\begin{align*}
\left\|  \sqrt{-\phi(\Delta)}\sum_{k=1}^{\infty} \int^t_0 e^{(t-s)\phi(\Delta)}g^k(s)dw^k_s \right\|_{L_p(\Omega\times (0,\infty)\times \bR^d)} \leq c \left\||g|_{l_2} \right\|_{L_p(\Omega\times (0,\infty)\times \bR^d)}.
\end{align*}

We give two independent proofs of \eqref{a priori}. One is based on Krylov's analytic approach and the other is  based on $H^{\infty}$-calculus. The first proof is much elementary,  but  it  requires  long calculus and some extra condition on $\phi$.

\subsection{Analytic approach} \label{Analytic}

In this subsection, we impose the following assumption on $\phi$;
\begin{assumption}
   \label{ass bernstein}
    $\phi$ is  a Bernstein function for which 
there exist constants $\delta_0\in (0,1]$ and $\kappa_0>0$  such that
\begin{equation}\label{e:H}
\kappa_0 \left(\frac{R}{r}\right)^{\delta_0}\leq\frac{\phi(R)}{\phi(r)}, \qquad 0<r<R<\infty.
\end{equation}
\end{assumption}
By Assumption \ref{ass bernstein} and the concavity of $\phi$, we have
\begin{equation}\label{phiratio}
\kappa_0 \left(\frac{R}{r}\right)^{\delta_{0}} \leq \frac{\phi(R)}{\phi(r)} \leq \frac{R}{r},\quad 0<r<R<\infty.
\end{equation}

Note that we admit the case $\delta_{0}=1$, and we assume \eqref{e:H} for all $0<r<R<\infty$.   Here are some examples of Bernstein functions satisfying Assumption \ref{ass bernstein}
 (see e.g. \cite[Chapter 16]{schilling2012bernstein} for more examples):

  \begin{enumerate}[(1)]
\item Stable subordinators : $\phi(\lambda)=\lambda^\beta, \quad 0<\beta\leq1$;

\item Sum of stable subordinators : $\phi(\lambda)=\lambda^{\beta_1}+\lambda^{\beta_2}, \quad 0<\beta_1,\beta_2\leq1$;

\item Stable with logarithmic correction : $\phi(\lambda)=\lambda^\beta (\log(1+\lambda))^\gamma, \quad \beta\in(0,1), \gamma\in(-\beta,1-\beta)$;

\item Relativistic stable subordinators : $\phi(\lambda)=(\lambda+m^{1/\beta})^\beta-m, \quad \beta\in(0,1), m>0$;

\item Conjugate geometric stable subordinators : $\phi(\lambda)=\frac{\lambda}{\log(1+\lambda^{\beta/2})}, \quad \beta\in(0,2)$.
\end{enumerate}

Recall that $S=\left(S_t\right)_{t\ge0}$ is a subordinator with Laplace exponent $\phi$ and  $W=\left(W_t\right)_{t\ge0}$ is a $d$-dimensional Brownian motion, independent of $S$.  It is well-known that the subordinate Brownian motion $X_t=W_{S_t}$ is a L\'evy process in $\bR^d$ with characteristic exponent $\phi(|\xi|^2)$ (see e.g. \cite{bogdan2009potential,sato1999levy}), that is, 
$$
\bE e^{-i X_t \cdot \xi}=e^{-t \phi(|\xi|^2)}, \qquad \forall\, \, t>0, \, \xi \in \bR^d.
$$
Here, by $p(t,x)=p_d(t,x)$, we denote  the transition density of $X_t$.

Let $Q_t$ be a subordinator,  independent of $X_t$, having the Laplace transform 
\begin{equation*}
\bE \exp(-\lambda Q_t)= \exp(-t\lambda^\alpha).
\end{equation*}
Such process exists since the function $\lambda \to \lambda^{\alpha}$ is a Bernstein function (see \eqref{eqn 4.10.1}). 
Let 
$$
R_t:=\inf\{s>0 : Q_s>t\}
$$ be the inverse process of the  subordinator $Q_{t}$, and let  $\varphi(t,r)$ denote the probability density function of $R_t$. Then,  it is known that   (see \cite[Lemma 5.1]{kim2020nonlocal}  or \cite[Theorem 1.1]{chen2017time}),  the function
\begin{eqnarray*}
q(t,x):=\int_0^\infty p(r,x)d_r \bP(R_t\leq r)
=\int_0^\infty p(r,x) \varphi(t,r) \,dr   \label{eqn 8.1.1}
\end{eqnarray*}
becomes the fundamental solution to  equation
\begin{equation*} 
\partial_t^\alpha u = \phi(\Delta)u ,\quad t>0; \quad u(0,\cdot)=u_0.
\end{equation*}
That is, $q(t,x)$ is the function such that under appropriate smoothness condition on $u_0$, the function $u(t,x):=(q(t,\cdot)\ast u_0(\cdot))(x)$ solves the above equation.
 Actually, the definition of  $q(t,x)$ implies that $q(t,x)$ is the transition density of $Y_t:=X_{R_t}$, which is called subordinate Brownian motion delayed by an inverse subordinator.

For $\beta\in\bR$, denote 
\begin{equation}
\label{varphi}
\varphi_{\alpha,\beta}(t,r):=D^{\beta-\alpha}_{t}\varphi(t,r):=(D_{t}^{\beta-\alpha}\varphi(\cdot,r))(t),
\end{equation}
and for $(t,x)\in(0,\infty)\times \bR^{d}\setminus \{0\}$ define
$$
q_{\alpha,\beta}(t,x):=\int_{0}^{\infty}p(r,x) \varphi_{\alpha,\beta}(t,r)dr,
$$
and
\begin{equation*}
q^{\gamma}_{\alpha,\beta} (t,x) :=\int_0^\infty \phi(\Delta)^\gamma p (r,x) \varphi_{\alpha,\beta}(t,r) dr.
\end{equation*}

Below we collect some   some properties of $q_{\alpha,\beta}$ and $q^{\gamma}_{\alpha,\beta}$. The proof will be given in Appendix \ref{appendix}. 
\begin{lem} \label{prop:kernel esti. of q}

Let  $m\in\bN_0$, $\alpha, \gamma \in(0,1)$, and  $\beta\in\bR$.

(i)  $D^{\beta-\alpha}_{t}q(t,x)$ is well-defined for $(t,x)\in(0,\infty)\times\bR^d\setminus\{0\}$, and it holds that
$$
D^{\beta-\alpha}_{t}q(t,x)=q_{\alpha,\beta}(t,x).
$$

(ii) $D^m_x q_{\alpha,\beta} (t,x)$ is well defined for $(t,x)\in(0,\infty)\times\bR^d\setminus\{0\}$, and there exists $C=C(\alpha,\beta, d, \delta_0, \kappa_0, m)$ such that
\begin{equation}\label{eqn 09.03.19:26}
|D^m_x q_{\alpha,\beta} (t,x)|\leq C \, t^{2\alpha-\beta}\frac{\phi(|x|^{-2})}{|x|^{d+m}}.
\end{equation}
Additionally, if $t^\alpha \phi(|x|^{-2})\geq1$, then 
\begin{equation}\label{eqn 09.03.19:26-2}
|D^m_x q_{\alpha,\beta} (t,x)|\leq C t^{-\beta} \int_{(\phi(|x|^{-2}))^{-1}}^{2t^{\alpha}} (\phi^{-1}(r^{-1}))^{(d+m)/2} dr,
\end{equation}
where $\phi^{-1}$ denotes the inverse of $\phi$.

(iii)
$D^m_x q^{\gamma}_{\alpha,\beta} (t,x)$ is well defined for $(t,x)\in(0,\infty)\times\bR^d\setminus\{0\}$,  and there exists $C=C(\alpha,\beta, d, \delta_0, \kappa_0, \gamma, m)$ such that
\begin{equation} \label{qgamma whole est}
|D^m_x q^{\gamma}_{\alpha,\beta} (t,x)| \leq C \, t^{\alpha-\beta}\frac{\phi(|x|^{-2})^{\gamma}}{|x|^{d+m}}.
\end{equation}
Additionally, if $t^\alpha \phi(|x|^{-2}))\geq1$, then 
\begin{align} \label{qgamma partial est}
| D^m_x q^{\gamma}_{\alpha,\beta} (t,x)| \leq C\,t^{-\beta}  \int_{(\phi(|x|^{-2}))^{-1}}^{2t^{\alpha}} (\phi^{-1}(r^{-1}))^{(d+m)/2} r^{-\gamma}dr.
\end{align}

(iv) For any $t>0$,
\begin{equation} \label{int of q}
\int_{\bR^d} |q_{\alpha,\beta}(t,x)|dx \leq C\, t^{\alpha-\beta},
\end{equation}
and 
\begin{equation} \label{int of q^gamma}
\int_{\bR^d} |q^{\gamma}_{\alpha,\beta}(t,x)|dx \leq C\, t^{(1-\gamma)\alpha-\beta},
\end{equation}
 where $C=C(\alpha,\beta, d, \delta_0, \kappa_0, \gamma)$.
 
(v) For any $t>0$ and $\xi\in\bR^d$,
\begin{gather} \label{Mittag repre}
\cF_d(q^\gamma_{\alpha,\beta})(t,\xi) =-t^{\alpha-\beta} \phi(|\xi|)^\gamma  E_{\alpha,1-\beta+\alpha}(-t^\alpha\phi(|\xi|^2)),
\\
\label{eqn 09.13.13:08}
\cF_d(q_{\alpha,\beta})(t,\xi) =t^{\alpha-\beta} E_{\alpha,1-\beta+\alpha}(-t^\alpha\phi(|\xi|^2)),
\end{gather}
where $E_{\alpha,\beta}$ is the two-parameter Mittag-Leffler function defined as
\begin{equation*}
E_{\alpha,\beta}(z)=\sum_{k=0}^\infty \frac{z^k}{\Gamma(\alpha k+\beta)}, \quad z\in\bC, \alpha>0, \beta\in\bC.
\end{equation*}

\end{lem}

Next, we introduce the representation formula of the solution. For the rest of this section we assume
$$
\alpha\in (0,1), \quad 1/2<\beta<\alpha+1/2.
$$

\begin{lem} \label{lem:solution representation}
For given $g\in\mathbb{H}_{0}^{\infty}(T,l_{2})$, 
define
\begin{align} \label{sto sol re}
u(t,x) :=\sum_{k=1}^{\infty}\int_{0}^{t}\int_{\mathbb{R}^{d}}q_{\alpha,\beta}(t-s,x-y)g^{k}(s,y)dydw_{s}^{k}.
\end{align}
Then $u\in \bH^{\phi,2}_{p}(T)$  and satisfies \eqref{eqn 4.10}  in the sense of distributions
(see Definition  \ref{defn 1}).
\end{lem}

\begin{proof}
The proof is almost the same as that of \cite[Lemma 4.2]{kim16timefractionalspde} which treats the case $\phi(\lambda)=\lambda$. The only difference is that we need to use \eqref{Mittag repre}, \eqref{eqn 09.13.13:08} and \cite[lemma 4.1]{kim2020nonlocal} in place of corresponding results when $\phi(\lambda)=\lambda$.
\end{proof}

Denote $c_1:=2-(2\beta-1)/{\alpha}$ and 
$$
T_{t-s}^{\alpha,\beta}h(x):=\int_{\mathbb{R}^{d}}q^{c_1/2}_{\alpha,\beta}(t-s,x-y) h(y)dy, \qquad h\in C^{\infty}_c(\bR^d).
$$
This is well defined due to   Lemma \ref{prop:kernel esti. of q} (iv).  It is also easy to check (cf. Remark \ref{Hvaluedconti} (ii))
\begin{eqnarray*}
T_{t-s}^{\alpha,\beta}h(x)&=&\int_{\mathbb{R}^{d}}q_{\alpha,\beta}(t-s,x-y) \phi(\Delta)^{c_1/2}h(y)dy\\
&=&\phi(\Delta)^{c_1/2}\int_{\mathbb{R}^{d}}q_{\alpha,\beta}(t-s,x-y) h(y)dy.
\end{eqnarray*}
Take the solution $u$ from \eqref{sto sol re}. Then,  by the Burkholder-Davis-Gundy inequality,
\begin{eqnarray} \nonumber
\|\phi(\Delta)^{c_1/2}u\|^p_{\bL_p(T)}
&\leq& C(p) \bE \int_{\bR^{d}} \int_{0}^{T} \left( \int_{0}^{t} \sum_{k=1}^{\infty}|T^{\alpha,\beta}_{t-s}g^{k}(s)(x) |^{2}ds \right)^{p/2} dt dx
\\
&=&C(p) \bE \Big\|  \left( \int_{0}^{t} |T^{\alpha,\beta}_{t-s}g(s)(x) |^{2}_{l_2}ds \right)^{1/2} \Big\|^p_{L_p((0,T)\times \bR^d)}.
\label{bdc}
\end{eqnarray}

Now we estimate the right hand side of \eqref{bdc} in terms of $\|g\|_{\bL_p(T,l_2)}$ under a slightly general setting. Let $H$ be a  Hilbert space. For functions $g\in C_{c}^{\infty}(\mathbb{R}^{d+1};H)$, we define the operator $\cT$ as
\begin{align*}
\mathcal{T}g(t,x)
&:=\left[\int_{-\infty}^{t}\left| T_{t-s}^{\alpha,\beta}g(s,\cdot)(x)\right| _{H}^{2}ds\right]^{1/2},
\end{align*}
where $|\cdot|_H$ denotes the given norm in  $H$. Note that $\cT$ is sublinear since the Minkowski inequality yields
\begin{equation*}
\|f+g\|_{L_2((-\infty,t);H)}\leq \|f\|_{L_2((-\infty,t);H)}+\|g\|_{L_2((-\infty,t);H)}.
\end{equation*}

Note that \eqref{a priori} is a consequence of \eqref{bdc} and Theorem \ref{thm:L-P} below.

\begin{thm}  \label{thm:L-P}
Let $H$ be a separable Hilbert space,  and $T\in(-\infty,\infty]$. Then  for any $g\in C_{c}^{\infty}(\mathbb{R}^{d+1};H)$,
\begin{equation}
\label{est 314}
\int_{\mathbb{R}^{d}}\int_{-\infty}^{T}\left|\mathcal{T}g(t,x)\right|^{p}dtdx\leq C\int_{\mathbb{R}^{d}}\int_{-\infty}^{T}|g(t,x)|_{H}^{p}dtdx,
\end{equation}
where $C=C(\alpha,\beta,d,p,\delta_0,\kappa_0)$. Consequently, the operator
$\mathcal{T}$ is continuously extended to $L_{p}(\mathbb{R}^{d+1};H)$.
\end{thm} 

The proof of the theorem is given later after some preparations. The main strategy is as follows.

\begin{itemize}

\item[1.] First, we control the sharp function of $\cT g$ in terms of maximal function of $g$ (the definitions of the sharp and maximal functions are given below), that is, we prove
$$
(\cT g)^{\sharp}(t,x) \leq C(\bM_t\bM_x|g|^2_{H}(t,x) )^{1/2}, \quad \forall \, (t,x) \quad \text{uniformly on}\, \Omega.
 $$

 \item[2.] Then, we apply Fefferman-Stein inequality and Hardy-Littlewood maximal inequality to obtain \eqref{est 314}.
 \end{itemize}

Recall that $g=(g^1,g^2,\cdots) \in\mathbb{H}_{0}^{\infty}(T,l_{2})$ if $g^{k}=0$ for
all sufficiently large $k$, and each $g^{k}$ is of the type
$$
g^{k}(t,x)=\sum_{i=1}^{n}1_{(\tau_{i-1},\tau_{i}]}(t)g^{ik}(x),
$$
where $0\leq \tau_0 \leq \tau_1 \leq \cdots \leq \tau_n$ are bounded  stopping times, and $g^{ik}\in C_{c}^{\infty}(\mathbb{R}^{d})$.

The following result, Lemma \ref{lem:L2 result}, is a version of Theorem \ref{thm:L-P} for $p=2$.   For the proof,  we use  the following  fact on the Mittag-Leffler function: if $\alpha \in (0,1)$ and $b\in \bR$, then  there exist positive constants  $\varepsilon=\varepsilon(\alpha)$
and $C=C(\alpha,b)$ such that
\begin{equation}
            \label{eqn 8.19.2}
|E_{\alpha,b}(z)|\leq C(1 \wedge  |z|^{-1}) \quad \text{ if } \quad \pi-\varepsilon \leq |\arg (z)|\leq \pi
\end{equation}
(see e.g. \cite[Theorem 1.6]{podlubny1998fractional}).

\begin{lem}
                        \label{lem:L2 result}
For any $T\in (-\infty, \infty]$ and $g\in C_{c}^{\infty}(\mathbb{R}^{d+1};H)$,
\begin{equation}
             \label{eqn 8.19.1}
\int_{\mathbb{R}^{d}}\int_{-\infty}^{T}|\mathcal{T}g(t,x)|^{2}dtdx\leq C \int_{\mathbb{R}^{d}}\int_{-\infty}^{T}\left| g(t,x)\right| _{H}^{2}dtdx,
\end{equation}
where $C=C(\alpha,\beta,d)$ is independent of $T$.
 \end{lem}
\begin{proof}
 We follow the proof of \cite[Lemma 3.5]{kim16timefractionalspde}.

{\bf{Step 1}}.  First, assume $g(t,x)=0$ for $t \leq 0$.  In this case we may further assume $T>0$ since the left hand side of \eqref{eqn 8.19.1} is zero if $T\leq 0$. Since  $g(t,x)=\cT g(t,x)=0$ for $t \leq 0$, by Parseval's identity,
  \begin{align}
 & \int_{\mathbb{R}^{d}}\int_{-\infty}^{T}|\mathcal{T}g(t,x)|^{2}dtdx  \nonumber \\
 & =\int_{0}^{T}\int_{0}^{t}\int_{\bR^d}\phi(|\xi|^{2})^{c_1}\left|\hat{q}_{\alpha,\beta}(t-s,\cdot) (\xi)\right|^{2}\left| \hat{g}(s,\xi)\right| _{H}^{2}d\xi dsdt  \nonumber
 \\
 & = \int_{\bR^d} \int_{0}^{T}\int_{s}^{T} \phi(|\xi|^{2})^{c_1}\left|\hat{q}_{\alpha,\beta}(t-s,\cdot) (\xi)\right|^{2}\left| \hat{g}(s,\xi)\right| _{H}^{2} dtds d\xi \nonumber
 \\
 & = \int_{\bR^d} \int_{0}^{T}\int_{0}^{T-s} \phi(|\xi|^{2})^{c_1}\left|\hat{q}_{\alpha,\beta}(t,\cdot) (\xi)\right|^{2}\left| \hat{g}(s,\xi)\right| _{H}^{2} dtds d\xi.  \label{eqn 4.12.1}
\end{align}
By Lemma \ref{prop:kernel esti. of q} (v) and \eqref{eqn 8.19.2} (recall $\beta>1/2$), for $0<s<T$,
\begin{align*}
&\int_{0}^{T-s} \phi(|\xi|^{2})^{c_1}\left|\hat{q}_{\alpha,\beta}(t,\cdot) (\xi)\right|^{2} dt\\
 & \leq \phi(|\xi|^{2})^{c_1}\int^{T}_0\left|t^{\alpha-\beta}E_{\alpha,1-\beta+\alpha}(-\phi(|\xi|^{2})t^{\alpha})\right|^{2}dt
 \\
 & \leq C \phi(|\xi|^{2})^{c_1} \int_{0}^{\phi(|\xi|^{2})^{-\frac{1}{\alpha}}} \hspace{-1.5mm} t^{2(\alpha-\beta)} dt
 +C 1_{A}(\xi) \phi(|\xi|^{2})^{c_1} \int_{\phi(|\xi|^{2})^{-\frac{1}{\alpha}}}^{T} \left|\frac{t^{\alpha-\beta}}{\phi(|\xi|^{2})t^{\alpha}}\right|^{2} dt
 \\
 &\leq C \phi(|\xi|^{2})^{(c_1-2+\frac{2\beta-1}{\alpha})}+C\phi(|\xi|^2)^{c_1-2}\phi(|\xi|^{2})^{(\frac{2\beta-1}{\alpha})}\leq C,
\end{align*}
where $A=\{\xi:\phi(|\xi|^2)\geq T^{-\alpha}\}$.
Thus,  \eqref{eqn 4.12.1} and  Parseval's identity yield
\begin{align*}
\int_{\mathbb{R}^{d}}\int_{-\infty}^{T}|\mathcal{T}g(t,x)|^{2}dtdx \leq C \int_{0}^{T}\int_{\mathbb{R}^{d}}\left| g(t,x)\right|_{H}^{2}dxdt,
\end{align*} 
and \eqref{eqn 8.19.1} holds for all $T>0$ with a constant independent of $T$. It follows that
 \eqref{eqn 8.19.1} also holds for $T= \infty$.

{\bf{Step 2}}. General case. Take $a\in \bR$ so that $g(t,x)=0$ for $t\leq a$. Then obviously, for $\bar{g}(t,x):=g(t+a,x)$ we have $\bar{g}(t)=0$ for $t\leq 0$. Also note that
\begin{align*}
\left(\mathcal{T}g(t+a)\right)^{2}&=\left(\int_{-\infty}^{t+a}  \left|  \int_{\bR^{d}}q^{c_{1}/2}_{\alpha,\beta}(t+a-s,x-y)g(s,y) dy \right|^{2}_H ds \right)
\\
&=\left(\int_{-\infty}^{t} \left| \int_{\bR^{d}}q^{c_{1}/2}_{\alpha,\beta}(t-s,x-y)\bar{g}(s,y) dy \right|^{2}_H ds \right)
\\
&=\left(\mathcal{T}\bar{g}(t)\right)^{2}.
\end{align*} 
Thus it is enough to apply the result of Step 1 with $\bar{g}$ and $T-a$ in place of $g$ and $T$ respectively. The lemma is proved.
\end{proof}

For $b>0$ and $(t,x)\in\bR^{d+1}$, we define 
\begin{equation*}
\kappa(b):=\left(\phi(b^{-2})\right)^{-1/\alpha},\quad B_b(x)=\{y\in\bR^d : |y-x|<b\},
\end{equation*}
and 
\begin{equation*}
I_b(t)=(t-\kappa(b),t), \quad Q_b(t,x)=I_b(t)\times {B}_b(x).
\end{equation*}
We also  denote
\begin{equation*}
I_b=I_b(0),  \quad B_b=B_b(0), \quad Q_b=Q_b(0,0).
\end{equation*}

For measurable functions $h(t,x)$ on $\mathbb{R}^{d+1},$ we define the sharp function
\begin{equation*}
h^{\#}(t,x)=\sup_{(t,x)\in Q}\aint_{Q}\left|h(r,z)-h_{Q}\right|drdz,
\end{equation*}
where
$$
h_{Q}=\aint_{Q}h(s,y)dyds,
$$
and the supremum is taken over all $Q\subset\mathbb{R}^{d+1}$  of the form $Q_b$ containing
$(t,x)$.

For functions $h$ on $\bR^d$ we define the maximal function
\begin{equation*}
\bM_xh(x):=\sup_{x\in B_r(z)}\frac{1}{|B_{r}(z)|}\int_{B_{r}(x)}|h(y)|dy=\sup_{x\in B_r(z)}\aint_{B_{r}(z)}|h(y)|dy.
\end{equation*}
We also use the notation $\bM_th(t)$ when $d=1$  for functions depending on $t$. 
For measurable functions $h(t,x)$ set
$$
\bM_xh(t,x)=\bM_x\left(h(t,\cdot)\right)(x),\quad\bM_{t}h(t,x)=\bM_{t}\left(h(\cdot,x)\right)(t),
$$
and
$$
\bM_t \bM_x h(t,x)=\bM_t\left(\bM_xh(\cdot,x)\right)(t).
$$

Below  we record some  useful computations which are often used in the rest of this section.

\begin{lem} \label{lem 09.05.17:48}

(i) Let $f$ be a nonnegative integrable function on $\bR$. Assume there exists $a>0$ such that $f(t)=0$ if $t>-2a$. Then for any  $\nu>1$ and $t\in(-a,0)$,
\begin{align}\label{f(s-r) 200413}
\int_{-\infty}^{-2a} \int_{-a-r}^{-r}  f(r) s^{-\nu} dsdr &= \int_{-\infty}^{-2a} \int_{-a}^0  f(r)(s-r)^{-\nu} dsdr \nonumber
\\
&\leq C(\nu) a^{2-\nu}\bM_t f(t).
\end{align}

(ii) For positive real numbers $\nu, \theta$ and $r$, define
\begin{equation}
\label{G}
G_{\nu,\theta}(\rho):=\frac{\phi(\rho^{-2})^{\nu}}{\rho^{\theta}}, \qquad \rho>0
\end{equation}
and
\begin{equation}
\label{H}
H_{\nu,\theta}(r,\rho):= \int_{(\phi(\rho^{-2}))^{-1}}^{2r^\alpha} (\phi^{-1}(l^{-1}))^{\theta/2} l^{-\nu}  dl, \qquad r^{\alpha}\phi(\rho^{-2})\geq 1.
\end{equation}
Then we have
\begin{gather}
\left|\frac{d}{d\rho}G_{\nu,\theta}(\rho) \right| \leq C(\nu,\theta) G_{\nu,\theta+1}(\rho), \qquad \rho>0, \label{eqn 09.05.17:21}
\\
\left|\frac{d}{d\rho}H_{\nu,\theta}(r,\rho) \right| \leq C(\alpha,\beta,d,\delta_{0},\kappa_0, \nu,\theta) H_{\nu,\theta+1}(r,\rho), \qquad r^{\alpha}\phi(\rho^{-2})\geq 1.  \label{eqn 09.05.17:21-2}
\end{gather}
\end{lem}
\begin{proof}

(i) Using  integration by parts, we get
\begin{align*}
&\int_{-\infty}^{-2a} \int_{-a}^0 f(r)(s-r)^{-\nu} dsdr \\
&= \frac{1}{\nu-1} \int_{-\infty}^{-2a} f(r) \left((-a-r)^{1-\nu}-(-r)^{1-\nu}\right) dr
\\
&\leq C(\nu)\int_{-\infty}^{-2a} \left(\int_{r}^0 f(\tilde{r})d\tilde{r} \right) \left((-a-r)^{-\nu}-(-r)^{-\nu}\right)dr.
\end{align*}
Since $0<p<q$ implies $p^{-\nu}-q^{-\nu}\leq \nu(q-p)p^{-\nu-1}$,  we have for  $r<-2a<0$,
$$
(-a-r)^{-\nu}-(-r)^{-\nu} \leq \nu a(-a-r)^{-\nu-1}.
$$
Therefore,
\begin{align*}
\int_{-\infty}^{-2a} \int_{-a}^0 f(r)(s-r)^{-\nu} dsdr &\leq C(\nu) a\int_{-\infty}^{-2a} \left(\int_{r}^0 f(\tilde{r})d\tilde{r} \right) (-a-r)^{-\nu-1} dr
\\
&\leq C(\nu)a\bM_t f(t)\int_{-\infty}^{-2a} -r(-r/2)^{-\nu-1} dr 
\\
&\leq C(\nu) a^{2-\nu}\bM_t f(t).
\end{align*}

(ii) By \eqref{eqn 07.19.14.35}, we have
\begin{align*}
\left|\frac{d}{d\rho} G_{\nu,\theta}(\rho)\right|&=\left|-\theta\frac{\phi(\rho^{-2})^{\nu}}{\rho^{\theta+1}}-2\nu\frac{\phi(\rho^{-2})^{\nu-1}}{\rho^{\theta}}\phi '(\rho^{-2})\rho^{-3}\right|
\\
&\leq C(\nu,\theta) \frac{\phi(\rho^{-2})^{\nu}}{\rho^{\theta+1}}=C(\nu,\theta) G_{\nu,\theta+1}(\rho).
\end{align*}
Thus \eqref{eqn 09.05.17:21} is proved. Now we prove \eqref{eqn 09.05.17:21-2}. For $0<v<r^{\alpha}$, define
$$
\mathcal{H}_{\nu,\theta}(r,v)= \int_{v}^{2r^\alpha} (\phi^{-1}(l^{-1}))^{\theta/2} l^{-\nu}  dl.
$$
Applying the fundamental theorem of calculus and using \eqref{eqn 07.19.14.35},
\begin{align*}
\left|\frac{d}{d\rho}H_{\nu,\theta}(r,\rho) \right|&= C \left|\frac{d}{dv}\mathcal{H}_{\nu,\theta}(r,v)\big|_{v=(\phi(\rho^{-2}))^{-1}} (\phi(\rho^{-2}))^{-2}\phi'(\rho^{-2})\rho^{-3}\right|
\\
& \leq C \frac{\phi(\rho^{-2})^{\nu}}{\rho^{\theta}} (\phi(\rho^{-2}))^{-1}\rho^{-1}=C  \frac{\phi(\rho^{-2})^{\nu-1}}{\rho^{\theta+1}}.
\end{align*}
By \eqref{e:H} with $R=\phi^{-1}(2l^{-1})$ and $r=\rho^{-2}$,
\begin{align*}
 \frac{\phi(\rho^{-2})^{\nu-1}}{\rho^{\theta+1}} &=\int_{(\phi(\rho^{-2}))^{-1}}^{2(\phi(\rho^{-2}))^{-1}}  \frac{\phi(\rho^{-2})^{\nu}}{\rho^{\theta+1}}  dl \nonumber
\\
&\leq C \int_{(\phi(\rho^{-2}))^{-1}}^{2(\phi(\rho^{-2}))^{-1}} (\phi^{-1}(l^{-1}))^{(\theta+1)/2} l^{-\nu}dr \nonumber
\\
&\leq C \int_{(\phi(\rho^{-2}))^{-1}}^{2r^{\alpha}} (\phi^{-1}(l^{-1}))^{(\theta+1)/2} l^{-\nu} dl
\\
&=C H_{\nu,\theta+1}(r,\rho).
\end{align*}
Therefore, we have \eqref{eqn 09.05.17:21-2}, and the lemma is proved.
\end{proof}

We will also frequently use the following version of integration by parts formula:  if $F$ and $G$ are smooth enough, then for any $0<\varepsilon <R<\infty$,
\begin{align}
\int_{\epsilon\leq|z|\leq R}F(z)G(|z|)dz & =-\int_{\epsilon}^{R}G'(\rho)\left[\int_{|z|\leq\rho}F(z)dz\right]d\rho\nonumber \\
 & \qquad+G(R)\int_{|z|\leq R}F(z)dz-G(\epsilon)\int_{|z|\leq\epsilon}F(z)dz.
                    \label{eq:integration by parts}
\end{align}
This is easily obtained  using the relations
\begin{align*}
\int_{\varepsilon \leq |z| \leq R} F(z) G(|z|)\,dz
&=\int_\varepsilon^R G(\rho)\left(\int_{\partial B_\rho(0)} F(s) \, dS_{\rho} \right) d\rho
\\
&=\int_\varepsilon^R G(\rho)\frac{d}{d\rho} \left(\int_{B_\rho(0)}F(z)\, dz \right)d\rho.
\end{align*}
and  applying standard integration by parts formula to the last term above.

\vspace{1mm}

In the following lemmas, Lemmas \ref{ininestimate 200316} -- \ref{outouttimeestimate 200414},  we estimate the mean oscillation of $\cT g$ on $Q_b$.  For this, 
we  consider the following two cases

\begin{itemize}

\item  $g$ has support in   $(-3\kappa(b), \infty)\times \bR^d$ (see Lemma  \ref{inwholeestimate 200316}),

 \item  $g$ has support in $(-\infty, -2\kappa(b))\times \bR^d$.

\end{itemize}

The second case above is further divided into the  cases

\begin{itemize}

\item  $g$ has support in  $(-\infty,-2\kappa(b))\times B_{3b}$  (see Lemma \ref{outinestimate 200414}),

\item  $g$ has  support in  $(-\infty,-2\kappa(b))\times B_{2b}^c$ (see Lemmas  \ref{outoutspaceestimate 200414} and  \ref{outouttimeestimate 200414}).
\end{itemize}

Note that, by Jensen's inequality,  for any $c\in \bR$ and function $h$,
\begin{eqnarray}
                \nonumber
 \left( \aint_{Q}\left|h(r,z)-h_{Q}\right|drdz\right)^2
&\leq& \aint_{Q}\left|h(r,z)-h_{Q}\right|^2drdz  \\ 
&=&\aint_{Q}\left|\aint_Q(h(r,z)-h(s,y))dsdy\right|^2drdz   \label{type1} \\
&\leq&4 \aint_Q |h(r,z)-c|^2 drdz.      \label{4times}
\end{eqnarray}
We will consider type \eqref{type1}  and type \eqref{4times} below.

\begin{lem} \label{ininestimate 200316}
Let $g\in C_{c}^{\infty}(\mathbb{R}^{d+1};H)$  have a support in $(-3\kappa(b),3\kappa(b))\times B_{3b}$.
Then for any $(t,x)\in Q_{b}$,
$$
\aint_{Q_{b}}\left|\mathcal{T}g(s,y)\right|^{2}dsdy\leq C(\alpha,\beta,d)\bM_{t}\bM_x\left| g\right| _{H}^{2}(t,x),
$$
where $C=C(\alpha,\beta,d)$.
\end{lem}

\begin{proof}
By Lemma \ref{lem:L2 result},
\begin{align*}
\int_{Q_{b}}\left|\mathcal{T}g(s,y)\right|^{2}dsdy &\leq C \int_{(-3\kappa(b),3\kappa(b))\times B_{3b}}\left| g(s,y)\right| _{H}^{2}dyds
\\
&\leq C b^d\int_{-3\kappa(b)}^{3\kappa(b)} \bM_x\left|g\right| _{H}^{2}(s,x)ds
\\
&\leq C\kappa(b)b^d \bM_t \bM_x\left|g\right| _{H}^{2}(t,x).
\end{align*}
This certainly proves the lemma.
\end{proof}

\begin{lem} \label{inwholeestimate 200316}

Let $g\in C_{c}^{\infty}(\mathbb{R}^{d+1};H)$  have a support in $(-3\kappa(b),\infty)\times\bR^d$. Then for any $(t,x)\in Q_{b}$,
$$
\aint_{Q_{b}}\left|\mathcal{T}g(s,y)\right|^{2}dsdy\leq C \bM_{t}\bM_x\left| g\right| _{H}^{2}(t,x),
$$
where $C=C(\alpha,\beta,d,\delta_0,\kappa_0)$.
\end{lem}

\begin{proof}
Take $\zeta_0=\zeta_0(t)\in C^{\infty}(\bR)$ such that $0\leq \zeta_0\leq 1$, $\zeta_0(t)=1$ for $t\leq 2\kappa(b)$, and $\zeta_0(t)=0$ for $t\geq 5\kappa(b)/2$.
Note that $\cT g=\cT (g\zeta_0)$ on $Q_b$, and $|g\zeta_0|\leq |g|$ leads to
$$
\bM_{t}\bM_x\left| g\zeta_0\right| _{H}^{2}(t,x) \leq \bM_{t}\bM_x\left| g\right| _{H}^{2}(t,x).
$$
Therefore, it is enough to assume $f(t,x)=0$ if $|t|\geq 3\kappa(b)$.
Take $\zeta\in C_{c}^{\infty}(\mathbb{R}^{d})$ such that $\zeta=1$
in $B_{2b}$ and $\zeta=0$ outside $B_{5b/2}$. Recall that $\cT$ is a sublinear operator, and therefore
$$
\mathcal{T}g\leq\mathcal{T}(\zeta g)+\mathcal{T}( (1-\zeta)g).
$$
Since $\mathcal{T}(\zeta g)$ can be estimated by Lemma \ref{ininestimate 200316},
we may assume that $g(t,x)=0$ if $x\in B_{2b}$.

Observe
that if $(s,y)\in Q_{b}$ and $\rho>b$, then
\begin{equation}\label{eq:2}
|x-y|\leq2b,\quad B_{\rho}(y)\subset B_{2b+\rho}(x)\subset B_{3\rho}(x),
\end{equation}
whereas if $\rho\leq b$ then for $z\in B_{\rho}$,
$|y-z|\leq 2b$ and thus $g(r,y-z)=0$.
Hence, by \eqref{qgamma whole est}, \eqref{eq:integration by parts} and \eqref{eq:2}, for $s>r$,
\begin{align*}
&\left|\int_{\bR^d} q_{\alpha,\beta}^{c_1/2}(s-r,z)  g(r,y-z) dz\right|_{H} \leq \int_{|z|\geq b} |q_{\alpha,\beta}^{c_1/2}(s-r,z)|  |g(r,y-z)|_{H} dz
\\
& \qquad \leq \int_{b}^{\infty} (s-r)^{\alpha-\beta} \frac{d}{d\rho}G_{c_{1}/2,d}(\rho) \left[ \int_{|z|\leq \rho}|g(r,y-z)|_{H} dz\right] d\rho
\\
& \qquad \leq C\int_b^\infty (s-r)^{\alpha-\beta} \frac{d}{d\rho}G_{c_{1}/2,d}(\rho) \left[\int_{B_{3\rho}(x)} |g(r,z)|_H dz \right] d\rho,
\end{align*}
where
$G_{c_{1}/2,d}(\rho)$ is taken from \eqref{G}. Therefore, by \eqref{eqn 09.05.17:21},
\begin{align*}
&\left|\int_{\bR^d} q_{\alpha,\beta}^{c_1/2}(s-r,z)  g(r,y-z) dz\right|_{H}
\\
&\leq C\int_b^\infty (s-r)^{\alpha-\beta} \frac{\phi(\rho^{-2})^{c_1/2}}{\rho^{d+1}} \left[\int_{B_{3\rho}(x)} |g(r,z)|_H dz \right] d\rho
\\
& \leq C \int_b^\infty (s-r)^{\alpha-\beta}\frac{\phi(\rho^{-2})^{c_1/2}}{\rho} \bM_x |g|_H(r,x) d\rho
\\
& \leq C (s-r)^{\alpha-\beta} \phi(b^{-2})^{c_1/2}\bM_x|g|_H(r,x).
\end{align*}
For the last inequality we use \eqref{phiint} with $\phi^{c_{1}/2}$. Since $(\bM_x\left| g\right| _{H})^{2}\leq\bM_x\left| g\right| _{H}^{2}$, we have
\begin{align*}
&\int_{Q_{b}} \left|\mathcal{T}g(s,y)\right|^{2}dsdy =  \int_{Q_{b}}\int_{-3\kappa(b)}^{s}\left|\int_{\bR^d} q_{\alpha,\beta}^{c_1/2}(s-r,z)  g(r,y-z) dz\right|_{H}^2 drdsdy\\
&\leq C\int_{Q_{b}}\int_{-3\kappa(b)}^{s}\left[(s-r)^{2(\alpha-\beta)}\phi(b^{-2})^{c_1}\bM_x\left| g\right| _{H}^{2}(r,x)\right]drdsdy\\
&\leq C\phi(b^{-2})^{c_1}\int_{B_b}\int_{-3\kappa(b)}^0\left(\int_{r}^{0}(s-r)^{2(\alpha-\beta)}ds\right)\bM_x\left| g\right|_{H}^{2}(r,x)drdy
\\
&\leq C\phi(b^{-2})^{c_1}\int_{B_b}\int_{-3\kappa(b)}^0 (-r)^{\alpha c_1}\bM_x\left| g\right|_{H}^{2}(r,x)drdy
\\
&\leq C(\kappa(b))^{\alpha c_1}\phi(b^{-2})^{c_1}
\\
&\quad \times\int_{B_b}\int_{-3\kappa(b)}^0\bM_x\left| g\right| _{H}^{2}(r,x)drdy \leq C \kappa(b)b^d\bM_t \bM_x\left| g\right|_{H}^{2}(t,x).
\end{align*}
The lemma is proved.
\end{proof}

\begin{lem} \label{outinestimate 200414}
Let $g\in C_{c}^{\infty}(\mathbb{R}^{d+1};H)$ have a support in $(-\infty,-2\kappa(b))\times B_{3b}$. Then for any $(t,x)\in Q_{b}$,
$$
\aint_{Q_{b}}\left|\mathcal{T}g(s,y)\right|^{2}dsdy\leq C \bM_{t}\bM_x\left| g\right| _{H}^{2}(t,x),
$$
where $C=C(\alpha,\beta,d,c,\delta_0,\kappa_0)$.
\end{lem}

\begin{proof}
By definition of $\cT g$ and  Fubini's theorem, 
\begin{align*}
&\int_{Q_b} |\cT g (s,y)|^2 dsdy
\\
&\leq \int_{-\kappa(b)}^0 \int_{-\infty}^{-2\kappa(b)}  \int_{B_b} \left|\int_{\mathbb{R}^{d}}  q_{\alpha,\beta}^{c_1/2}\left(s-r,z\right) g(r,y-z) dz\right|^{2}_{H} dy dr ds.
\end{align*}
By assumption, for  any $y\in B_b$, the function $g(r,y-\cdot)$ vanishes on $B^c_{4b}$.  Therefore, by Minkowski's inequality, for $s>r$,
\begin{align}
& \int_{B_b} \left|\int_{\mathbb{R}^{d}}  q_{\alpha,\beta}^{c_1/2}\left(s-r,z\right) g(r,y-z) dz\right|^{2}_{H} dy \nonumber
\\
&\leq  \int_{B_b} \left|\int_{B_{4b}} \left| q_{\alpha,\beta}^{c_1/2}\left(s-r,z\right)\right|\left| g(r,y-z)\right| _{H}dz\right|^{2}dy  \nonumber
\\
 & \leq\left(\int_{B_{4b}}\left[\int_{B_b}\left| g(r,y-z)\right|_{H}^{2}dy\right]^{1/2}\left|q_{\alpha,\beta}^{c_1/2}\left(s-r,z\right) \right|dz\right)^{2} \nonumber
\\
 & \leq\left(\int_{B_{4b}}\left[\int_{B_{5b}}\left| g(r,y)\right|_{H}^{2}dy\right]^{1/2}\left|q_{\alpha,\beta}^{c_1/2}\left(s-r,z\right)\right|dz\right)^{2}  \nonumber
 \\
 & \leq C b^{d}\bM_x\left| g\right| _{H}^{2}(r,x)\left(\int_{B_{4b}}\left|q_{\alpha,\beta}^{c_1/2}\left(s-r,z\right)\right|dz\right)^{2}. \label{eqn 4.13.5}
\end{align}
Therefore, by \eqref{eqn 4.13.5},  we have
\begin{align}
&\int_{Q_b} |\cT g (s,y)|^2 dsdy  \nonumber
\\
&\leq C b^{d}\int_{-\kappa(b)}^0 \int_{-\infty}^{-2\kappa(b)} \bM_x\left| g\right| _{H}^{2}(r,x)\left(\int_{B_{4b}}\left|q_{\alpha,\beta}^{c_1/2}\left(s-r,z\right)\right|dz\right)^{2}drds  \nonumber
\\
&=C b^{d} \int_{-\infty}^{-2\kappa(b)}\int_{-\kappa(b)}^0  \bM_x\left| g\right| _{H}^{2}(r,x)\left(\int_{B_{4b}}\left|q_{\alpha,\beta}^{c_1/2}\left(s-r,z\right)\right|dz\right)^{2}dsdr   \nonumber
\\
&=C b^{d} \int_{-\infty}^{-2\kappa(b)} \bM_x\left| g\right| _{H}^{2}(r,x)  \left[\int_{-\kappa(b)-r}^{-r} \left(\int_{B_{4b}}\left|q_{\alpha,\beta}^{c_1/2}\left(s,z\right)\right|dz\right)^{2}ds \right] dr   \label{eqn 4.13.11}
\\
&=C b^d \left( \int_{-m\kappa(b)}^{-2\kappa(b)}\cdots dr+   \int_{-\infty}^{-m\kappa(b)}\cdots  dr \right)=: Cb^d \left( I(b)+II(b) \right), \nonumber
\end{align}
where $m$ is any fixed integer such that 
$m\geq 3$ and  $(m-1)\kappa(b)>\kappa(4b)$. Such integer $m$ exists due to \eqref{phiratio}.  

Obviously, to finish the proof of the lemma, it suffices to prove
\begin{equation}
   \label{eqn 4.15.9}
I(b)+II(b)\leq C \kappa(b)\bM_t \bM_x |g|_H^2(t,x).
\end{equation}
To prove this, we first  consider the integral inside the square brackets in \eqref{eqn 4.13.11}.  If $-m\kappa(b)<r<-2\kappa(b)$, then 
by  \eqref{int of q^gamma},
\begin{align}
\nonumber
\Big[ \,\, \Big]&:=\int_{-\kappa(b)-r}^{-r}  \left(\int_{B_{4b}}\left|q_{\alpha,\beta}^{c_1/2}\left(s,z\right)\right|dz\right)^{2}ds
\\
&\leq\int_{\kappa(b)}^{m\kappa(b)}  \left(\int_{B_{4b}}\left|q_{\alpha,\beta}^{c_1/2}\left(s,z\right)\right|dz\right)^{2}ds \nonumber
\\
\nonumber
&\leq \int_{\kappa(b)}^{m\kappa(b)}  \left(\int_{\bR^d}\left|q_{\alpha,\beta}^{c_1/2}\left(s,z\right)\right|dz\right)^{2}ds
\\
&\leq C \int_{\kappa(b)}^{m\kappa(b)} s^{-1} ds = C\ln m\leq C. \nonumber
\end{align}
Therefore, 
\begin{equation}
  \label{eqn 4.13.10}
I(b)\leq C \int_{-m\kappa(b)}^{-2\kappa(b)} \bM_x\left| g\right| _{H}^{2}(r,x) dr \leq C \kappa(b)\bM_t \bM_x |g|_H^2(t,x).
\end{equation}

Next, we estimate $II(b)$. If
 $r\leq -m\kappa(b)$, then  $-\kappa(b)-r>(m-1)\kappa(b)>\kappa(4b)$. Therefore,  by \eqref{qgamma partial est},
\begin{align}
\nonumber
\Big[ \,\, \Big]&:=\int_{-\kappa(b)-r}^{-r}  \left(\int_{B_{4b}}\left|q_{\alpha,\beta}^{c_1/2}\left(s,z\right)\right|dz\right)^{2}ds
\\
\label{eqn 4.14.5}
&\leq C \int_{-\kappa(b)-r}^{-r} \left( s^{-\beta} \int_{B_{4b}} \int_{(\phi(|z|^{-2}))^{-1}}^{2s^{\alpha}} (\phi^{-1}(l^{-1}))^{d/2} l^{-c_1/2} dl dz\right)^2 ds.
\end{align}

\noindent
By Fubini's theorem, for $s>\kappa(4b)$ (equivalently, $2s^\alpha > (\phi(b^{-2}/16))^{-1}$),
\begin{align}
\nonumber
&s^{-\beta} \int_{B_{4b}}\int_{(\phi(|z|^{-2}))^{-1}}^{2s^{\alpha}} (\phi^{-1}(l^{-1}))^{d/2} l^{-c_1/2} dl dz
\\
\nonumber
&= s^{-\beta}\int_0^{\left(\phi(b^{-2}/16) \right)^{-1}} \int_{|z|\leq \left(\phi^{-1}(l^{-1})\right)^{-1/2}} (\phi^{-1}(l^{-1}))^{d/2} l^{-c_1/2} dz dl 
\\
\nonumber
& \quad +s^{-\beta} \int_{\left(\phi(b^{-2}/16) \right)^{-1}}^{2s^\alpha} \int_{B_{4b}} (\phi^{-1}(l^{-1}))^{d/2} l^{-c_1/2} dz dl
\\
&\leq C \left(\phi(b^{-2}/16) \right)^{c_1/2-1} s^{-\beta} + C b^d   s^{-\beta} \int_{\left(\phi(b^{-2}/16) \right)^{-1}}^{2s^\alpha} (\phi^{-1}(l^{-1}))^{d/2}l^{-c_1/2} dl \nonumber\\
&\leq C \left(\phi(b^{-2}/16) \right)^{c_1/2-1} s^{-\beta} + C    s^{-\beta}  \phi ( b^{-2}/16 )^{-d/2} \int_{\left(\phi(b^{-2}/16) \right)^{-1}}^{2s^\alpha} l^{(-d-c_1)/2} dl.  
\label{eqn 4.15.10}
\end{align}
The last inequality above is due to  \eqref{phiratio} with $R=b^{-2}/16$ and $r=\phi^{-1}(l^{-1})$.   To estimate the integral above, we use the following inequality: for any $c\in \bR$ and $\varepsilon>0$,
$$
\int^a_b l^{c}dl\leq C(c) \left(a^{c+1}+b^{c+1}\right)+C(\varepsilon) 1_{c=-1} a^{\varepsilon}b^{-\varepsilon}, \qquad \forall\, 0< b<a.
$$
This is obvious if $c\neq -1$, and  if $c=-1$ then we use $\ln (a/b)\leq C(\varepsilon)(a/b)^{\varepsilon}$.  Thus we get
\begin{eqnarray*}
 \int_{\left(\phi(b^{-2}/16) \right)^{-1}}^{2s^\alpha} l^{(-d-c_1)/2} dl &\leq& Cs^{\alpha(2-d-c_1)/2}+C\left(\phi(b^{-2}/16) \right)^{(-2+d+c_1)/2} \\
 &&+1_{c_1+d=2} C(\varepsilon)s^{\alpha \varepsilon} \left(\phi(b^{-2}/16) \right)^{\varepsilon}.
  \end{eqnarray*} 
Coming back to \eqref{eqn 4.15.10} and \eqref{eqn 4.14.5}, and using the definition of $II(b)$,  we get 
\begin{align*}
&II(b)\\
 &\leq C \left(\phi(b^{-2}/16) \right)^{c_1-2} \int_{-\infty}^{-m\kappa(b)} \int_{-\kappa(b)-r}^{-r} \bM_x\left| g\right| _{H}^{2}(r,x)  s^{-2\beta} ds dr\\
&\,\,\,+C \left(\phi(b^{-2}/16) \right)^{-d}  \int_{-\infty}^{-m\kappa(b)} \int_{-\kappa(b)-r}^{-r} \bM_x\left| g\right| _{H}^{2}(r,x) s^{-2\beta+\alpha(2-d-c_1)} dsdr\\
&\,\,\,+C(\varepsilon)1_{c_1+d=2} \left(\phi(b^{-2}/16) \right)^{-d+2\varepsilon}  
\\
&\quad\quad\times \int_{-\infty}^{-m\kappa(b)} \int_{-\kappa(b)-r}^{-r} \bM_x\left| g\right| _{H}^{2}(r,x)  s^{-2\beta+2\alpha \varepsilon} dsdr\\
&=:II_1(b)+II_2(b)+II_3(b).
\end{align*}
Now we fix $\varepsilon>0$ such that $-2\beta+2\alpha \varepsilon <-1$ (recall $\beta>1/2$). Then,
$$
-2\beta<-1, \quad  -2\beta+\alpha(2-d-c_1)<-1, \quad  -2\beta+2\alpha \varepsilon <-1.
$$
Therefore, all of $II_i(b)$, $i=1,2,3$, can be handled by \eqref{f(s-r) 200413}. For instance, by  \eqref{f(s-r) 200413} with $\nu=2\beta-2\alpha \varepsilon$,
\begin{eqnarray}
\nonumber
II_3(b)&\leq& C  1_{c_1+d=2} \left(\phi(b^{-2}/16) \right)^{-d+2\varepsilon} \kappa(b)^{2-2\beta+2\alpha \varepsilon} \bM_t \bM_x\left| g\right| _{H}^{2}(t,x)  \\
&\leq& C\kappa(b) \bM_t \bM_x\left| g\right| _{H}^{2}(t,x).  \label{eqn 4.17.3}
\end{eqnarray}
One can handle $II_1(b)$ and $II_2(b)$ in the same way, and get 
$$II(b)\leq C \kappa(b) \bM_t \bM_x\left| g\right| _{H}^{2}(t,x).
$$
This together with \eqref{eqn 4.13.10} proves \eqref{eqn 4.15.9}, and the lemma is proved.
\end{proof}

\begin{lem} \label{outoutspaceestimate 200414}

Let $g\in C_{c}^{\infty}(\mathbb{R}^{d+1};H)$ have a support in $(-\infty,-2\kappa(b))\times B_{2b}^{c}$.
Then for any $(t,x)\in Q_{b}$,
$$
\aint_{Q_{b}}\aint_{Q_{b}}\left|\mathcal{T}g(s,y_1)-\mathcal{T}g(s,y_2)\right|^{2}dsdy_1d\tilde{s}dy_2\leq  C \bM_{t}\bM_x\left| g\right| _{H}^{2}(t,x),
$$
where $C=C(\alpha,\beta,d,\delta_0,\kappa_0)$.
\end{lem}

\begin{proof}
Since $g(r,z)=0$ if $|z|\leq 2b$,  by the fundamental theorem of calculus,  for any $s\in (-\kappa(b),0)$ and $y_1,y_2\in B_b$ we have
\begin{align}
&|\cT g(s,y_1)- \cT g(s,y_2)|^2  \nonumber
\\
&\leq \int_{-\infty}^s\left|\int_{\bR^d} \int_0^1 \nabla q_{\alpha,\beta}^{c_1/2}(s-r,\bar{u}-z) \cdot (y_2-y_1) g(r,z) dudz\right|_H^2dr     \nonumber
\\
&= \int_{-\infty}^s\left|\int_{\bR^d} \int_0^1 \nabla q_{\alpha,\beta}^{c_1/2}(s-r,z) \cdot (y_2-y_1) g(r,\bar{u}-z) dudz\right|_H^2dr       \nonumber
\\
& \leq 4b^2   \int_{-\infty}^{-2\kappa(b)} \left| \int_{B_b^c} \int_0^1 |\nabla q_{\alpha,\beta}^{c_1/2}(s-r,z)||  g(r,\bar{u}-z)|_H dudz\right|^{2}dr   \label{eqn 4.16.11}
\end{align}
where $\bar{u}=\bar{u}(u,y_{1},y_{2})=(1-u)y_1+uy_2$, and $|\bar{u}|\leq b$. 

 Note that for $s\in (-\kappa(b),0)$ and $r <-2\kappa(b)$, 
\begin{equation}
\label{eqn 4.17.1}
s-r> \kappa(b), \quad (\phi^{-1}((s-r)^{-\alpha}))^{-1/2}> b.
\end{equation}
 Hence, by  Lemma \ref{prop:kernel esti. of q} (iii),
\begin{align}
&  \int_{B_b^c} \int_{0}^{1} | \nabla q_{\alpha,\beta}^{c_1/2}(s-r,z) |  |g(r,\bar{u}-z)|_{H} du dz    \nonumber
\\ 
&= \int_{|z|\geq (\phi^{-1}((s-r)^{-\alpha}))^{-1/2}} \int_{0}^{1} | \nabla q_{\alpha,\beta}^{c_1/2}(s-r,z) |  |g(r,\bar{u}-z)|_{H} du dz        \nonumber
\\
&\quad + \int_{b<|z|<(\phi^{-1}((s-r)^{-\alpha}))^{-1/2}} \int_{0}^{1} | \nabla q_{\alpha,\beta}^{c_1/2}(s-r,z) |  |g(r,\bar{u}-z)|_{H} du dz    \nonumber
\\
&\leq \int_{0}^{1} \int_{|z|\geq (\phi^{-1}((s-r)^{-\alpha}))^{-1/2}}  (s-r)^{\alpha-\beta}G_{c_{1}/2,d+1}(|z|) |g(r,\bar{u}-z)|_{H} dz du       \nonumber
\\
&\quad + \int_{0}^{1}  \int_{b<|z|<(\phi^{-1}((s-r)^{-\alpha}))^{-1/2}} H_{c_{1}/2,d+1}(s-r,|z|) |g(r,\bar{u}-z)|_{H} dz  du           \nonumber
\\
&=: I(s,r,y_1,y_2)+II(s,r,y_1,y_2),             \label{eqn 4.16.12}
\end{align}
where $G_{c_{1}/2,d+1}(|z|)$ and $H_{c_{1}/2,d+1}(s-r,|z|)$ are from \eqref{G} and \eqref{H}, respectively.   

By \eqref{eqn 4.16.12}  and \eqref{eqn 4.16.11},  
\begin{eqnarray}
  \label{eqn 4.16.13}
     && \int_{Q_{b}}\int_{Q_{b}}\left|\mathcal{T}g(s,y_1)-\mathcal{T}g(s,y_2)\right|^{2}dsdy_1d\tilde{s}dy_2       \\
     & \leq& C  b^2     
       \int_{Q_{b}}\int_{Q_{b}}              \int^{-2\kappa(b)}_{-\infty} \left(I^2(s,r,y_1,y_2)+II^2(s,r,y_1,y_2) \right)dr dsdy_1d\tilde{s}dy_2.             \nonumber
        \end{eqnarray}
        
As in the proof of Lemma \ref{inwholeestimate 200316},  using \eqref{eq:integration by parts}, \eqref{eqn 09.05.17:21} and \eqref{eq:2} (recall $(\phi^{-1}((s-r)^{-\alpha}))^{-1/2}> b$), we get
\begin{align*}
&I(s,r,y_1,y_2)
\\
&\leq  C \int_{(\phi^{-1}((s-r)^{-\alpha}))^{-1/2}}^\infty  (s-r)^{\alpha-\beta}\frac{\phi(\rho^{-2})^{c_1/2}}{\rho^{d+2}} \left(\int_{B_{3\rho}(x)} |g(r,z)|_H dz\right) d\rho\\
&\leq C  \bM_x |g|_H (r,x)  \int_{(\phi^{-1}((s-r)^{-\alpha}))^{-1/2}}^\infty  (s-r)^{\alpha-\beta}\frac{\phi(\rho^{-2})^{c_1/2}}{\rho^{2}} d\rho.
\end{align*}
Hence, by \eqref{phiint}, for $r<-2\kappa(b)<s$, 
\begin{align*}
&I(s,r,y_1,y_2)
\\
&\leq  C \bM_x |g|_H (r,x) (\phi^{-1}((s-r)^{-\alpha}))^{1/2}\int_{(\phi^{-1}((s-r))^{-\alpha})^{-1/2}}^\infty  (s-r)^{\alpha-\beta}\frac{\phi(\rho^{-2})^{c_1/2}}{\rho} d\rho
\\
&\leq  C \bM_x |g|_H (r,x) (\phi^{-1}((s-r)^{-\alpha}))^{1/2} (s-r)^{\alpha(1-c_1/2)-\beta}
\\
&= C \bM_x |g|_H (r,x) (\phi^{-1}((s-r)^{-\alpha}))^{1/2} (s-r)^{-1/2}
\\
&\leq  C \bM_x |g|_H (r,x) \phi(b^{-2})^{-1/2}b^{-1} (s-r)^{-\alpha/2-1/2}.
\end{align*}
For the last inequality, we use \eqref{phiratio} with $r=\phi^{-1}((s-r)^{-\alpha})$ and $R=b^{-2}$. 
Therefore, by \eqref{f(s-r) 200413},
\begin{align}
 &b^2     
       \int_{Q_{b}}\int_{Q_{b}}              \int^{-2\kappa(b)}_{-\infty} I^2(s,r,y_1,y_2)dr dsdy_1d\tilde{s}dy_2   \nonumber\\
&\leq C b^{2+2d} \kappa(b) \int_{-\kappa(b)}^0\int_{-\infty}^{s-\kappa(b)} I^2(s,r,y_1,y_2) dr ds               \nonumber \\
&\leq  C b^{2+2d} \kappa(b)\int_{-\kappa(b)}^0\int_{-\infty}^{-2\kappa(b)} \bM_x |g|_H^2 (r,x) \phi(b^{-2})^{-1}b^{-2} (s-r)^{-\alpha-1}drds               \nonumber
\\
&\leq  C  \phi(b^{-2})^{-1}b^{2d} (\kappa(b))^{2-\alpha}\bM_t \bM_x |g|_H^2(t,x)            \nonumber
\\
&=  C  (\kappa(b))^2b^{2d}\bM_t \bM_x |g|_H^2(t,x) .              \label{eqn 4.16.15}
\end{align}
Similarly, by \eqref{eq:integration by parts}, \eqref{eqn 09.05.17:21-2} and \eqref{eq:2} (recall $(\phi^{-1}((s-r)^{-\alpha}))^{-1/2}> b$), we have
\begin{align*}
&II(s,r,y_1,y_2) \\
& \leq C \int_b^{(\phi^{-1}((s-r)^{-\alpha}))^{-\frac{1}{2}}} H_{c_1/2,d+2}(s-r,\rho) \left(\int_{B_{3\rho}(x)} |g(r,z)|_H dz\right) d\rho
\\
&  \quad + C H_{c_1/2,d+1}(s-r,(\phi^{-1}((s-r)^{-\alpha}))^{-\frac{1}{2}}) \left(\int_{B_{(\phi^{-1}((s-r)^{-\alpha}))^{-\frac{1}{2}}}} |g(r,\bar{u}-z)|_H dz\right)
\\
&=: II_1 (s,r) + II_2 (s,r,y_1,y_2).
\end{align*}
By definition of $H_{c_{1}/2,d+1}$ (see \eqref{H}), 
\begin{align*}
&II_1(s,r)       \nonumber \\
&\leq   C \bM_x |g|_H (r,x) (s-r)^{-\beta} \int_{(\kappa(b))^{\alpha}}^{2(s-r)^\alpha} \int_{b}^{(\phi^{-1}(l^{-1}))^{-\frac{1}{2}}}(\phi^{-1}(l^{-1}))^{(d+2)/2} l^{-c_1/2} \rho^d d\rho dl  \nonumber
\\
& \leq  C \bM_x |g|_H (r,x)(s-r)^{-\beta} \int_{(\kappa(b))^{\alpha}}^{2(s-r)^\alpha} (\phi^{-1}(l^{-1}))^{1/2} l^{-c_1/2} dl \nonumber
\\
& \leq C  \bM_x |g|_H (r,x)(s-r)^{-\beta} \phi(b^{-2})^{-1/2}b^{-1}\int_{(\kappa(b))^{\alpha}}^{2(s-r)^\alpha} l^{-1/2-c_1/2} dl   \nonumber \\
 \end{align*}
 For the last inequality above we used   \eqref{phiratio} with $R=b^{-2}$ and $r=\phi^{-1}(l^{-1})$.
 
 Also, by   \eqref{H} and \eqref{phiratio} with $R=\phi^{-1}(l^{-1})$ and $r=\phi^{-1}((s-r)^{-\alpha})$ (recall $(\phi^{-1}((s-r)^{-\alpha}))^{-1/2}> b$),
\begin{align*}
&II_2(s,r,y_1,y_2)
\\
&\leq C \bM_x |g|_H (r,x)\int_{(s-r)^\alpha}^{2(s-r)^\alpha} l^{-1/2-c_1/2} (s-r)^{\alpha/2-\beta} (\phi^{-1}((s-r)^{-\alpha}))^{1/2} dl.
\end{align*}
We use \eqref{phiratio}  again with $R=b^{-2}$ and $r=\phi^{-1}((s-r)^{-\alpha})$, and get
\begin{align*}
&II_2(s,r,y_1,y_2) 
\\
&\leq C \phi(b^{-2})^{-1/2}b^{-1} (s-r)^{-\beta} \bM_x |g|_H (r,x)\int_{(s-r)^\alpha}^{2(s-r)^\alpha} l^{-1/2-c_1/2}dl \\
&\leq C \phi(b^{-2})^{-1/2}b^{-1} (s-r)^{-\beta} \bM_x |g|_H (r,x)\int_{(\kappa(b))^\alpha}^{2(s-r)^\alpha} l^{-1/2-c_1/2}dl. \end{align*}
The second inequality above is due to \eqref{eqn 4.17.1}. Thus
\begin{eqnarray}
\nonumber
&&II^2\leq 2(II_1(s,r)+II_2(s,r,y_1,y_2))^2  \nonumber \\
&\leq&  C\phi(b^{-2})^{-1}b^{-2} (s-r)^{-2\beta} \bM_x |g|^2_H (r,x)   \left(\int_{(\kappa(b))^{\alpha}}^{2(s-r)^\alpha} l^{-1/2-c_1/2} dl \right)^2   \nonumber \\
&\leq& C\phi(b^{-2})^{-1}b^{-2} (s-r)^{-2\beta} \bM_x |g|^2_H (r,x)  \nonumber \label{eqn 4.16.21}\\
&& \times  \left[ (s-r)^{\alpha (1-c_1)}+(\kappa(b))^{\alpha(1-c_1)}+C(\varepsilon)1_{c_1=1}(s-r)^{2\alpha \varepsilon} (\kappa(b))^{-2\alpha \varepsilon} \right], 
\nonumber
\end{eqnarray}
where $\varepsilon>0$ is chosen so that $-2\beta+2\alpha \varepsilon<-1$.  Note that 
$$
-2\beta+\alpha(1-c_1)=-1-\alpha<-1, \quad -2\beta<-1, \quad -2\beta+2\alpha \varepsilon<-1.
$$
As is done for \eqref{eqn 4.17.3}, applying \eqref{f(s-r) 200413} three times with $\nu=1+\alpha, 2\beta$ and $2\beta-2\alpha \varepsilon$, we get
\begin{align*}
 &b^2     
       \int_{Q_{b}}\int_{Q_{b}}              \int^{-2\kappa(b)}_{-\infty} II^2(s,r,y_1,y_2)dr dsdy_1d\tilde{s}dy_2   \nonumber\\
&\leq C b^{2+2d} \kappa(b) \int_{-\kappa(b)}^0\int_{-\infty}^{-2\kappa(b)} II^2(r,s,y_1,y_2) dr ds  \nonumber   \\
&\leq C  (\kappa(b))^2b^{2d}\bM_t \bM_x |g|_H^2(t,x) .              \label{eqn 4.16.19}
\nonumber \\
\end{align*}
This, \eqref{eqn 4.16.15} and \eqref{eqn 4.16.13} prove the lemma. 
\end{proof}

\begin{lem} \label{outouttimeestimate 200414}

Let $g\in C_{c}^{\infty}(\mathbb{R}^{d+1};H)$ have a support in $(-\infty,-2\kappa(b))\times B_{2b}^{c}$.
Then for any $(t,x)\in Q_{b}$,
$$
\aint_{Q_{b}}\aint_{Q_{b}}\left|\mathcal{T}g(s_1,y)-\mathcal{T}g(s_2,y)\right|^{2}ds_1dyds_2d\tilde{y}\leq C\bM_{t}\bM_x\left| g\right| _{H}^{2}(t,x),
$$
where $C=C(\alpha,\beta,d,\delta_0,\kappa_0)$.
\end{lem}

\begin{proof}
Since $g(r,z)=0$ if $|z|\leq 2b$, for $s_1,s_2\in (-\kappa(b),0)$, and $y\in B_b$,
\begin{align*}
&|\cT g(s_1,y)- \cT g(s_2,y)| 
\\
&\leq \left[\int_{-\infty}^{-2\kappa(b)}\left|\int_{\bR^d} \int_0^1 q_{\alpha,\beta+1}^{c_1/2}(\bar{s}-r,y-z)(s_1-s_2) g(r,z) dudz\right|_H^2dr\right]^{1/2}
\\
&\leq \left[\int_{-\infty}^{-2\kappa(b)}\left|\int_{\bR^d} \int_0^1 q_{\alpha,\beta+1}^{c_1/2}(\bar{s}-r,z)(s_1-s_2) g(r,y-z) dudz\right|_H^2dr\right]^{1/2}
\\
& \leq C \kappa(b) \left[\int_{-\infty}^{-2\kappa(b)} \left| \int_{B_b^c} \int_0^1 q_{\alpha,\beta+1}^{c_1/2}(\bar{s}-r,z) g(r,y-z) dudz\right|_{H}^2dr\right]^{1/2}
\end{align*}
where $\bar{s}=\bar{s}(u,s_{1},s_{2})=(1-u)s_1+us_2$.

Note that for $B_{\rho}(y) \subset B_{2\rho}(x)$ if $\rho\geq b$, and $x,y\in B_{b}$. Therefore, by Lemma \ref{prop:kernel esti. of q} (ii), \eqref{eq:integration by parts}, \eqref{eqn 09.05.17:21} and \eqref{phiint}, for $s_1,s_2\in(-\kappa(b),0)$ and $r<-2\kappa(b)$,
\begin{align*}
&\left| \int_{B_b^c}  q_{\alpha,\beta+1}^{c_1/2}(\bar{s}-r,z) g(r,y-z) dz\right|_{H}
\\
&\leq C\int_{b}^\infty  (\bar{s}-r)^{\alpha-\beta-1}\frac{\phi(\rho^{-2})^{c_1/2}}{\rho^{d+1}} \left(\int_{B_{2\rho}(x)} |g(r,z)|_H dz\right) d\rho
\\
&\leq C \bM_x |g|_H(r,x) \int_{b}^\infty(\bar{s}-r)^{\alpha-\beta-1}\frac{\phi(\rho^{-2})^{c_1/2}}{\rho}  d\rho
\\
&\leq C \bM_x |g|_H(r,x) (\bar{s}-r)^{\alpha-\beta-1} \phi(b^{-2})^{c_1/2}.
\end{align*}
Since $-\kappa(b)<s_{1},s_{2}<0$, and $r<-2\kappa(b)$, we have $\frac{1}{2}\leq\frac{\bar{s}-r}{s-r}\leq2$ for any $s\in(-\kappa(b),0)$. Hence, by \eqref{f(s-r) 200413} (recall $2\alpha-2\beta-2<-1$),
\begin{align*}
&\int_{-\kappa(b)}^0 \int_{-\kappa(b)}^0 \int_{-\infty}^{-2\kappa(b)}  \left| \int_{B_b^c} \int_0^1 q_{\alpha,\beta+1}^{c_1/2}(\bar{s}-r,z) g(r,y-z) dudz\right|_{H}^2dr ds_1ds_2
\\
&\leq C \phi(b^{-2})^{c_1}\kappa(b)\int_{-\kappa(b)}^0 \int_{-\infty}^{-2\kappa(b)} \bM_x |g|_H^2(r,x) (s-r)^{2\alpha-2\beta-2} drds
\\
&\leq C \bM_t \bM_x |g|_H^2(t,x).
\end{align*} 
The lemma is proved.
\end{proof}

{\textbf{Proof of Theorem \ref{thm:L-P}}}

Due to  Lemma \ref{lem:L2 result},  we may  assume  $p>2$.
\\
First we prove for each $Q=Q_{b}(t_0,x_0)$ and $(t,x)\in Q$,
\begin{equation} \label{sharpmax before sup}
\aint_{Q}|\mathcal{T}g-\left(\mathcal{T}g\right)_{Q}|^{2}dsdy\leq C \bM_{t}\bM_x\left| g\right| _{H}^{2}(t,x).
\end{equation}
Note that for any $t_{0}\in\mathbb{R}$ and $x_{0}\in\mathbb{R}^{d}$,
\begin{align*}
&\mathcal{T}g(t+t_{0},x+x_{0}) 
\\
& =\left[\int_{-\infty}^{t+t_{0}}\left| \phi(\Delta)^{{c_1}/2}T_{t+t_{0}-s}^{\alpha,\beta}g(s,\cdot)(x+x_{0})\right| _{H}^{2}ds\right]^{1/2}
\\
& =\left[\int_{-\infty}^{t+t_{0}}\left| \int_{\mathbb{R}^{d}}q_{\alpha,\beta}^{c_1/2}(t+t_{0}-s,x+x_{0}-y)g(s,y)dy\right| _{H}^{2}ds\right]^{1/2}
\\
& =\left[\int_{-\infty}^{t}\left| \int_{\mathbb{R}^{d}}q_{\alpha,\beta}^{c_1/2}(t-s,x-y)\bar{g}(s,y)dy\right| _{H}^{2}ds\right]^{1/2}
\\
& =\mathcal{T}\bar{g}(t,x),
\end{align*}
where $\bar{g}(s,y):=g(s+t_{0},y+x_{0})$.  Therefore,
\begin{eqnarray*}
\aint_{Q_b(t_{0},x_{0})}|\mathcal{T}g-\left(\mathcal{T}g\right)_{Q_b(t_{0},x_{0})}|^{2}dsdy= \aint_{Q_b(0,0)}|\mathcal{T}\bar{g}-\left(\mathcal{T}\bar{g}\right)_{Q_b(0,0)}|^{2}dsdy.
\end{eqnarray*}
This implies that, to prove \eqref{sharpmax before sup}, it suffices to  consider the only case $Q=Q_{b}(0,0)$.

Now we fix $b>0$ and  take a function $\zeta\in C^{\infty}(\bR)$ such that $\zeta=1$ on $[-7\kappa(b)/3,\infty)$, $\zeta=0$ on $(-\infty,-8\kappa(b)/3)$, and $0\leq\zeta\leq1$. We also choose a function $\eta\in C_{c}^{\infty}(\mathbb{R}^{d})$ such that $\eta=1$ on $B_{7b/3}$, $\eta=0$ outside of $B_{8b/3}$,
and $0\leq\eta\leq1$. Set
$$
g_{1}(t,x)=g\zeta,\quad g_{2}=g(1-\zeta)\eta,\quad g_{3}=g(1-\zeta)(1-\eta).
$$

We   show that for any $c\in \bR$,
\begin{align}
\nonumber
\left|\mathcal{T}g(s,y)-c\right| & \leq\left|\mathcal{T}g_{1}(s,y)\right|+\left|\mathcal{T}(g_{2}+g_{3})(s,y)-c\right|\\
 & \leq\left|\mathcal{T}g_{1}(s,y)\right|+\left|\mathcal{T}g_{2}(s,y)\right|+\left|\mathcal{T}g_{3}(s,y)-c\right|.  \label{eqn 09.11.16:13}
\end{align}
Fix $c\in \bR$. If $\cT g (s,y)>c$, then due to the sublinearity of $\cT$
\begin{align*}
\cT g(s,y)-c &\leq \cT g_1 (s,y) + \cT g_2 (s,y) + \cT g_3 (s,y)-c 
\\
&\leq\left|\mathcal{T}g_{1}(s,y)\right|+\left|\mathcal{T}g_{2}(s,y)\right|+\left|\mathcal{T}g_{3}(s,y)-c\right|
\end{align*}
Suppose $\cT g(s,y)< c$. Again by the sublinearity,
\begin{gather*}
\cT g \geq -\cT g_{1}+\cT (g_{2}+g_{3}),
\end{gather*}
Therefore,
\begin{align*}
c-\cT g(s,y) &\leq \cT g_1 (s,y) + c -\cT (g_2+g_3)(s,y)
\\
&\leq \left|\mathcal{T}g_{1}(s,y)\right| + \left|c - \cT g_3 (s,y) \right|+\left|\cT g_3 (s,y) -\cT (g_2+g_3)(s,y)\right|
\\
&\leq \left|\mathcal{T}g_{1}(s,y)\right| +\left|\mathcal{T}g_{2}(s,y)\right|+\left|\mathcal{T}g_{3}(s,y)-c\right|.
\end{align*}
Thus \eqref{eqn 09.11.16:13} is proved. 

By Lemma \ref{inwholeestimate 200316} and  Lemma \ref{outinestimate 200414},  we have
\begin{align}
\aint_{Q}|\cT g_{1}(s,y)|^{2} dyds \leq C \bM_{t}\bM_{x}|g_1|^{2}_{H} \leq C \bM_{t}\bM_{x}|g|^{2}_{H}, \label{eqn 09.11.16:11}
\\
\aint_{Q} |\cT g_{2}(s,y)|^{2} dyds \leq C \bM_{t}\bM_{x}|g_2|^{2}_{H} \leq C \bM_{t}\bM_{x}|g|^{2}_{H}. \label{eqn 09.11.16:12}
\end{align}
Also using Lemma \ref{outoutspaceestimate 200414} and Lemma \ref{outouttimeestimate 200414}  we have
\begin{align}
&\aint_{Q}|\cT g_{3}(s,y)  -  ( \cT g_{3})_{Q}|^{2}dzdrdyds \nonumber
\\
& \leq C \aint_{Q}\aint_{Q}|\cT g_{3}(s,y)  -   \cT g_{3}(s,z)|^{2}dzdrdyds \nonumber
\\
&\quad + C \aint_{Q}\aint_{Q}|\cT g_{3}(s,z)  -   \cT g_{3}(r,z)|^{2}dzdrdyds \nonumber
\\
&\leq C \bM_{t}\bM_{x}|g|^{2}_{H}(t,x). \label{eqn 09.11.16:23}
\end{align}
Therefore, if we take $c=\left(\mathcal{T}g_{3}\right)_{Q}$, by \eqref{4times}, \eqref{eqn 09.11.16:11}, \eqref{eqn 09.11.16:12}, and \eqref{eqn 09.11.16:23} it follows that
\begin{align*}
&\aint_{Q}|\cT g- (\cT g)_{Q}|^{2}dyds 
\\
& \leq C \aint_{Q}|\cT g_{1}(s,y)|^{2}dyds + C \aint_{Q}|\cT g_{2}(s,y)|^{2}dyds 
\\
&\quad+ C \aint_{Q}|\cT g_{3}(s,y)-(\cT g_{3})_{Q}|^{2}dyds
\\
& \leq C \bM_{t}\bM_{x}|g|^{2}_{H}(t,x),
\end{align*}
and thus \eqref{sharpmax before sup} is proved. 
By (\ref{sharpmax before sup}) and Jensen's inequality,
\begin{equation}
 \label{ffer}
(\cT g)^{\#}(t,x)\leq C \left(\bM_{t}\bM_x\left| g\right| _{H}^{2}(t,x) \right)^{1/2}, \quad \forall \, (t,x).
\end{equation}
Therefore, by  Fefferman-Stein Theorem (e.g. \cite[Theorem IV.2.2]{stein1993harmonic}) and \eqref{ffer}, 
\begin{eqnarray*}
\|\cT g\|^p_{L_p(\bR^{d+1})}&\leq& C \|(\cT g)^{\#}\|^p_{L_p(\bR^{d+1})}
\leq C \|\bM_{t}\bM_x\left| g\right| _{H}^{2}\|^{p/2}_{L_{p/2}(\bR^{d+1})}.
\end{eqnarray*}
Next we use Hardy-Littlewood Maximal theorem (e.g. \cite[Theorem I.3.1]{stein1993harmonic})  twice with respect to time and spatial variables in order, and get
\begin{eqnarray*}
         \label{eqn 4.17.10}
   \int_{\bR^d}      \int_{\bR}  \left(\bM_t \bM_x |g|^2_H \right)^{p/2} dt dx&\leq&C\int_{\bR^d} \int_{\bR} \left(\bM_x|g|^2_H\right)^{p/2} dt dx\\
   &=&C\int_{\bR} \int_{\bR^d } \left(\bM_x|g|^2_H\right)^{p/2} dx dt \\
   &\leq&C \int_{\bR}\int_{\bR^d}( |g|^2_H)^{p/2}dxdt =C\||g|_H\|^p_{L_p(\bR^{d+1})}.
   \end{eqnarray*}
This proves the theorem if $T=\infty$.  For $T<\infty$ take $\xi \in C^{\infty}(\bR)$ such that $0\leq \xi \leq 1$, $\xi=1$ for $t\leq T$ and $\xi=0$ for $t\geq T+\varepsilon$, $\varepsilon>0$. Then it is enough to apply the result for $T=\infty$ with $g \xi$.   Since $\varepsilon>0$ is arbitrary, the theorem is proved.

\subsection{$H^\infty$-calculus} \label{Functional}

First, we provide some definitions related to $H^\infty$-calculus.  For $\eta>0$, let
$$\Sigma_\eta:=\{z\in \bC\setminus\{0\}: |\arg(z)|<\eta\}.
$$
 By $H^p(\Sigma_\eta)$ ($1\leq p \leq \infty$), we denote the (complex) Banach space of all holomorphic functions $f:\Sigma_\eta\to \bC$ satisfying
$$
\|f\|_{H^p(\Sigma_\eta)}:=\sup_{|\nu|<\eta} \|f(e^{i\nu}t)\|_{L_p(\bR_+,\frac{dt}{t})}<\infty.
$$
Let $A$ be a linear operator on a Banach space $X$. We say that $z$ is in the resolvent set $\rho(A)$ of $A$ if the range of $A_z:=z-A$ is dense in $X$ and $A_z$ has a continuous inverse.
Here, for $z\in \rho(A)$, we can define $R(z,A):=(z-A)^{-1}$. We say that a linear operator $A$ is sectorial if there exists $\omega\in(0,\pi)$ such that the spectrum $\sigma(A):=\bC\setminus \rho(A)$ is contained in $\overline{\Sigma_\omega}$ and
$$
\sup_{z\in (\overline{\Sigma_\omega})^c} \| z R(z,A)\|<\infty.
$$
In this case we say $A$ is $\omega$-sectorial.
The infimum of all $\omega$ such that $A$ is $\omega$-sectorial is called the angle of sectoriality of $A$ and is denoted by $\omega(A)$.

Let $A$ be a sectorial operator with angle of sectoriality $\omega(A)$. For functions $f\in H^1(\Sigma_\omega)$,  denote
$$ 
f(A):=\frac{1}{2\pi i}\int_{\partial \Sigma_\nu} f(z)R(z,A)dz
$$
where $\omega(A)<\nu<\sigma$ is chosen arbitraraily. It is well known (see \cite[Section 10.2]{Veraar book2}) that the definition of $f(A)$ is independent of the choice of $\nu$.
For a constant $\sigma\in  (\omega(A), \pi)$, we say that the operator $A$  has a bounded $H^\infty(\Sigma_\sigma)$-calculus if there exists a constant $C>0$ such that
$$
\|f(A)\|\leq C\|f\|_{L_\infty}, \quad f\in H^1(\Sigma_\sigma)\cap H^\infty(\Sigma_\sigma).
$$
We define 
$$
\omega_{H^\infty}(A):=\inf\{ \sigma\in (\omega(A), \pi)\,:\, A \text{ has a bounded } H^\infty(\Sigma_\sigma)\text{-calculus} \},
$$
and we say that $A$ has a bounded $H^{\infty}$-calculus of angle $\omega_{H^\infty}(A)$. For instance, the Laplace operator $-\Delta$ has a bounded $H^\infty$-calculus on $L_p(\bR^d)$ of angle $0$ (see e.g. \cite[Theorem 10.2.25]{Veraar book2}).

Now we are ready to prove the following:
\begin{thm} \label{bounded H}
Let $\phi$ be a Bernstein function. Then, $-\phi(\Delta)=\phi(-\Delta)$ has a bounded $H^\infty$-calculus on $L_p(\bR^d)$ of angle $0$.
\end{thm}

\begin{proof}

Note that $\phi$ can be extended to a holomorphic function which maps $\bC_{+}:=\{z\in \bC : Re(z)>0\}$ into itself and satisfies $\phi(\Sigma_\sigma)\subset \Sigma_\sigma$ for any $\sigma\in[0,\pi)$ (see \cite[Proposition 3.3]{Resolvent}). Hence, \cite[Theorem 1.1]{Resolvent} yields that $-\phi(\Delta)=\phi(-\Delta)$ is $0$-sectorial.

Next, for $f\in H^\infty(\Sigma_\sigma)$ with $0<\sigma<\pi$, define
\begin{gather*}
\Psi_{\Delta}(f)\zeta  := \cF^{-1} \left[ f(|\cdot|^2)\hat{\zeta}(\cdot) \right], \quad \zeta \in \cS(\bR^d).
\end{gather*}
For any multi indices $\alpha\in \{0,1\}^d$ and $\xi\in \Sigma_\sigma$, one can easily show that
$$
|\xi|^{|\alpha|} D^\alpha( f(|\xi|^2)) = |\xi|^{|\alpha|} 2^{|\alpha|} \xi^\alpha (D^\alpha f)(|\xi|^2)= (2\xi/|\xi|)^\alpha (|\xi|^2)^{|\alpha|} (D^\alpha f)(|\xi|^2).
$$
Also, by the Cauchy formula,
$$
(|\xi|^2)^{|\alpha|} (D^\alpha f)(|\xi|^2) \leq \frac{1}{2\pi} \int_{\Gamma_{\sigma'}} \frac{|\xi|^{2|\alpha|} |f(z)|}{|z-|\xi|^2|^{|\alpha|+1}} |dz| \leq \frac{C}{2\pi} \int_{\Gamma_{\sigma'}} \frac{1}{|z-1|^{|\alpha|+1}}|dz|<\infty,
$$
where $\Gamma_{\sigma'}:=\{z\in\bC\setminus\{0\} : arg(z)=\pm \sigma' \}$ with $\sigma'\in(0,\sigma)$. Thus, by Mihlin's multiplier theorem (see e.g. \cite[Theorem 5.5.10]{Veraar book1}), $\Psi_\Delta(f)$ is a bounded operator on $L_p$. Since  $f\circ \phi \in H^\infty(\Sigma_\sigma)$, by considering $f\circ \phi$ instead of $f$, we find that the operator
$$
\Psi(f)\zeta:= \cF^{-1} \left[ f(\phi(|\cdot|^2))\hat{\zeta}(\cdot) \right], \quad \zeta \in \cS(\bR^d),
$$
is a bounded operator on $L_p$. Moreover, by following the proof of \cite[Theorem 10.2.25]{Veraar book2}, one can check that the mapping $f\to \Psi(f)$ satisfies the assumptions of \cite[Theorem 10.2.14]{Veraar book2}. Thus, by \cite[Theorem 10.2.14]{Veraar book2}, $-\phi(\Delta)$ has a  bounded $H^\infty(\Sigma_\sigma)$-calculus. Since $\sigma>0$ is arbitrary, the theorem is proved.

\end{proof}

Now we define an operator associated with $(\alpha,\beta,-\phi(\Delta))$.
Let
\begin{align*}
T_{\alpha,\beta}(t)v:= \frac{1}{2\pi i}\int_{\Gamma_{1,\psi}} e^{\lambda t}(\lambda^\alpha-\phi(\Delta))^{-1}\lambda^{\beta-1}v d\lambda, \quad t>0,
\end{align*}
for $v\in L_p$, $\psi\in(\frac{\pi}{2},\pi)$, and
$$
\Gamma_{1,\psi}:=\{e^{it} : |t|\leq \psi \} \cup \{\rho e^{i\psi} : 1<\rho<\infty \} \cup \{\rho e^{-i\psi} : 1<\rho<\infty \}.
$$

We have the following representation for solution (cf. \cite{desch2011p}).

\begin{lem} \label{lem:solution func representation}
 For  given $g\in \bH_0^\infty(T,l_2)$, the function
\begin{align} \label{Func sto sol re}
u(t,x):=\sum_{k=1}^\infty \int_0^t T_{\alpha, \beta}(t-s) g^k(s,x) dw_s^k.
\end{align}
is in $\bH_p^{\phi,2}(T)$ and satisfies \eqref{eqn 4.10} in the sense of distributions.
\end{lem}

\begin{proof}
We first show that $u\in \bH_p^{\phi,2}(T)$.
By the Burkholder-Davis-Gundy inequality,  
\begin{align*}
\|u\|^p_{\bL_p(T)} &\leq C(p) \bE \Big\|  \left( \int_{0}^{t} |T_{\alpha, \beta}(t-s)g(s)(x) |^{2}_{l_2}ds \right)^{1/2} \Big\|^p_{L_p((0,T)\times \bR^d)} \nonumber
\\
&\leq C \|g\|^p_{\bL_p(T, l_2)}.
\end{align*}
For the second inequality above we used  \cite[Lemma 5.6]{desch2011p}.

Applying (5.16) in \cite{desch2011p}, we get
\begin{align} \label{ineq 0310-1}
\phi(\Delta) u(t,x)=\sum_{k=1}^\infty \int_0^t T_{\alpha, \beta}(t-s) \phi(\Delta)g^k(s,x) dw_s^k, \quad \forall \, t>0.
\end{align}
Hence, using  \cite[Lemma 5.6]{desch2011p} again,  we have (recall $g\in \bH_0^\infty(T,l_2)$)
\begin{align*}
\|\phi(\Delta)u\|_{\bL_p(T)} \leq C \|\phi(\Delta)g\|_{\bL_p(T, l_2)} <\infty.
\end{align*}
Therefore, by Lemma \ref{H_p^phi,gamma space} $(iv)$, we get $u\in \bH_p^{\phi,2}(T)$.

Next we show that $u$ satisfies \eqref{eqn 4.10}. By (5.17) in \cite{desch2011p}, 
\begin{align} \label{ineq 0310-2}
u(t,x) &= \frac{1}{\Gamma(\alpha)} \int_0^t (t-s)^{\alpha-1} \phi(\Delta) u(s,x)ds \nonumber
\\
&\quad + \frac{1}{\Gamma(\alpha-\beta+1)} \sum_{k=1}^\infty \int_0^t (t-s)^{\alpha-\beta} g^k(s,x)dw_s^k
\end{align}
for all $t>0$, a.e. on $\bR^d\times \Omega$.
Let $\varphi\in\cS(\bR^d)$. Multiplying both sides of \eqref{ineq 0310-2} by $\varphi$ and applying (stochastic) Fubini's theorem (see e.g. \cite[Lemma 2.7]{Krylov Fubini}), one can easily show that $u$ satisfies \eqref{eqn 4.10} in the sense of distributions. The lemma is prvoed.
\end{proof}

\begin{thm} \label{thm:L-P func}
Let $p\in[2,\infty)$, $g\in \bH_0^\infty(T,l_2)$, and let $u(t,x)$ be taken from Lemma \ref{lem:solution func representation}. Then we have
\begin{equation} \label{Func esti}
\|\phi^{c_{1}/2}(\Delta)u\|_{\mathbb{L}_{p}(T)}\leq C\| g\|_{\mathbb{L}_{p}(T,l_{2})},
\end{equation}
where $c_1:=2-(2\beta-1)/{\alpha}$ and $C$ is a constant independent of $g$ and $T$.
\end{thm}

\begin{proof}
Due to Theorem \ref{bounded H}, the operator $\phi(\Delta)$ satisfies the assumption in \cite[Theorem 3.1]{desch2013maximal}. Thus, \eqref{Func esti} is a direct consequence of \cite[Theorem 3.1]{desch2013maximal} if we take $(\alpha,1+\alpha-\beta,0,c_1/2)$ in place of $(\alpha,\beta,\eta,\theta)$ therein.
\end{proof}

\begin{rem}
(i) Actually, Theorem \ref{bounded H} and \cite[Theorem 3.1]{desch2013maximal} together yield $L_p$-estimates for $D_t^\eta \phi(\Delta)^{\theta} u$, where 
$\eta\in(-1,1)\cap (-\beta+1/2,\alpha-\beta+1/2)$ and  $\theta:=c_1/2-\eta/\alpha$. This also can be obtained using Krylov's analytic approach if one considers $q_{\alpha,\beta+\eta}^{\theta}=q_{\alpha,\beta+\eta}^{c_1/2-\eta/\alpha}$ in place of $q_{\alpha,\beta}^{c_1/2}$.

(ii) As in  \cite[Section 8G]{Veraar Singular operator}, one can also obtain sharp estimates of $D_t^\eta \phi(\Delta)^{\theta} u$ in the  space $L_q(L_p)$ given with appropriate weights. Here $p\geq 2, q>2$. \end{rem}

\mysection{Proof of Theorem \ref{thm:main results}}
         \label{sec proof}

The following lemma  is used to  estimate solutions of SPDEs when  $\beta <1/2$.
\begin{lem} \label{lem:pde approach}
Let $\gamma\in\mathbb{R}$, $p \geq 2$, $\beta<\frac{1}{2}$, and $g\in\mathbb{H}_{p}^{\phi,\gamma}(T,l_{2})$. Then
for any  $t \in [0,T]$,
$$
\mathbb{E}\int_{0}^{t}\left\Vert\sum_{k=1}^{\infty} \partial_{r}^{\beta}\int_{0}^{\cdot}g^{k}(s,\cdot)dw_{s}^{k}\right\Vert _{H_{p}^{\phi,\gamma}}^{p}dr
\leq C(d,p,\beta,T) I_{t}^{1-2\beta}\|g\|_{\mathbb{H}_{p}^{\phi,\gamma}(\cdot,l_{2})}^{p}(t).
$$
In particular,
$$
\mathbb{E} \int_{0}^{r}\left\Vert\sum_{k=1}^{\infty} \partial_{t}^{\beta}\int_{0}^{\cdot}g^{k}(s,\cdot)dw_{s}^{k}\right\Vert _{H_{p}^{\phi,\gamma}}^{p}dr
\leq C \|g\|_{\mathbb{H}_{p}^{\phi,\gamma}(t,l_{2})}^{p}.
$$
\end{lem}

\begin{proof}
Due to the isometry $(1-\phi(\Delta))^{\gamma/2}:H^{\phi,\gamma}_p \to L_p$, it is enough to prove the case $\gamma=0$. In this case  it is a  consequence of \cite[Lemma 4.1]{kim16timefractionalspde}. The lemma is proved.
\end{proof}

Recall
$$
c_0=\frac{(2\beta-1)^+}{\alpha} +\kappa 1_{\beta=1/2} \quad \in [0,2),
$$
where $\kappa>0$ is a fixed constant.

\begin{lem} \label{thm:model eqn}
Let $\gamma\in\mathbb{R}, p\geq 2$, $\alpha\in (0,1)$ and $\beta<\alpha+1/2$. Then for any  $g\in\mathbb{H}_{p}^{\phi,\gamma+c_{0}}(T,l_{2})$,  the linear  equation 
\begin{equation}
 \label{eqn sto}
\partial^{\alpha}_tu=\phi(\Delta)u+\sum_{k=1}^{\infty} \partial^{\beta}_t \int^t_0g^k dw^k_s, \quad t>0; \quad u(0,\cdot)=0
\end{equation}
 has a unique solution $u\in\bH_{p}^{\phi, \gamma+2}(T)$ in the sense of distributions, and for this solution we have
\begin{equation} \label{eq:model a priori}
\|u\|_{\bH_{p}^{\phi,\gamma+2}(T)}\leq C\|g\|_{\mathbb{H}_{p}^{\phi,\gamma+c_{0}}(T,l_{2})},
\end{equation}
where $C=C(\alpha,\beta,d,p,\delta_0,\kappa_0, \gamma,\kappa,T)$.
Furthermore, if $\beta > 1/2$ then
\begin{equation} \label{model t ind}
\|\phi(\Delta)u\|_{\mathbb{H}_{p}^{\phi,\gamma}(T)}\leq C\| \phi(\Delta)^{c_0/2}g\|_{\mathbb{H}_{p}^{\phi,\gamma}(T,l_{2})},
\end{equation}
where $C=C(\alpha,\beta,d,p,\delta_0,\kappa_0, \gamma)$ is independent of $T$.
\end{lem}

\begin{proof}
Due to Remark \ref{remark 4.7.3} and  Lemma \ref{H_p^phi,gamma space} (ii), without loss of generality we may assume $\gamma=0$.

The uniqueness is a consequence of the corresponding result for the deterministic equation, \cite[Theorem 8.7 (a)]{Pruss}. Indeed, if $u$ is a solution to the equation with $g=0$, then for each fixed $\omega$, the  function $u(\omega,\cdot,\cdot)$ satisfies the deterministic equation 
\begin{equation}
   \label{deterministic}
\partial^{\alpha}_tv=\phi(\Delta)v, \quad t>0, x\in \bR^d; \quad v(0,\cdot)=0,
\end{equation}
and we conclude $u(\omega,\cdot,\cdot)\equiv 0$. Therefore,  it is sufficient to prove  the
existence result together with estimates \eqref{eq:model a priori}
and \eqref{model t ind}. 

{\bf{Step 1}}. First, assume $g\in\mathbb{H}_{0}^{\infty}(T,l_{2})$.
 Define $u$ by  \eqref{Func sto sol re}. 
Then by Lemma \ref{lem:solution func representation},
$u\in \bH^{\phi,2}_{p}(T)$  becomes a solution to  equation \eqref{eqn sto}.
Now we  prove the estimates. We divide the proof according to the range of $\beta$.

\textit{Case 1}. $\beta > \frac{1}{2}$.
\\
We first show \eqref{model t ind}. Denote
$$
v=\phi(\Delta)^{c_0 /2}u, \quad \bar{g}=\phi(\Delta)^{c_0 /2}g .
$$
By \eqref{ineq 0310-1} and Theorem \ref{thm:L-P func},
\begin{align*}
\|\phi(\Delta) u\|^p_{\bL_p(T)}=\left\Vert \phi(\Delta)^{c_1/2}v\right\Vert _{\mathbb{L}_{p}(T)}^{p}
 & \leq C\mathbb{E}\int_{0}^{T}\int_{\mathbb{R}^{d}}\left|\bar{g}\right|^{p}_{l_2}dxdt\\
 & =C\|\phi(\Delta)^{c_0/2} g\|^p_{\bL_{p}(T,l_{2})}.
\end{align*}
Thus \eqref{model t ind} is proved. 
Furthermore, by \eqref{eq:solution space estimate 1},
\begin{align*}
\|u\|^p _{\bL_{p}(T)}
&\leq C \int_{0}^{T}(T-s)^{\theta-1} \left(\|\phi(\Delta) u\|^p_{\bL_p(s)}
+\|g\|^p_{\bL_p(s,l_2)} \right)ds
\\
&\leq C\int_{0}^{T}(T-s)^{\theta-1} \|g\|^p_{\bH^{\phi,c_0}_p(s,l_2)} ds
\\
&\leq C\|g\|^p_{\bH_p^{\phi,c_0}(T,l_2)}\int^T_0(T-s)^{\theta-1}ds \leq C\|g\|^p_{\bH_p^{\phi,c_0}(T,l_2)}.
\end{align*}
Thus we have \eqref{eq:model a priori}.

\textit{Case 2}. $\beta<\frac{1}{2}$.
\\
In this case,  ${c_0}=0$. By Lemma \ref{fracint of dw} (iii),   $u$ satisfies
$$
\partial_{t}^{\alpha}u=\phi(\Delta) u+\bar{f},
$$
where
$$
\bar{f}(t)=\frac{1}{\Gamma(1-\beta)}\sum_{k=1}^{\infty} \int^t_0(t-s)^{-\beta}g^k(s)dw^k_s.
$$
Due to \cite[Theorem 8.7 (a)]{Pruss} and Lemma \ref{lem:pde approach},
$$
\|u\|^p_{\bH^{\phi,2}_p(T)}\leq C\|\bar{f}\|^p_{\bL_p(T)}\leq C \|g\|^p_{\bL_p(T,l_2)}.
$$

\textit{Case 3:} $\beta=\frac{1}{2}$.
\\
Put $\delta=\frac{\kappa \alpha}{2}$ and  $\tilde{\beta}=\frac{1}{2}+\delta$. Then, $0<\delta<\alpha$
and $\frac{1}{2}<\tilde{\beta}<2$. Define $v$  by  \eqref{Func sto sol re} with $\tilde{\beta}$ instead of $\beta$.
By  the result from Case 1 with ${c_0}=(2 \tilde \beta -1)/\alpha =\kappa $, 
 $v\in \bH^{\phi,2}_{p}(T)$  satisfies
\begin{align*}
\partial_{t}^{\alpha}v=\phi(\Delta) v+\sum_{k=1}^{\infty}\partial_{t}^{\tilde{\beta}}\int_{0}^{t}g^{k} dw_{s}^{k}, \quad ; \quad v(0,\cdot)=0,
\end{align*}
and it also holds that 
$$
\left\Vert v\right\Vert _{\mathbb{H}_{p}^{\phi,2}(T)}\leq C\left\Vert g\right\Vert _{\mathbb{H}_{p}^{\phi,c_{0}}(T,l_{2})}.
$$
Note that  $I^{\delta}_tv$ also satisfies \eqref{eqn sto}. Thus, by  the uniqueness of solution, we conclude that $I_{t}^{\delta}v=u$.
Therefore, by the result for the case $\beta>1/2$ and \eqref{eq:Lp continuity of I},
\begin{align*}
\left\Vert u\right\Vert _{\mathbb{H}_{p}^{\phi,2}(T)} & =\|I_{t}^{\delta}v\|_{\mathbb{H}_{p}^{\phi,2}(T)} \leq C \left\Vert v\right\Vert _{\mathbb{H}_{p}^{\phi,2}(T)}\leq C\left\Vert g\right\Vert _{\mathbb{H}_{p}^{\phi,c_{0}}(T,l_{2}).}
\end{align*}
Thus, the lemma is proved if $g\in\mathbb{H}_{0}^{\infty}(T,l_{2})$.

{\bf{Step 2}}. General case. For given 
 $g\in\mathbb{H}_{p}^{{\phi,c_0}}(T,l_{2})$, we take a sequence  $g_{n}\in\mathbb{H}_{0}^{\infty}(T,l_{2})$ so that
$g_{n}\rightarrow g$ in $\mathbb{H}_{p}^{\phi,{c_0}}(T,l_{2})$.
Define $u_{n}$ using \eqref{Func sto sol re} with $g_{n}$ in place of $g$.
Then
\begin{equation*}
\|u_{n}\|_{\mathbb{H}_{p}^{\phi,2}(T)}\leq C\|g_{n}\|_{\mathbb{H}_{p}^{\phi,{c_0}}(T,l_{2})},\label{eq:model thm proof 1}
\end{equation*}
\begin{equation*}
 \|u_{n}-u_{m}\|_{\mathbb{H}_{p}^{\phi,2}(T)}\leq C \|g_{n}-g_{m}\|_{\mathbb{H}_{p}^{\phi,{c_0}}(T,l_{2})}.
\end{equation*}
Thus, $u_n$ converges to a function  $u\in \bH^{\phi,2}_{p}(T)$.  Considering \eqref{eq:solution space 2} corresponding to $u_n$ and using Lemma \ref{fracint of dw} (iv), we conclude that $u$ satisfies equation \eqref{eqn sto} in the sense of distributions. The estimates of $u$ also easily follow. 
The lemma is proved.
\end{proof}

\begin{rem}
The uniqueness result in Lemma \ref{thm:model eqn} yields that two representations \eqref{sto sol re} and \eqref{Func sto sol re} coincide under Assumption \ref{ass bernstein}.
\end{rem}

Now we are ready to   prove Theorem \ref{thm:main results}.

\vspace{3mm}

{\textbf{Proof of Theorem \ref{thm:main results}}}

\textbf{Step 1 (linear equation)}. Let  $f$ and $g$ be independent of $u$.

As before, due to Remark \ref{remark 4.7.3} and  Lemma \ref{H_p^phi,gamma space} (ii), we may assume $\gamma=0$. Also, by the uniqueness of deterministic equation \eqref{deterministic}, we only need to prove the existence result and estimates of the solution.

\textit{Case 1:} Let  $g\equiv 0$. Then,  roughly speaking, using \cite[Theorem 2.8]{kim2020nonlocal}, for each fixed $\omega$ one can solve the deterministic equation
\begin{equation}
\label{eqn 4.20.1}
\partial^{\alpha}_t u^{\omega}(t,x)=\phi(\Delta)u^{\omega}(t,x)+f(\omega,t,x), \quad t>0, x\in \bR^d;\quad u^{\omega}(0,x)=0,
\end{equation}
and can define $u(\omega,t,x)=u^{\omega}(t,x)$ so that $u$ solves the equation
\begin{equation*}
\partial^{\alpha}_t u=\phi(\Delta)u+f, \quad t>0;\quad u(0,x)=0
\end{equation*}
 on $\Omega \times [0,T]$.
 However, this method may leave the measurability issue. Therefore, we argue as follows. First, assume $f$ is sufficiently smooth, that is, let $f$ be of the type
$$
f(t,x)=\sum_{i=1}^{n(k)}1_{(\tau_{i-1},\tau_{i}]}(t)h^i (x), \quad h^{i}\in C_c^\infty(\mathbb{R}^{d})
$$
where $0\leq \tau_0 \leq \tau_1 \leq \cdots \leq \tau_{n}$ are bounded  stopping times.  Define
$$
u_1(t,x)=\int^t_0 \int_0^\infty P_r^0 f(s,\cdot) \varphi_{\alpha,1}(t,r)  drds,
$$
where $\varphi_{\alpha,1}$ is defined in \eqref{varphi} and  $\{P_t^0; t\geq0\}$ is the  strongly continuous semigroup generated by $\phi(\Delta)$ (cf. \cite{Chen Poisson}). Due to \eqref{eqn 200824 1516}, the above integral is well defined. Thus, we have the desired measurability of $u_1$. Moreover, \cite[Proposition 2.3]{Chen Poisson} implies that $u_1$ becomes a solution to \eqref{eqn 4.20.1} in the sense of distributions and $u_1 \in L_\infty(\Omega\times[0,T], \cP;H_p^{\phi,2}) \subset \bH_p^{\phi,2}(T)$. Also, estimate \eqref{eq: a priori estimate non-div} for this $u_1$ follows from \cite[Theorem 8.7 (a)]{Pruss}.

For general $f$, one can use the standard approximation argument as in the proof of Lemma \ref{thm:model eqn} (see Step 2 there).  

\textit{Case 2:} Let $g\not\equiv 0$.  Take $u_1\in \bH^{\phi,2}_p(T)$ from Case 1. Also take $u\in \bH^{\phi,2}_p(T)$ from Lemma \ref{thm:model eqn}. Then, thanks to the linearity of the equations, $v:=u_1+u$ satisfies
$$
\partial^{\alpha}_tv=\phi(\Delta)v+f+\sum_{k=1}^{\infty}\partial^{\beta}_t \int^t_0 g^k dw^k_s, \quad t>0; \quad v(0,\cdot)=0,
$$
and estimate \eqref{eq: a priori estimate non-div} for $v$ follows from those for $u_1$ and $u$. Therefore, the theorem is proved for the linear equation.

\textbf{Step 2 (non-linear equation)}.

We first prove the uniqueness result of the equation.  
\begin{equation*}
\partial^{\alpha}_tu=\phi(\Delta)u+f(u)+\sum_{k=1}^{\infty}g^k(u)dw^k_t, \quad t>0; \quad u(0,\cdot)=0.
\end{equation*}
Let  $u_1, u_2 \in \bH^{\phi,\gamma+2}_{p}(T)$ be two solutions to the equation. 
Then, $\tilde{u}:=u_1-u_2$ satisfies
$$
\partial^{\alpha}_t \tilde{u}=\phi(\Delta)\tilde{u}+f(u_1)-f(u_2)+\sum_{k=1}^{\infty}(g^k(u_1)-g^k(u_2))dw^k_t, \quad t>0; \quad \tilde{u}(0,\cdot)=0.
$$
By the continuity of $f$ and $g$ (Assumption \ref{asm 11.15.13:16} with $\varepsilon=1$), for each $t\leq T$,
\begin{eqnarray}
\nonumber
&&\|\phi(\Delta)(u_1-u_2)\|_{\bH^{\phi,\gamma}_p(t)} +\|f(u_1)-f(u_2)\|_{\bH^{\phi,\gamma}_p(t)}\\
&&+ \|g(u_1)-g(u_2)\|_{\mathbb{H}^{\phi,\gamma+c_{0}}_{p}(t,l_{2})} 
\leq C\|u_1-u_2\|_{\bH^{\phi,\gamma+2}_p(t)}.   \label{eqn 4.20.6}
\end{eqnarray}
 Also, by the result for the linear case and Assumption \ref{asm 11.15.13:16},  any $\varepsilon>0$ and $t\leq T$, we have
\begin{eqnarray*}
&& \|u_1- u_2\|^{p}_{\bH^{\phi,\gamma+2}_{p}(t)}   \\ 
&&\leq
 C ( \|f(u_1)-f(u_2)\|^{p}_{\mathbb{H}^{\phi,\gamma}_{p}(t)} + \|g(u_1)-g(u_2)\|^{p}_{\mathbb{H}^{\phi,\gamma+c_{0}}_{p}(t,l_{2})})
\\
&& \leq C \varepsilon^{p}  \|u_1-u_2\|^{p}_{\mathbb{H}^{\phi,\gamma+2}_{p}(t)}+CN_{0}(\varepsilon)\|u_1-u_2\|^{p}_{\mathbb{H}^{\phi,\gamma}_{p}(t)}
\\
&&\leq  C \varepsilon^{p}\|u_1-u_2\|^{p}_{\bH^{\phi,\gamma+2}_{p}(t)} +CN_{0}(\varepsilon)\int_{0}^{t}(t-s)^{\theta-1}\|u_1-u_2\|^{p}_{\bH^{\phi,\gamma+2}_{p}(s)}ds.
\end{eqnarray*}
For the last inequality above we used  \eqref{eq:solution space estimate 1} and \eqref{eqn 4.20.6}. Now we take $\varepsilon>0$ so that $C\varepsilon^p<1/2$, and we  conclude $u_1=u_2$ due to the fractional Gronwall lemma (see \cite[Corollary 1]{YGD}). The uniqueness is proved.

 Next we prove the existence result.   Let $u^0\in \bH^{\phi,\gamma+2}_p(T)$ denote the solution obtained in Step 1 corresponding to the inhomogeneous terms $f(0)$ and $g(0)$. For $n\geq 0$, using  the result of Step 1,   we   define $u^{n+1}\in \bH^{\phi,\gamma+2}_{p}(T)$  as the solution to  the equation
\begin{equation}
\label{eqn 4.21.5}
\partial^{\alpha}_{t}u^{n+1}=\phi(\Delta)u^{n+1}+f(u^n)+\partial^{\beta}_{t}\sum_{k=1}^{\infty}\int_{0}^{t}g^{k}(u^n)dw^{k}_{s}, \quad t>0;
\quad u(0)=0.
\end{equation}
Then $\tilde{u}^n:=u^{n+1}-u^n$ satisfies
$$
\partial^{\alpha}_t \tilde{u}^n=\phi(\Delta)\tilde{u}^n +f(u^n)-f(u^{n-1})+\sum_{k=1}^{\infty}(g^k(u^n)-g^k(u^{n-1})) dw^k_t, \quad t>0,
$$
with $\tilde{u}^n(0,\cdot)=0$.
By Step 1 and Assumption \ref{asm 11.15.13:16}, for any $\varepsilon>0$ and $t\leq T$, we have
\begin{eqnarray}\label{eqn 10.22.13:14}
 && \|u^{n+1}-u^{n}\|^{p}_{\bH^{\phi,\gamma+2}_{p}(t)}   \nonumber \\
  &&\leq C \left( \|f(u^{n})-f(u^{n-1})\|^{p}_{\mathbb{H}^{\phi,\gamma}_{p}(t)} + \|g(u^{n})-g(u^{n-1})\|^{p}_{\mathbb{H}^{\phi,\gamma+c_{0}}_{p}(t,l_{2})}  \right) \nonumber
\\
&&\leq C\varepsilon^{p}  \|u^{n}-u^{n-1}\|^{p}_{\mathbb{H}^{\phi,\gamma+2}_{p}(t)}+CN_{0}(\varepsilon)\|u^{n}-u^{n-1}\|^{p}_{\mathbb{H}^{\phi,\gamma}_{p}(t)}. \label{eqn 4.21.11}
\end{eqnarray}
In particular, taking $\varepsilon=1$, for any $n\geq 1$, 
\begin{equation}
   \label{eqn 4.21.22}
 \|u^{n+1}-u^{n}\|^{p}_{\bH^{\phi,\gamma+2}_{p}(t)} \leq C \|u^{n}-u^{n-1}\|^{p}_{\bH^{\phi,\gamma+2}_{p}(t)}.
 \end{equation}
Note that $(u^n-u^{n-1})(0,\cdot)=0$. Thus,  by \eqref{eq:solution space estimate 1} and \eqref{eqn 4.21.22},
\begin{eqnarray}
&&\|u^{n}-u^{n-1}\|^{p}_{\mathbb{H}^{\phi,\gamma}_{p}(t)}  \nonumber \\
&& \leq  C \int^t_0 (t-s)^{\theta-1} \Big(\|\phi(\Delta)(u^{n}-u^{n-1})+f(u^{n-1})-f(u^{n-2})\|^p_{\mathbb{H}^{\phi,\gamma}_{p}(s)}   \nonumber \\
&&\qquad  \qquad+\|g(u^{n-1})-g(u^{n-2})\|^p_{\mathbb{H}^{\phi,\gamma}_{p}(s,l_2)} \Big)ds   \nonumber \\
&&\leq C  \int^t_0 (t-s)^{\theta-1}\|u^{n-1}-u^{n-2}\|^p_{\bH^{\phi,\gamma+2}_p(s)} ds.    \label{eqn 4.21.21}
\end{eqnarray}
 Plugging \eqref{eqn 4.21.22} and \eqref{eqn 4.21.21} into  \eqref{eqn 4.21.11}, we get
\begin{eqnarray*}
 \|u^{n+1}-u^{n}\|^{p}_{\bH^{\phi,\gamma+2}_{p}(t)} &\leq&  C \varepsilon^{p}\|u^{n-1}-u^{n-2}\|^{p}_{\bH^{\phi,\gamma+2}_{p}(t)} \\
 &&+CN_{0}\int_{0}^{t}(t-s)^{\theta-1}\|u^{n-1}-u^{n-2}\|^{p}_{\bH^{\phi,\gamma+2}_{p}(s)}ds,  
\end{eqnarray*}
where  $N_{0}=N_0(\varepsilon)$.   Considering $\varepsilon C^{-1/p}$ in place of $\varepsilon$, and repeating the above argument one more time, we get for $t\leq T$,
\begin{eqnarray*}
 && \|u^{n+1}-u^{n}\|^{p}_{\bH^{\phi,\gamma+2}_{p}(t)}   \\
  &&\leq   \varepsilon^{p}\Big(  \varepsilon^{p} \|u^{n-3}-u^{n-4}\|^{p}_{\bH^{\phi,\gamma+2}_{p}(t)} +N_{1}\int_{0}^{t}(t-s)^{\theta-1}\|u^{n-3}-u^{n-4}\|^{p}_{\bH^{\phi,\gamma+2}_{p}(s)}ds \Big)
  \\
  &&+N_{1}\int_{0}^{t}(t-s)^{\theta-1}\Big(   \varepsilon^{p}\|u^{n-3}-u^{n-4}\|^{p}_{\bH^{\phi,\gamma+2}_{p}(s)}\\
  &&\qquad \qquad \qquad \qquad \qquad +N_1 \int_{0}^{r}(s-r)^{\theta-1}\|u^{n-3}-u^{n-4}\|^{p}_{\bH^{\phi,\gamma+2}_{p}(r)}dr \Big)ds.
\end{eqnarray*}
Therefore, by using the identity
\begin{align*}
&\frac{\Gamma(\theta)^{n}}{\Gamma(n\theta+1)}t^{n\theta}
\\
&=\int_{0}^{t}(t-s_{1})^{\theta-1}\int_{0}^{s_{1}}(s_{1}-s_{2})^{\theta-1}\cdots \int_{0}^{s_{n-1}}(s_{n-1}-s_{n})^{\theta-1}ds_{n}\dots ds_{1}
\end{align*}
and repeating above inequality,  for $n\in \bN_0$ 
we get
\begin{equation*}\label{eqn 10.22.16:01}
\begin{aligned}
   &\|u^{2n+1}-u^{2n}\|^{p}_{\bH^{\phi,\gamma+2}_{p}(T)}   
   \\
&\leq \sum_{k=0}^{n} \binom{n}{k}\varepsilon^{(n-k)p}(T^{\theta}N_{1})^{k}\frac{\Gamma(\theta)^{k}}{\Gamma(k\theta+1)}    \|u^1-u^0\|^{p}_{\mathcal{H}^{\phi,\gamma+2}_{p}(T)}  
\\
&\leq 2^{n}\varepsilon^{np}\max_{k}\left(  \frac{(\varepsilon^{-1} T^{\theta}N_{1} \Gamma(\theta))^{k}}{\Gamma(k\theta+1)}  \right) \|u^1-u^0\|^{p}_{\bH^{\phi,\gamma+2}_{p}(T)}.  
\end{aligned}
\end{equation*}
Now fix $\varepsilon<1/8$. Note that  the above maximum is finite and is independent of $n$. 
This and \eqref{eqn 4.21.22} imply 
$$
\sum_{n=1}^{\infty}  \|u^{n+1}-u^{n}\|^{p}_{\bH^{\phi,\gamma+2}_{p}(T)}   <\infty,
$$
and therefore $u^{n}$ is a Cauchy sequence in $\bH^{\phi,\gamma+2}_p(T)$.  Now let $u$ denote the its limit in $\bH^{\phi,\gamma+2}_p(T)$. Then, taking $n\to \infty$ from \eqref{eqn 4.21.5} and using the continuity of $f$ and $g$, we easily find that $u$ is a solution to \eqref{main equation} in the sense of distributions. 

Finally we prove  estimate \eqref{eq: a priori estimate non-div} for the solution $u$ obtained above.  Obviously,
$$
\|u\|_{\bH^{\phi,\gamma+2}_p(t)} \leq \|u-u^0\|_{\bH^{\phi,\gamma+2}_p(t)}+\|u^0\|_{\bH^{\phi,\gamma+2}_p(t)}.
$$ 
Also, by Step 1 and Assumption \ref{asm 11.15.13:16}, for any $\varepsilon>0$ and $t\leq T$,
\begin{eqnarray}
&& \|u-u^0\|^p_{\bH^{\phi,\gamma+2}_p(t)}  \nonumber \\
&&\leq C\|f(u)-f(0)\|^p_{\bH^{\phi,\gamma}_p(t)}+ C\|g(u)-g(0)\|^p_{\bH^{\phi,\gamma+c_0}_p(t,l_2)} \nonumber \\
 &&\leq C\varepsilon^p \|u\|^p_{\bH^{\phi,\gamma+2}_p(t)} +C(\varepsilon)\|u\|^p_{\bH^{\phi,\gamma}_p(t)} \nonumber\\
 &&\leq C\varepsilon^p \|u\|^p_{\bH^{\phi,\gamma+2}_p(t)} +C(\varepsilon)\|u-u^0\|^p_{\bH^{\phi,\gamma}_p(t)}+
 C(\varepsilon)\|u^0\|^p_{\bH^{\phi,\gamma+2}_p(t)}.   \label{eqn 4.24.5}
 \end{eqnarray}
Recall $(u-u^0)(0,\cdot)=0$.  Thus, by \eqref{eq:solution space estimate 1} and the continuity of $f$ and $g$,
 \begin{eqnarray*}
 \|u-u^0\|^p_{\bH^{\phi,\gamma}_p(t)}&\leq& C\int^t_0 (t-s)^{\theta-1} \Big(\|\phi(\Delta)u-\phi(\Delta)u^0+f(u)-f(0)\|^p_{\bH^{\phi,\gamma}_p(s)}\\
 &&\qquad \qquad + \|g(u)-g(0)\|^p_{\bH^{\phi,\gamma}_p(s,l_2)} \Big) ds\\
 &\leq& C\int^t_0 (t-s)^{\theta-1}\|u\|^p_{\bH^{\phi,\gamma+2}_p(s)}ds+ C \|u^0\|^p_{\bH^{\phi,\gamma+2}_p(T)}.
   \end{eqnarray*}
  Using this and \eqref{eqn 4.24.5}, and
 taking $\varepsilon>0$ sufficiently small, we get for $t\leq T$
\begin{eqnarray*}
\|u\|_{\bH^{\phi,\gamma+2}_p(t)} & \leq& C\|u^0\|_{\bH^{\phi,\gamma+2}_p(T)}+C \int^t_0(t-s)^{\theta-1} \|u\|^p_{\bH^{\phi,\gamma+2}_p(s)}  ds.
\end{eqnarray*}
Since the estimate of $u^0$ is obtained in Step 1, the desired estimate follows from the fractional Gronwall lemma. The theorem is proved.
 
\mysection{Proof of Theorem \ref{thm 10.27:15:35}}

\label{section 6}

Recall $H_p^{\phi,1}\subset H_p^{\delta_0}$ for $\delta_0\in(1/4,1]$,  and 
\begin{equation*}
\beta < \left(1-\frac{1}{4\delta_0} \right)\alpha +\frac{1}{2}, \quad 
d<2\delta_0\left(2-\frac{(2\beta-1)^+}{\alpha}\right)=:d_{0}.
\end{equation*}

\begin{rem}\label{rmk 03.12.16:42}
From the relation $H_p^{\phi,1}\subset H_p^{\delta_0}$, one can deduce $H_p^{\phi,\gamma}\subset H_p^{\delta_0\gamma}$ for $\gamma \geq 0$, and $H_p^{\delta_0\gamma}\subset H_p^{\phi,\gamma}$ for $\gamma \leq 0$.
Indeed, first, by \cite[Theorem 2.4.6]{farkas2001function}  we have $H_p^{\phi,\gamma}\subset H_p^{\delta_0\gamma}$ for $\gamma \in (0,1)$. Also, for $\gamma\in(1,2]$,
\begin{align*}
\|u\|_{H_p^{\delta_0\gamma}}&=\|(1-\Delta)^{\delta_0\gamma/2}u\|_{L_p}=\|(1-\Delta)^{(\delta_0\gamma-\delta_0)/2}(1-\Delta)^{\delta_0/2}u\|_{L_p}
\\
&\leq N \|(1-\phi(\Delta))^{(\gamma-1)/2}(1-\Delta)^{\delta_0/2}u\|_{L_p} 
\\
&= N\|(1-\Delta)^{\delta_0/2}(1-\phi(\Delta))^{(\gamma-1)/2}u\|_{L_p}
\\
&\leq N \|(1-\phi(\Delta))^{\gamma/2}u\|_{L_p}= \|u\|_{H_p^{\phi,\gamma}}.
\end{align*}
Repeating above argument, we get $H_p^{\phi,\gamma}\subset H_p^{\delta_0\gamma}$ for $\gamma \geq 0$. Second, by the duality theorem \cite[Theorem 2.2.10]{farkas2001function}, we have $H_p^{\delta_0\gamma}\subset H_p^{\phi,\gamma}$ for $\gamma \leq 0$. 
\end{rem}

Also recall that the equation
\begin{equation*}
\partial^{\alpha}_{t}u=\phi(\Delta)u+f(u)+\partial^{\beta-1}_{t}h(u) \dot{W}, \quad t>0; \quad u(0,\cdot)=0,
\end{equation*}
can be written as 
\begin{equation*}
\begin{aligned}
\partial^{\alpha}_{t}u=\phi(\Delta)u+f(u)+\partial^{\beta}_{t}\sum_{k=1}^{\infty}\int_{0}^{t}g^{k}(u)dw^{k}_{t},\quad t>0; \quad 
 u(0,\cdot)=0,
\end{aligned}
\end{equation*}
where $g^{k}(t,x,u)=h(t,x,u)\eta^{k}(x)$.  Therefore, to prove the theorem, it suffices to prove that $f$ and $g$ satisfy conditions in Assumption \ref{asm 11.15.13:16}. To check this we first prove some auxiliary results below. 

Note that by definition,   for any  $\gamma>0$ and smooth function $\varphi$,  we have
\begin{equation}\label{eqn 11.17.16:11}
\begin{aligned}
\cF\{(1-\Delta)^{-\gamma/2}\varphi\}(\xi)&=(1+|\xi|^{2})^{-\gamma/2}\cF\{\varphi\}(\xi)
\\
&=c\hat{\varphi}(\xi) \int_{0}^{\infty}t^{\gamma/2}e^{-t}e^{-t|\xi|^{2}}\frac{1}{t}dt,
\end{aligned}
\end{equation}
where $c=c(\gamma)>0$. Set
\begin{align*}
R_{\gamma,d}(x):=\int_{0}^{\infty}t^{\gamma/2}e^{-t} \bar{p}(t,x) \frac{1}{t}dt,
\end{align*}
where $\bar{p}(t,x)=(4\pi t)^{-d/2} e^{-|x|^2/(4t)}$.
Then,
\begin{equation*}
\int_{\bR^d} R_{\gamma,d}(x)dx =\int_{0}^{\infty}t^{\gamma/2-1}e^{-t}\int_{\bR^{d}}\bar{p}(t,x)dxdt = \int_{0}^{\infty} t^{\gamma/2-1}e^{-t} dt < \infty.
\end{equation*}
Therefore,  by Fubini's theorem,
$$
\cF \left\{  R_{\gamma,d} \right\}(\xi)=c(\gamma,d) \int_{0}^{\infty}t^{\gamma/2}e^{-t}e^{-t|\xi|^{2}}\frac{1}{t}dt.
$$
 Hence, from \eqref{eqn 11.17.16:11} for any $\gamma>0$ we have 
\begin{equation}\label{eqn 11.18.11:10}
(1-\Delta)^{-\gamma/2}\varphi=c(\gamma,d)\int_{\bR^{d}}R_{\gamma,d}(x-y)\varphi(y)dy.
\end{equation}

It is known that $R_{\gamma,d}$ decays exponentially at infinity and is comparable to $|x|^{-d+\gamma}$ near $x=0$ (see \cite{kry99analytic} or \cite{Krylov sobolev}). Thus, for $\gamma<d$ and $2r<\frac{d}{d-\gamma}$, we have $R_{\gamma,d}\in L_{2r}$.

\begin{lem}\label{lem 10.16.15:41}
Assume
$$
k_{0}\in \left(\frac{d}{2\delta_{0}},\frac{d}{\delta_{0}} \right),\quad 2\leq 2r\leq p, \quad 2r< \frac{d}{d-k_{0}\delta_{0}}.
$$
Let  $h=h(x,u)$ be a function of $(x,u)$ and $\xi=\xi(x)$ a function of $x$  such that
$$
|h(x,u)-h(x,v)|\leq \xi(x)|u-v|, \quad \forall\, x\in \bR^d,\, u,v\in \bR.
$$
If we set $g^{k}(u)=h(u)\eta^k$ and $g=(g^1,g^2,\cdots)$, then for $u,v\in L_{p}$, we have
$$
\|g(u)-g(v)\|_{H^{\phi,-k_{0}}_{p}(l_{2})} \leq C \|\xi\|_{L_{2s}} \|u-v\|_{L_{p}},
$$
where $s=r/(r-1)$, and $C=C(r)<\infty$. In particular, if $r=1$, and $\xi\in L_{\infty}$, then
$$
\|g(u)-g(v)\|_{H^{\phi,-k_{0}}_{p}(l_{2})} \leq C  \|u-v\|_{L_{p}}.
$$
\end{lem}
\begin{proof}
By Remark \ref{rmk 03.12.16:42}, \eqref{eqn 11.18.11:10} and Parseval's identity in Hilbert space $L_2(\bR^d)$, we have
\begin{equation*}
\begin{aligned}
&\|g(u)-g(v)\|_{H^{\phi,-k_{0}}_{p}(l_{2})} \leq \|g(u) - g(v) \|_{H^{-k_{0}\delta_{0}}_{p}(l_{2})}
\\
&=C\left\| \left( \int_{\bR^{d}}|R_{k_{0}\delta_{0},d}(\cdot-y)|^{2} |h(y,u(y))-h(y,v(y))|^{2} dy \right)^{1/2} \right\|_{L_{p}}
\\
& \leq C \left\| \left( \int_{\bR^{d}}|R_{k_{0}\delta_{0},d}(\cdot-y)|^{2} |\xi(y)|^{2}|u(y)-v(y)|^{2} dy \right)^{1/2} \right\|_{L_{p}}.
\end{aligned}
\end{equation*}
Note that $R_{k_0\delta_0}\in L_{2r}$ since $k_0\delta_0<d$ and $2r<d/(d-k_0\delta_0)$.
Hence, by H\"older's inequality and Minkowski's inequality, we have
\begin{equation*}
\begin{aligned}
&\|g(u)-g(v)\|_{H^{\phi,-k_{0}}_{p}(l_{2})}
\\
& \leq C \left\| \left( \int_{\bR^{d}}|R_{k_{0}\delta_{0},d}(\cdot-y)|^{2} |\xi(y)|^{2}|u(y)-v(y)|^{2} dy \right)^{1/2} \right\|_{L_{p}}
\\
& \leq C \|\xi\|_{L_{2s}} \left\|  \left( \int_{\bR^{d}}|R_{k_{0}\delta_{0},d}(\cdot-y)|^{2r} |u(y)-v(y)|^{2r} dy \right)^{1/2r} \right\|_{L_{p}}.
\\
&\leq C \|\xi\|_{L_{2s}} \|R_{k_{0}\delta_{0},d}\|_{L_{2r}}\|u-v\|_{L_{p}}\leq C \|\xi\|_{L_{2s}} \|u-v\|_{L_{p}}.
\end{aligned}
\end{equation*}
The lemma is proved.
\end{proof}

\textbf{Proof of Theorem \ref{thm 10.27:15:35}}.

As mentioned above, it suffices to prove that the conditions in Theorem \ref{thm:main results} hold with $\gamma=-k_{0}-c_{0}$.
By \cite[Theorem 2.4.6]{farkas2001function}, if  $\nu_1<\nu_2<\nu_3$, then for any $\varepsilon>0$
\begin{equation}
   \label{eqn 4.23.1}
\|u\|_{H_p^{\phi,\nu_2}}\leq \varepsilon \|u\|_{H_p^{\phi,\nu_3}}+N(\varepsilon) \|u\|_{H_p^{\phi,\nu_1}},
\end{equation}
where $N(\varepsilon)$ depends on $\varepsilon, \nu_1,\nu_2$ and $\nu_3$. Due to \eqref{eqn 10.17.13:27} one can choose $\kappa$ small enough such that $\gamma+2>0$. 
Since $\gamma<0$ and $\gamma+2>0$, by the assumption of $f$ and \eqref{eqn 4.23.1},  we have for any $\varepsilon>0$ 
  \begin{align*}
\|f(u)-f(v)\|_{H_p^{\phi,\gamma}} &\leq C\|f(u)-f(v)\|_{L_p} \leq C\|u-v\|_{L_p}
\\
&\leq \varepsilon \|u-v\|_{H_p^{\phi,\gamma+2}} + N(\varepsilon) \|u-v\|_{H_p^{\phi,\gamma}}.
  \end{align*}
Therefore it remains to check the conditions for $g(u)$. Let $r=s/(s-1)$. Then $2r<d/(d-k_{0}\delta_{0})$ due to the assumption on $s$. Since $\gamma+c_{0}=-k_{0}$, by Lemma \ref{lem 10.16.15:41} and \eqref{eqn 4.23.1},  for any $\varepsilon>0$, we have
\begin{eqnarray*}
\|g(u)-g(v)\|_{H^{\phi,\gamma+c_{0}}_{p}(l_{2})} &\leq& C\|\xi\|_{L_{2s}}\|u-v\|_{L_{p}}\\
& \leq& \varepsilon \|u-v\|_{H_p^{\phi,\gamma+2}} + N(\varepsilon) \|u-v\|_{H_p^{\phi,\gamma}}.
\end{eqnarray*}
Hence, the condition for $g$ is also fulfilled.  Furthermore, by inspecting the proof of Lemma \ref{lem 10.16.15:41}, one can easily check 
$$\|g(0)\|_{\mathbb{H}^{-k_{0}}_{p}(T,l_{2})}\leq C \|h(0)\|_{\mathbb{L}_{p}(T)}.
$$
Therefore, we finish the proof of the theorem.

\appendix

\mysection{Auxiliary results} \label{appendix}

In this section we obtain some sharp upper bounds of  space-time fractional derivatives of the  fundamental solution $q(t,x)$ related to the equation
\begin{equation*} 
\partial_t^\alpha u = \phi(\Delta)u ,\quad t>0; \quad u(0,\cdot)=u_0.
\end{equation*}

First we record some elementary facts on Bernstein functions.

\begin{lem} \label{gammaversionj}
Let $\phi$ be a Bernstein function satisfying Assumption \ref{ass bernstein}.

(i) There exists a constant $c=c(\gamma,\kappa_0,\delta_0)$ such that for any $\lambda>0$,
\begin{equation} \label{phiint}
\int_{\lambda^{-1}}^\infty r^{-1}\phi(r^{-2}) dr \leq c \phi(\lambda^2). 
\end{equation}

(ii) For  any $\gamma\in(0,1)$,  the function $\phi^\gamma:=\left(\phi(\cdot)\right)^\gamma$ is also a Bernstein function with no drift, and it satisfies Assumption \ref{ass bernstein} with $\gamma \delta_0$ and $\kappa_0^{\gamma}$, in place of $\delta_0$ and $\kappa_0$ respectively.

(iii) Let $\mu_\gamma$ be the L\'evy measure of $\phi^\gamma$ (i.e. $\phi^{\gamma}(\lambda)=\int_{(0,\infty)} (1-e^{-\lambda t})\mu_{\gamma}(dt)$), and set
\begin{equation}\label{defofjgamma}
j_{\gamma,d}(r):=\int_0^\infty (4\pi t)^{-d/2} \exp(-r^2/4t)\mu_\gamma (dt), \qquad r>0.
\end{equation}
Then 
\begin{equation} \label{jgamma ineq}
j_{\gamma,d} (r) \leq c(d) \,r^{-d}\phi(r^{-2})^\gamma, \qquad \forall r>0,
\end{equation}
and for any  $f\in C_b^2(\bR^d)$ and $r>0$,  it holds that 
\begin{align} \label{probrepresentphigamma}
\phi(\Delta)^\gamma f(\cdot)(x) =& \int_{\bR^d} \left(f(x+y)-f(x)-\nabla f(x)\cdot y \mathbf{1}_{|y|\leq r} \right)j_{\gamma,d} (|y|) dy
\end{align}
\end{lem}

\begin{proof}
(i) By \eqref{phiratio} this and the change of variables,
\begin{equation*}
\begin{aligned}
\int_{\lambda^{-1}}^{\infty}r^{-1}\phi(r^{-2})dr  &  =\int_{1}^{\infty}r^{-1}\phi(\lambda^{2}r^{-2})dr=\int_{1}^{\infty}r^{-1}\phi(\lambda^{2}r^{-2})\frac{\phi(\lambda^{2})}{\phi(\lambda^{2})}dr
\\
&\leq c \int_{1}^{\infty}r^{-1-2\delta_{0}}dr\phi(\lambda^{2})=c \phi(\lambda^{2}).
\end{aligned}
\end{equation*}

(ii) $\phi^\gamma$ is a Bernstein function due to \cite[Corollary 3.8 (iii)]{schilling2012bernstein}, and \eqref{e:H} easily yields
$$
\kappa_0^\gamma \left(\frac{R}{r}\right)^{\gamma\delta_0}\leq\frac{\phi(R)^\gamma}{\phi(r)^\gamma}, \qquad 0<r<R<\infty.
$$
If we denote drift of $\phi^{\gamma}$ by $b_{\gamma}$ (see \eqref{eqn 07.17.16:36}), it follows that

$$
\lim_{\lambda\to\infty} \frac{\phi(\lambda)}{\lambda}=b, \quad \lim_{\lambda\to\infty} \frac{\phi(\lambda)^{\gamma}}{\lambda}=b_{\gamma}.
$$
Hence, we have
$$
b_\gamma=\lim_{\lambda\to\infty} \frac{\phi(\lambda)^\gamma}{\lambda}=\lim_{\lambda\to\infty} \left(\frac{\phi(\lambda)}{\lambda}\right)^\gamma \lambda^{\gamma-1}=0.
$$

(iii)  \eqref{jgamma ineq} follows from \cite[Lemma 3.3]{kim2013parabolic} (recall that $\phi^{\gamma}$ is a Bernstein function with no drift), and the second assertion is a consequence of \eqref{fourier200408}.
\end{proof}

Recall that $p(t,x)$ is the transition density of the subordinate Brownian motion $X_t$ with characteristic exponent $\phi(|\xi|^2)$. Also, for any $t>0$, and $x\in \bR^d$,
\begin{eqnarray}
\label{fourierofp}
p(t,x)&=&\frac{1}{(2\pi)^d} \int_{\bR^d} e^{i \xi \cdot x} e^{-t\phi(|\xi|^2)} \,d\xi
\\
&=&\int_{(0,\infty)} (4\pi s)^{-d/2} \exp \left(-\frac{|x|^2}{4s}\right) \eta_t (ds)  \nonumber
\end{eqnarray}
where $\eta_t(ds)$ is the distribution of $S_t$ (see \cite[Section 5.3.1]{bogdan2009potential}). Thus $X_t$  is rotationally invariant.

\begin{lem} \label{pestimate}
(i) There exists a constant $C=C(d,\delta_0,\kappa_0)$ such that for $(t,x)\in(0,\infty)\times \bR^d$,
\begin{equation*}
p(t,x)\leq C\left(\left(\phi^{-1}(t^{-1})\right)^{d/2}\wedge t\frac{\phi(|x|^{-2})}{|x|^d}\right).
\end{equation*}

(ii) For any $m\in \bN$, there exists a constant $C=C(d,\delta_0,\kappa_0, m)$ so that for any $(t,x)\in (0,\infty)\times \bR^d$,
\begin{align*}
& |D^m_x p(t,x)| \leq C \sum_{m-2n\geq0,n\in\bN_0} |x|^{m-2n}\left((\phi^{-1}(t^{-1}))^{d/2+m-n}\wedge t\frac{\phi(|x|^{-2})}{|x|^{d+2(m-n)}}\right).
\end{align*}

\end{lem}
\begin{proof}
See  \cite[Lemma 3.4, Lemma 3.6]{kim2020nonlocal}.
\end{proof}

The following lemma is an extension of \cite[Lemma  4.2]{kim2013parabolic}. The main difference is that our estimate  holds for all $t>0$.  Such result is needed for us to prove estimates of solutions to SPDEs (see e.g. \eqref{a priori}).

\begin{lem} \label{pfracderivativeestimate}
Let $\gamma\in(0,1)$ and $m\in \bN_0$. Then for any $(t,x)\in(0,\infty)\times \bR^d$, 
\begin{align*}
&|\phi(\Delta)^\gamma D^m_x p(t,\cdot)(x)| \nonumber
\\
&\leq C(d,\delta_0, \kappa_0, m, \gamma) \left(t^{-\gamma}(\phi^{-1}(t^{-1}))^{(d+m)/2}\wedge \frac{\phi(|x|^{-2})^\gamma}{|x|^{d+m}}\right).
\end{align*}
\end{lem}

\begin{proof}
Note first that for any given $a>0$, 
$$
s^{a}e^{-s}\leq c(a)e^{-s/2}, \quad \forall s>0.
$$
Also note that by \eqref{e:H}, if $a^2 \geq \phi^{-1}(t^{-1})$, then
$$
\kappa_0 \left(\frac{a^2}{\phi^{-1}(t^{-1})}\right)^{\delta_0} \leq \frac{\phi(a^2)}{t^{-1}}=t\phi(a^2).
$$
Therefore,
by  \eqref{fourierofp},
\begin{align*}
&\left|\phi(\Delta)^\gamma D^{m}_{x} p(t,x)\right| =\left| \cF^{-1} \left[\cF (\phi(\Delta)^\gamma D^{m}_{x} p(t,\cdot)) (\xi) \right]  (x)\right|
\\
& \qquad \leq C\int_{\bR^d} \left| t^{-\gamma}\left(t\phi(|\xi|^2)\right)^\gamma  e^{-t\phi(|\xi|^2)}\right| |\xi|^{m} d\xi
\\
&\qquad \leq Ct^{-\gamma} \left(\int_{|\xi|^2>\phi^{-1}(t^{-1})} |\xi|^{m} e^{-t\frac{\phi(|\xi|^2)}{2}}d\xi + \int_{|\xi|^2\leq\phi^{-1}(t^{-1})} |\xi|^{m}   d\xi \right)
\\
&\qquad \leq Ct^{-\gamma}\left( \int_{|\xi|^2>\phi^{-1}(t^{-1})}  |\xi|^{m} e^{-\frac{\kappa_0}{2} \left(\frac{|\xi|^2}{\phi^{-1}(t^{-1})}\right)^{\delta_0}}d\xi +\left(\phi^{-1}(t^{-1})\right)^{\frac{d+m}{2}} \right)
\\
& \qquad \leq C t^{-\gamma}\left(\phi^{-1}(t^{-1})\right)^{\frac{d+m}{2}} \left( \int_{|\xi|^2>1} e^{-\frac{\kappa_0}{4}|\xi|^{2\delta_0}} d\xi +1 \right)
\\
&\qquad \leq C t^{-\gamma}\left(\phi^{-1}(t^{-1})\right)^{\frac{d+m}{2}}.
\end{align*}
Hence, to finish the proof, we may assume $t^{-\gamma}(\phi^{-1}(t^{-1}))^{\frac{d+m}{2}} \geq  \frac{\phi(|x|^{-2})^\gamma}{|x|^{d+m}}$ (equivalently, $t \phi(|x|^{-2})\leq 1$) and prove
$$
|\phi(\Delta)^\gamma D^m_x p(t,\cdot)(x)|
\leq C  \frac{\phi(|x|^{-2})^\gamma}{|x|^{d+m}}.
$$
By \eqref{probrepresentphigamma} with $r=|x|/2$,
\begin{align*}
&\left|\phi(\Delta)^\gamma D^{m}_{x} p(t,\cdot)(x)\right|
\\
&= \left|\int_{\bR^d} \left(D^{m}_{x}p(t,x+y)-D^{m}_{x}p(t,x)-\nabla D^{m}_{x}p(t,x)\cdot y \mathbf{1}_{|y|\leq\frac{|x|}{2}}\right)j_{\gamma,d} (|y|)dy \right|
\\
&\leq |D^{m}_{x}p(t,x)|\int_{|y|>|x|/2} j_{\gamma,d} (|y|)dy \\
&\quad+ \left| \int_{|y|>|x|/2} D^{m}_{x}p(t,x+y)j_{\gamma,d} (|y|)dy\right|
\\
& \quad+ \int_{|x|/2>|y|} \int_0^1 \left| D^{m+1}_{x}p(t,x+sy)- D^{m+1}_{x}p(t,x)\right| |y|j_{\gamma,d}(|y|) dsdy
\\
&=: |D^{m}_{x}p(t,x)|\times I + II + III.
\end{align*}
By \eqref{jgamma ineq}, \eqref{phiint} with $\phi^{\gamma}$, and \eqref{phiratio}, we have
\begin{align*}
I \leq C \int_{r>|x|/2} r^{-1}\phi(r^{-2})^\gamma dr \leq C \phi(4|x|^{-2})^{\gamma} \leq C \phi(|x|^{-2})^{\gamma}.
\end{align*}
This together with Lemma \ref{pestimate} (ii) yields (recall we assume $t\phi(|x|^{-2})\leq 1$)
\begin{align*}
|D^{m}_{x}p(t,x)|\times I \leq C t\frac{\phi(|x|^{-2})^{1+\gamma}}{|x|^{d+m}}\leq C  \frac{\phi(|x|^{-2})^{\gamma}}{|x|^{d+m}} .
\end{align*}
For $III$, by the fundamental theorem of calculus and Lemma \ref{pestimate} (ii),
\begin{align*}
III \leq&C(d)\int_{|x|/2>|y|} \int_0^1 \int_0^1 \left|D^{m+2}_{x}p(t,x+usy)\right| |y|^2j_{\gamma,d}(|y|) dudsdy
\\
\leq& C\int_{|x|/2>|y|} \int_0^1\int_0^1 t\frac{\phi(|x+usy|^{-2})}{|x+usy|^{d+m+2}} |y|^2j_{\gamma,d}(|y|) dudsdy
\\
\leq& C t\frac{\phi(|x|^{-2})}{|x|^{d+m+2}}\int_{|x|/2>|y|} |y|^2j_{\gamma,d}(|y|) dy.
\end{align*}
For the last inequality above, we used $|x+usy|\geq |x|/2$. By \eqref{jgamma ineq}, and \eqref{phiratio} with $r=|x|^{-2}$ and $R=\rho^{-2}$, 
\begin{align*}
\int_{|x|/2>|y|} |y|^2j_{\gamma,d}(|y|)dy \leq& C \int_0^{|x|} \rho\phi(\rho^{-2})^\gamma d\rho
\\
\leq& C |x|^{2\gamma}\phi(|x|^{-2})^\gamma \int_0^{|x|} \rho^{1-2\gamma}  d\rho
\\
\leq& C |x|^2 \phi(|x|^{-2})^\gamma.
\end{align*}
Therefore, it follows that for $t \phi(|x|^{-2})\leq 1$,
$$
III\leq  C t\frac{\phi(|x|^{-2})^{1+\gamma}}{|x|^{d+m}} \leq C \frac{\phi(|x|^{-2})^{\gamma}}{|x|^{d+m}}.
$$
Now we estimate $II$. By using the integration by parts $m$-times, we have
\begin{align*}
II &\leq \sum_{k=0}^{m-1} \int_{|y|=\frac{|x|}{2}} \left| \left(\frac{d^{k}}{d\rho^{k}} j_{\gamma,d}\right)(|y|) D^{m-1-k}_{x}p(t,x+y) \right| dS
\\
&\quad+\int_{|y|>\frac{|x|}{2}} \left| \left(\frac{d^{m}}{d\rho^{m}}j_{\gamma,d}\right)(|y|)p(t,x+y) \right| dy.
\end{align*}
Differentiating $j_{\gamma,d}(\rho)$, and then using   \eqref{defofjgamma} and \eqref{jgamma ineq}, for $k\in\bN_{0}$ we get
\begin{align*}
\left| \frac{d^{k}}{d\rho^{k}} j_{\gamma,d}(\rho)  \right| 
\leq C \sum_{k-2l\geq 0,l\in\bN_{0}} \rho^{k-2l}|j_{\gamma,d+2(k-l)}(\rho)|
\leq C \rho^{-d-k}\phi(\rho^{-2})^{\gamma}.
\end{align*}
This and Lemma \ref{pestimate} (ii) yield that
\begin{align*}
II
 &\leq C\sum_{k=0}^{m-1} \Big( |x|^{d-1}|x|^{-d-k}  |x|^{-d-m+1+k}  t \phi(|x|^{-2})^{\gamma+1} \Big)
 + C \frac{\phi(|x|^{-2})^{\gamma}}{|x|^{d+m}}\\
  &\leq C \frac{\phi(|x|^{-2})^{\gamma}}{|x|^{d+m}}.
    \end{align*}
for $t\phi(|x|^{-2})\leq 1$. Hence, the lemma is proved.
\end{proof}

Below we provide the proof of Lemma \ref{prop:kernel esti. of q}.
The proof  is mainly based on \cite[Lemma 3.7, Lemma 3.8]{kim2020nonlocal}.

{\textbf{Proof of Lemma \ref{prop:kernel esti. of q}} }
\begin{proof}
(i) See \cite[Lemma 3.7 (iii)]{kim2020nonlocal}.

(ii) See  \cite[Lemma 3.8]{kim2020nonlocal} for \eqref{eqn 09.03.19:26} with arbitrary $\beta$ and for   \eqref{eqn 09.03.19:26-2} when $\beta\notin \bN$.   Hence, we only prove \eqref{eqn 09.03.19:26-2} when $\beta\in\bN$.  Let $\beta\in\bN$, then  by \cite[Lemma 3.8]{kim2020nonlocal}, we have
\begin{align*}
|q_{\alpha,\beta} (t,x)|&\leq C \int_{(\phi(|x|^{-2}))^{-1}}^{2t^{\alpha}} (\phi^{-1}(r^{-1}))^{(d+m)/2} r t^{-\alpha-\beta} dr
\\
&\leq C \int_{(\phi(|x|^{-2}))^{-1}}^{2t^{\alpha}} (\phi^{-1}(r^{-1}))^{(d+m)/2} t^{-\beta} dr.
\end{align*}
For the last inequality, we used $rt^{-\alpha}\leq 2$ whenever $r\leq 2t^{\alpha}$.

(iii)  We follow the proof of \cite[Lemma 3.8]{kim2020nonlocal}.  By \cite[Lemma 3.7]{kim2020nonlocal}, there exist  constants $c,C>0$ depending only on $\alpha,\beta$ such that 
\begin{equation} \label{philarge}
|\varphi_{\alpha,\beta}(t,r)|\leq C t^{-\beta}e^{-c(rt^{-\alpha})^{1/(1-\alpha)}}
\end{equation}
for $rt^{-\alpha}\geq 1$, and 
\begin{eqnarray}\label{betainteger}
|\varphi_{\alpha,\beta}(t,r)|\leq \left\{ \begin{array}{ll}  C rt^{-\alpha-\beta}~&:\, \beta\in \bN \\ C t^{-\beta}~&:\, \beta \notin \bN\end{array} \right.
\end{eqnarray}
for $rt^{-\alpha}\leq 1$. Therefore, we have
\begin{equation} \label{eqn 200824 1516}
\int_0^\infty |\varphi_{\alpha,\beta}(t,r)|dr<\infty.
\end{equation}
Let $x\in \bR^d\setminus\{0\}$. Then for any $r>0$ and  $y\neq 0$ sufficiently close to $x$, we have
$$
 |\phi(\Delta)^\gamma D^{\sigma}p(r,y)| \leq C(\phi,x,d,m,\gamma) , \quad |\sigma|\leq m
$$
due to Lemma \ref{pfracderivativeestimate}. Using \eqref{eqn 200824 1516} and the dominated convergence theorem, we get
\begin{equation*}
\label{eqn 7.20.6}
D^m_x q^\gamma_{\alpha,\beta} (t,x)=\int_0^\infty \phi(\Delta)^\gamma  D^m_x p(r,x)\varphi_{\alpha,\beta}(t,r) dr.
\end{equation*}
Hence, by \eqref{philarge} and \eqref{betainteger} (also recall $rt^{-\alpha}\leq 2$ whenever $r\leq 2t^{\alpha}$),
\begin{align}
|D^m_x q^\gamma_{\alpha,\beta} (t,x)| &\leq C \int_0^{t^\alpha} |\phi(\Delta)^\gamma D^m_x p(r,x)|t^{-\beta} dr   \nonumber
\\
&\quad + C \int_{t^\alpha}^\infty |\phi(\Delta)^\gamma D^m_x p(r,x)| t^{-\beta}e^{-c(rt^{-\alpha})^{1/(1-\alpha)}} dr  \nonumber
\\
& =: I+II.    \label{eqn 4.15.1}
\end{align}
By Lemma \ref{pfracderivativeestimate},

\begin{align*}
I \leq C \int_0^{t^\alpha} t^{-\beta} \frac{\phi(|x|^{-2})^{\gamma}}{|x|^{d+m}}dr
\leq C t^{\alpha-\beta} \frac{\phi(|x|^{-2})^{\gamma}}{|x|^{d+m}}.
\end{align*}
Also, by the change of variables $rt^{-\alpha}\to r$,
\begin{align*}
II&\leq C t^{-\beta} \int_{t^\alpha}^\infty  \frac{\phi(|x|^{-2})^{\gamma}}{|x|^{d+m}} e^{-c(rt^{-\alpha})^{1/(1-\alpha)}} dr
\\
&\leq C t^{\alpha-\beta}\frac{\phi(|x|^{-2})^{\gamma}}{|x|^{d+m}} \int_1^\infty  e^{-cr^{1/(1-\alpha)}} dr
\\
&\leq C t^{\alpha-\beta}\frac{\phi(|x|^{-2})^{\gamma}}{|x|^{d+m}}.
\end{align*}
Hence, \eqref{qgamma whole est} is proved. 

Next we prove \eqref{qgamma partial est}.
  Assume $t^\alpha \phi(|x|^{-2})\geq1$.  Again we consider $I$ and $II$ defined in \eqref{eqn 4.15.1}. For $I$ we have

\begin{align*}
I&=t^{-\beta} \int_0^{(\phi(|x|^{-2}))^{-1}}|\phi(\Delta)^\gamma D^m_x p(r,x)|dr 
\\
&\quad\quad +t^{-\beta}  \int_{(\phi(|x|^{-2}))^{-1}}^{t^\alpha}|\phi(\Delta)^\gamma D^m_x p(r,x)|t^{-\beta} dr 
\\
&=:I_1 + I_2.
\end{align*}
By Lemma \ref{pfracderivativeestimate} (recall that $t^{\alpha}\phi(|x|^{-2})\geq 1$),
\begin{align*}
I_1&\leq C t^{-\beta} \int_0^{(\phi(|x|^{-2}))^{-1}} \frac{\phi(|x|^{-2})^{\gamma}}{|x|^{d+m}}dr\\
&= C t^{-\beta} \frac{\phi(|x|^{-2})^{\gamma-1}}{|x|^{d+m}}
= Ct^{-\beta} \int_{(\phi(|x|^{-2}))^{-1}}^{2(\phi(|x|^{-2}))^{-1}}  \frac{\phi(|x|^{-2})^{\gamma}}{|x|^{d+m}} dr \nonumber
\end{align*}
Note that if $r \leq 2(\phi(|x|^{-2}))^{-1}$, then by \eqref{phiratio}
\begin{equation} \label{21.04.28.14.55}
|x|^{-2} \leq \phi^{-1}(2r^{-1}) \leq  \left(\frac{2}{k_0} \right)^{1/\delta_{0}} \phi^{-1}(r^{-1}).
\end{equation}
Thus, using $r\leq 2(\phi(|x|^{-2}))^{-1}$, we get
\begin{align*}
I_1
&\leq C t^{-\beta}\int_{(\phi(|x|^{-2}))^{-1}}^{2(\phi(|x|^{-2}))^{-1}} (\phi^{-1}(r^{-1}))^{(d+m)/2} r^{-\gamma} dr \nonumber
\\
&\leq Ct^{-\beta}  \int_{(\phi(|x|^{-2}))^{-1}}^{2t^{\alpha}} (\phi^{-1}(r^{-1}))^{(d+m)/2} r^{-\gamma} dr.   \label{eqn 09.05.17:32}
\end{align*}
  We also get, by  
 Lemma \ref{pfracderivativeestimate}, 
\begin{align*}
I_2 
\leq C  t^{-\beta}  \int_{(\phi(|x|^{-2}))^{-1}}^{2t^{\alpha}} (\phi^{-1}(r^{-1}))^{(d+m)/2} r^{-\gamma}dr.
\end{align*}
Thus $I$ is handled. Next we estimate $II$. By \eqref{phiratio}, we find that
$$
\phi^{-1}(r^{-1})\leq t^\alpha r^{-1} \phi^{-1}(t^{-\alpha}), \quad t^\alpha \leq r.
$$
Therefore, by Lemma \ref{pfracderivativeestimate} and the change of variables $rt^{-\alpha}\to r$,
\begin{eqnarray}
II &\leq& Ct^{-\beta} \int_{t^\alpha}^\infty r^{-\gamma}(\phi^{-1}(r^{-1}))^{(d+m)/2} e^{-c(rt^{-\alpha})^{1/(1-\alpha)}} dr   \nonumber
\\
&\leq& Ct^{-\beta} \int_{t^\alpha}^\infty r^{-\gamma}(t^\alpha r^{-1} \phi^{-1}(t^{-\alpha}))^{(d+m)/2} e^{-c(rt^{-\alpha})^{1/(1-\alpha)}} dr  \nonumber
\\
&=& C t^{(1-\gamma)\alpha-\beta} (\phi^{-1}(t^{-\alpha}))^{(d+m)/2} \int_1^\infty r^{-\gamma-(d+m)/2}e^{-cr^{1/(1-\alpha)}} dr  \nonumber
\\
&\leq& C t^{(1-\gamma)\alpha-\beta} (\phi^{-1}(t^{-\alpha}))^{(d+m)/2}. \label{eqn 4.15.7}
\end{eqnarray}
As \eqref{21.04.28.14.55},  if $r \leq 2t^{\alpha}$, then by \eqref{phiratio}
\begin{equation*}
\phi^{-1}(t^{-\alpha})\leq \phi^{-1}(2r^{-1}) \leq  \left(\frac{2}{\kappa_0} \right)^{1/\delta_{0}} \phi^{-1}(r^{-1}).
\end{equation*}
Therefore,
\begin{align*}
t^{(1-\gamma)\alpha-\beta} (\phi^{-1}(t^{-\alpha}))^{(d+m)/2} &\leq C \int_{t^\alpha}^{2t^{\alpha}} (\phi^{-1}(r^{-1}))^{(d+m)/2} r^{-\gamma}t^{-\beta} dr
\\
&\leq C \int_{(\phi(|x|^{-2}))^{-1}}^{2t^{\alpha}} (\phi^{-1}(r^{-1}))^{(d+m)/2} r^{-\gamma}t^{-\beta} dr.
\end{align*}
This and \eqref{eqn 4.15.7} take care of $II$, and consequently  \eqref{qgamma partial est} is proved.

(iv)  See  \cite[Corollary 3.9]{kim2020nonlocal} for \eqref{int of q}.  We prove  \eqref{int of q^gamma}. 
By \eqref{qgamma whole est}, \eqref{qgamma partial est}, Fubini's theorem, and  \eqref{phiint},
\begin{eqnarray*}
\int_{\bR^d} |q^\gamma_{\alpha,\beta}(t,x)|dx &=&\int_{|x|\geq (\phi^{-1}(t^{-\alpha}))^{-\frac{1}{2}}} |q^\gamma_{\alpha,\beta}(t,x)|dx
\\
&&+\int_{|x|< (\phi^{-1}(t^{-\alpha}))^{-\frac{1}{2}}} |q^\gamma_{\alpha,\beta}(t,x)|dx
\\
&\leq& C \int_{|x|\geq (\phi^{-1}(t^{-\alpha}))^{-\frac{1}{2}}} t^{\alpha-\beta}\frac{\phi(|x|^{-2})^\gamma}{|x|^d} dx
\\
&&+ C \int_{|x|< (\phi^{-1}(t^{-\alpha}))^{-\frac{1}{2}}} \int_{(\phi(|x|^{-2}))^{-1}}^{2t^{\alpha}} (\phi^{-1}(r^{-1}))^{d/2} r^{-\gamma}t^{-\beta} dr dx
\\
&\leq& C \int_{r\geq \left( \phi^{-1}(t^{-\alpha}) \right)^{-\frac{1}{2}}} t^{\alpha-\beta} \frac{\phi(r^{-2})^\gamma}{r}dr
\\
&&+ C \int_{0}^{2t^\alpha} \int_{(\phi(|x|^{-2}))^{-1}\leq r} (\phi^{-1}(r^{-1}))^{d/2} r^{-\gamma}t^{-\beta} dx dr
\\
&\leq& C t^{(1-\gamma)\alpha-\beta} + C \int_0^{2t^\alpha} r^{-\gamma}t^{-\beta} dr
\leq C t^{(1-\gamma)\alpha-\beta}.
\end{eqnarray*}

(v) By  (3.24) in \cite{kim2020nonlocal} (or see (34) of \cite{gorenflo2007analytical}),
\begin{align*}
\int_0^\infty e^{-sr}\varphi_{\alpha,\beta}(t,r)dr = t^{\alpha-\beta} E_{\alpha,1-\beta+\alpha}(-st^\alpha).
\end{align*}
Hence, by Fubini's theorem and \eqref{fourierofp} for $\gamma\in (0,1)$,
\begin{align*} 
\cF_d(q^\gamma_{\alpha,\beta})(t,\xi) &= \int_0^\infty \varphi_{\alpha,\beta}(t,r)\left[\int_{\bR^d}e^{-ix\cdot \xi} \phi(\Delta)^\gamma p(r,x)dx \right] dr \nonumber
\\
&=-\int_0^\infty \varphi_{\alpha,\beta}(t,r) \phi(|\xi|^2)^\gamma e^{-r\phi(|\xi|^2)} dr \nonumber
\\
&=-t^{\alpha-\beta} \phi(|\xi|)^\gamma E_{\alpha,1-\beta+\alpha}(-t^\alpha\phi(|\xi|^2))).
\end{align*}
Similarly, we get
\begin{align*} 
\cF_d(q_{\alpha,\beta})(t,\xi) &= \int_0^\infty \varphi_{\alpha,\beta}(t,r)\left[\int_{\bR^d}e^{-ix\cdot \xi} p(r,x)dx \right] dr \nonumber
\\
&=t^{\alpha-\beta} E_{\alpha,1-\beta+\alpha}(-t^\alpha\phi(|\xi|^2)).
\end{align*}
Thus (v) is also proved.
\end{proof}

\end{document}